\documentclass[aihp-draft]{imsart}

\RequirePackage{amsthm,amsmath,amsfonts,amssymb}
\RequirePackage[numbers]{natbib}
\RequirePackage[colorlinks,citecolor=blue,urlcolor=blue]{hyperref}
\RequirePackage{graphicx}

\usepackage{enumitem}
\usepackage{bbm}
\usepackage{xcolor}
\usepackage[ruled,vlined,linesnumbered]{algorithm2e}
\usepackage{cleveref}
\usepackage{float}
\usepackage{mathrsfs}
\usepackage{centernot}
\usepackage{booktabs}
\usepackage{aligned-overset}
\usepackage{verbatim}
\usepackage{subfigure}
\usepackage{float}

\startlocaldefs
\numberwithin{equation}{section}
\numberwithin{table}{section}
\numberwithin{figure}{section}

\theoremstyle{plain}
        \newtheorem{theorem}{Theorem}[section]
        \newtheorem*{theorem*}{Theorem}
        \newtheorem*{conj*}{Conjecture}
        \newtheorem{lemma}[theorem]{Lemma}
        \newtheorem{prop}[theorem]{Proposition}
        \newtheorem{cor}[theorem]{Corollary}
\theoremstyle{definition}
        \newtheorem{definition}[theorem]{Definition}
        \newtheorem{rem}[theorem]{Remark}

        \newtheorem*{remark}{Remark}

\newcommand{\T}{\mathbb{T}} 
\newcommand{\Z}{\mathbb{Z}}
\renewcommand{\P}{\mathsf{P}}
\newcommand{\E}{\mathsf{E}}
\newcommand{\N}{\mathbb{N}}
\newcommand{\WP}{\widetilde{\mathsf{P}}}

\newcommand{\I}{\mathcal{I}}
\newcommand{\V}{\mathcal{V}}
\renewcommand{\L}{\mathcal{L}}
\newcommand{\CF}{\mathcal{F}}
\newcommand{\CE}{\mathcal{E}}
\newcommand{\CA}{\mathcal{A}}
\newcommand{\G}{\mathcal{G}}
\newcommand{\R}{\mathcal{R}}
\newcommand{\C}{\mathcal{C}}
\newcommand{\CU}{\mathcal{U}}
\newcommand{\sfC}{\mathsf{C}}
\newcommand{\frC}{\mathfrak{C}}
\newcommand{\frM}{\mathfrak{M}}
\newcommand{\bbP}{\mathbb{P}}

\newcommand{\bs}{\boldsymbol}
\newcommand{\x}{\boldsymbol{x}}
\newcommand{\y}{\boldsymbol{y}}
\newcommand{\z}{\boldsymbol{z}}
\newcommand{\F}{\boldsymbol{F}}
\newcommand{\BQ}{\boldsymbol{Q}}
\newcommand{\HH}{\boldsymbol{H}}
\newcommand{\K}{\boldsymbol{K}}

\newcommand{\err}{\mathsf{err}}

\newcommand{\g}{\alpha_N}

\newcommand{\ai}{\beta_N}
\newcommand{\wt}{\widetilde}
\newcommand{\ov}{\overline}
\newcommand{\wh}{\widehat}

\newcommand{\Dir}{\mathsf{Dir}}

\newcommand{\mdist}{\mathsf{Mdist}}
\newcommand{\insmap}{\mathsf{Ins}}
\newcommand{\dltmap}{\mathsf{Del}}
\newcommand{\near}{\mathsf{Near}}
\newcommand{\Obj}{\mathsf{Obj}}
\newcommand{\Extra}{\mathsf{Ext}}
\newcommand{\ext}{extension }

\renewcommand{\l}{\ell}

\DeclareMathOperator{\capacity}{cap}
\DeclareMathOperator{\cov}{cov}
\DeclareMathOperator{\sgn}{sgn}
\DeclareMathOperator{\bulk}{bulk}
\DeclareMathOperator{\edge}{edge}
\DeclareMathOperator{\Tree}{Tree}

\endlocaldefs

\begin{document}

\begin{frontmatter}

\title{Large Deviations of Cover Time of Tori in Dimensions $d\geq 3$}
\runtitle{Large Deviations of Cover Time of Tori in Dimensions $d\geq 3$}

\begin{aug}
\author[A]{\fnms{Xinyi}~\snm{Li} \ead[label=e1]{xinyili@bicmr.pku.edu.cn}},
\author[A]{\fnms{Jialu}~\snm{Shi} \ead[label=e2]{shijialu2002@alumni.pku.edu.cn}}
\and
\author[A]{\fnms{Qiheng}~\snm{Xu} \ead[label=e3]{xuqiheng@alumni.pku.edu.cn}}
\address[A]{Peking University\printead[presep={,\ }]{e1,e2,e3}}
\end{aug}

\begin{abstract}
We consider large deviations of the cover time of the discrete torus $(\mathbb{Z}/N\mathbb{Z})^d$, $d \geq 3$ by simple random walk. We prove a lower bound on the probability that the cover time is smaller than $\gamma\in (0,1)$ times its expected value, with exponents matching the upper bound from \cite{Goodman2014} and \cite{2dLDP}. Moreover, we derive sharp asymptotics for $\gamma \in (\frac{d+2}{2d},1)$. The strong coupling of the random walk on the torus and random interlacements developed in a recent work \cite{prevost2023phase} serves as an important ingredient in the proof. 
\end{abstract}

\begin{abstract}[language=french]
Nous consid\'erons les grandes d\'eviations du temps de couverture du tore discret $(\mathbb{Z}/N\mathbb{Z})^d$, $d \geq 3$, par une marche al\'eatoire simple. Nous d\'emontrons une borne inf\'erieure pour la probabilit\'e que le temps de couverture soit inf\'erieur \`a $\gamma \in (0,1)$ fois sa  esp\'erance, avec des exposants correspondant \`a la borne sup\'erieure \'etablie dans \cite{Goodman2014} et \cite{2dLDP}. De plus, nous obtenons des asymptotiques pr\'ecises pour $\gamma \in \left(\frac{d+2}{2d},1\right)$. Le fort couplage entre la marche al\'eatoire sur le tore et les entrelacs al\'eatoires d\'evelopp\'e dans un travail r\'ecent \cite{prevost2023phase} constitue un ingr\'edient cl\'e dans la d\'emonstration.
\end{abstract}



\end{frontmatter}


\setcounter{page}{1}
\setcounter{section}{-1}

\section{Introduction}
The cover time of a Markov chain is a natural mathematical concept that finds many applications in probability theory, combinatorics and theoretical computer science. In particular, cover times by simple random walk on tori have been studied extensively. For $d\geq 3$,  asymptotics and Gumbel fluctuation of cover times are given in \cite[Chapter 7.2.2]{Aldous2014}, and \cite{Bel13} respectively. For the much harder $d=2$ case, the works \cite{Dembo2004}, \cite{Belius2017} (see also \cite{Abe2021}) and \cite{belius2020tightness} give precise first- and second-order asymptotics and the tightness of the third-order fluctuation respectively.

Intimately related to the cover time is the structure of late points or avoided points. We refer readers to \cite{abe2021avoided} for a survey on this topic. In particular, for $d\geq 3$, \cite{Bel13} and \cite{miller2017uniformity} prove the uniformity of the late points on torus and more recently in \cite{prevost2023phase}, a very strong coupling of the trace of random walk and random interlacements is obtained, yielding a structural characterization of the late points.

In this article, we focus on large deviations of the cover time of the $d$-dimensional discrete torus, $d\geq 3$. Let $X = (X_n)_{n \geq 0}$ be a discrete-time simple random walk on the $d$-dimensional torus $\T_N = (\Z/N\Z)^d$ for $d\geq 3$ and $\P$ be the law (and $\E$ the respective expectation) of the walk starting from the uniform distribution. The cover time $\frC_N$ is the first time $X$ has visited every vertex of the torus.
For $d \geq 3$, it is classical that 
\begin{equation}\label{eq:tcovdef}
    \E{\frC_N} \sim t_{\cov}\overset{\text{def.}}{=} g(0) N^d \log N^d, \text{ as }N \to \infty, 
\end{equation}
where $g(\cdot)$ denotes Green's function on $\Z^d$.
In \cite{Bel13}, finer asymptotics on $\frC_N$ are obtained: 
$$
\frC_N/(g(0)N^d)-\log N^d \Longrightarrow \mbox{Gumbel distribution as }N \to \infty.
$$
An upper bound (on the cover probability) is established for the $\varepsilon$-cover time of a torus by Brownian motion in dimensions $d \geq 3$; see \cite[Corollary 1.9]{Goodman2014}. 
As noted in \cite{2dLDP}, the proof can be adapted to the discrete setting, yielding an upper bound with a conjecturally correct exponent; see \eqref{eq:rough upper bound} below for the precise statement. 
In the 2D case, the correct order is obtained in \cite{2dLDP} with the help of soft local time techniques introduced in \cite{PT15} in the upper bound and a covering strategy (which fails in higher dimensions) for the lower bound, but sharp asymptotics remain open. We also remark that for even larger deviations at the order of the volume of torus, the general result  ``linear cover time is exponentially unlikely'' applies (see \cite{benjamini2013linear, dubroff2021linearcovertimeexponentially}), but it remains a hard problem to determine a sharp decay rate in this regime. 

\medskip
We now state our main results.
For $\gamma \in (\gamma_0, 1)$, where
\begin{equation}
\label{eq: sharp regime}
    \gamma_0 \overset{\text{def.}}{=} \frac{d+2}{2d},
\end{equation}
we give sharp asymptotics of the  probability of the event 
\begin{equation}
\label{eq: def: downward large deviation}
U_{\gamma,N}\overset{\text{def.}}{=}\big\{\frC_N \leq \gamma t_{\cov}\big\}
\end{equation}
that the cover time falls below a proportion of its expectation:
\begin{theorem} 
    \label{thm: large deviation}
    For $\gamma \in (\gamma_0, 1)$,
    \begin{equation}
    \label{eq: large deviation bound}
       \lim_{N\to\infty} \frac{\log\P \big(U_{\gamma,N}\big)}{N^{d(1-\gamma)}} = -1.
    \end{equation}
\end{theorem}

The method we use to prove the lower bound in Theorem \ref{thm: large deviation} can also be applied to obtain the following lower bound for the whole range of $\gamma \in (0,1)$:
\begin{theorem}
\label{thm: rough lower bound}
    There exists a constant $C>0$ such that for all $\gamma \in (0,1)$, 
    \begin{equation}
    \label{eq:rough lower bound}
        \P \big(U_{\gamma,N}\big) \geq \exp\Big(-C \cdot N^{d(1-\gamma)}\Big).
    \end{equation}
\end{theorem}
As discussed above, the deviation probability admits a rough upper bound:
    \begin{equation}
    \label{eq:rough upper bound}
        \P \big(U_{\gamma,N}\big) \leq \exp\Big(-N^{d(1-\gamma)+o(1)}\Big)\quad\mbox{ for   all }\gamma \in (0,1).
    \end{equation}
For completeness we will sketch a proof of \eqref{eq:rough upper bound} in Appendix \ref{apx: proof of rough upper bound}. 
Theorem \ref{thm: rough lower bound} and \eqref{eq:rough upper bound} immediately imply the following asymptotics. 
\begin{cor}
    For all $\gamma \in (0,1)$,
    \begin{equation*}
        \P \big(U_{\gamma,N}\big) = \exp\Big(-N^{d(1-\gamma)+o(1)}\Big).
    \end{equation*}
\end{cor}

We also state below asymptotics of the upward deviation probability: for $\gamma > 1$,
    \begin{equation}
    \label{eq: upward deviation}
        \P \big(\frC_N\geq \gamma t_{\cov}\big) = \big(1+o(1)\big)N^{-d(\gamma-1)}.
    \end{equation}
Although \eqref{eq: upward deviation} is a straightforward application of moment bounds, a proof sketch will still be provided in Remark \ref{rem: proof of upward} for completeness.

\subsection{Proof overview and discussion}\label{sec: outline and discussion}
A powerful tool for analyzing the behavior of simple random walk in three and higher dimensions is the model of random interlacements first introduced in \cite{Szn10}.
It is a Poissonian cloud of bi-infinite random walk trajectories.
Random interlacements naturally appear in the local limit of the trace of a simple random walk on a large torus when it runs up to a time proportional to its volume.
Several coupling results between random walks on torus and random interlacements have been established; see e.g., \cite{Windisch2008, TW11, Bel13, CT16, prevost2023phase}.
In particular, \cite{Bel13} considers the excursions between two concentric cubes to establish a coupling result at a mesoscopic scale.
Using the soft local times method introduced in \cite{PT15}, 
\cite{CT16} constructs a coupling at a macroscopic scale.
The soft local times method is further refined in \cite{prevost2023phase}, leading to a stronger coupling result with an explicit error term, namely \cite[Theorem 5.1]{prevost2023phase} (paraphrased in this article as Proposition \ref{prop: macroscopic coupling}), which is a key ingredient in this article. The notion corresponding to cover time for interlacements is called the cover level and has been studied in detail in  \cite{Belcoverlevels}.

We will now briefly outline the proof of Theorem \ref{thm: large deviation}. We begin with notation.
For $\alpha\in(0,1)$, let $\L^\alpha$ denote the set of uncovered sites up to time $\alpha g(0)N^d\log N^d$.
The elements of $\L^\alpha$ will be referred to as $\alpha$-late points, or simply late points if $\alpha$ can be specified from the context.
To better indicate the level parameter of random interlacements to be coupled with the random walk, for $\gamma>0$, we write 
\begin{equation}\label{eq:uNgamma}
    u_N(\gamma)\overset{\text{def.}}{=}\gamma g(0) \log N^d\quad\mbox{ and }\quad u_N=u_N(1).
\end{equation}
Note that $u_N(\gamma)N^d=\gamma t_{\cov}$.

We first discuss the upper bound. We divide (the bulk of) $\T_N$ into cubes of side-length $N^\gamma$ which are $\delta N^\gamma$ away from each other, where $\delta>0$ is a small quantity to be sent to $0$ eventually. 
By Corollary \ref{cor: stronger RW to independent RI coupling}, a coupling result derived from \cite[Theorem 5.1]{prevost2023phase},
we obtain that with high probability the trace of random walk up to $\gamma t_{\cov}$ intersected with these cubes can be dominated by a collection of independent RI processes indexed by these cubes with intensity slightly bigger than $u_N(\gamma)$.
Moreover, our choice of the side-length of cubes ensures that the typical cover level of each cube is approximately $u_N(\gamma)$, so the probability for the random interlacements to cover each cube ($\approx \mathrm{e}^{-1}$ thanks to \cite[Theorem 0.1]{Belcoverlevels}) is determined by the asymptotics of the cover level.
Combining the two facts gives the desired bound.

We now turn to the more delicate lower bound. For the lower bound in Theorem \ref{thm: large deviation}, we summarize our four-stage strategy as follows:

\smallskip
\noindent {\bf Stage 1}: Let the random walk roam freely until $$T_1\overset{\text{def.}}{=}\lfloor\alpha_Nt_{\cov}\rfloor,\quad \mbox{with}\quad \alpha_N\overset{\text{def.}}{=}\gamma -KN^d/t_{\cov},\;\mbox{for some large }K.$$
 This choice of $T_1$ ensures that the typical number of late points is of order $N^{d(1-\gamma)}$; see Section \ref{sec: structure of late points} for moment bounds on late points. Some specific uniformity requirements of the late points $\L^{\alpha_N}$ (see \eqref{eq:nice} for precise statements) are also typically satisfied; see Lemma \ref{lem: prob: late points nice}. 
 
\smallskip
\noindent {\bf Stage 2}: We divide $\T_N$ into the ``bulk'' part and the ``edge'' part, with the former isomorphic to $Q(0,(1-\delta)N)\subset \mathbb{Z}^d$, a $d$-dimensional cube of side-length $(1-\delta)N$  (see \eqref{eq:Qdeltaboxes} for definition), with some small $\delta>0$. We then apply the coupling Proposition \ref{prop: macroscopic coupling} in the bulk to show that the probability that the points of $\L^{\alpha_N}$ in the bulk are covered between $T_1$ and $T_2$, where
 $$T_2\overset{\text{def.}}{=}\lfloor\beta_Nt_{\cov}\rfloor,\quad \mbox{with}\quad \beta_N\overset{\text{def.}}{=}\gamma -5\varepsilon N^d/t_{\cov}\;\mbox{ (for some small }\varepsilon>0),$$ 
 exceeds $\exp\big(-(1+\eta) N^{d(1-\gamma)}\big)$ for a small $\eta=\eta(\delta,\varepsilon)>0$ (to be sent to $0$ eventually) by turning the estimate on the cover probability in the bulk into a cover-level large deviation probability estimate for random interlacements and by benefiting from Harris-FKG inequality for Poisson point processes. See Proposition \ref{prop: prob: lower bound on event E} for the precise statement. Restricting $\gamma$ to $(\gamma_0,1)$ ensures that the coupling error is ignorable even with respect to the cover probability.

\smallskip
\noindent {\bf Stage 3}: We perform ``surgeries'' on the trace of random walk between $T_2$ and 
$$
T_3\overset{\text{def.}}{=}\lfloor\gamma t_{\cov}-\varepsilon N^d\rfloor
$$ by inserting meticulously designed loops (paths starting and ending at the same site) into the trajectory (see Section \ref{sec: Construction of the mapping Psi and its properties} for details) 
to cover points of $\L^{\alpha_N}$ in the edge part. 
Specifically, we devide the late points in the edge into small clusters; see below \eqref{eq: above loop insertion description} for details. 
The loop insertion consists of two steps. 
We first cover each representative vertex of the clusters by inserting disjoint self-avoiding loops that begin at the closest point on the trajectory to each respective vertex. 
We then cover the remaining vertices of each cluster by adding loops that begin at each representative vertex.
We refer readers to Section \ref{sec: Construction of loops to be inserted} for the detailed construction of loops. 
For paths that are ``nice'' during Stage 1 and satisfy some extra requirements (see Definition \ref{def: good paths prop} for details), the total length of these loops is bounded by a constant times the number of late points, which does not exceed $\varepsilon N^d$; see Lemma \ref{lem: upper bound for path with loops} for a precise statement. Moreover, the whole loop-insertion mechanism is designed in a way so that it is possible to recover the original trajectory deterministically from the trajectory after insertion; see Proposition \ref{prop: Psi properties}, in particular \eqref{eq: injective mapping Psi}, for details. This allows us to establish a lower bound on the probability that the late points in the edge are covered after $T_2$; see Proposition \ref{prop: loop insertion ineq} for details.
 
\smallskip
\noindent {\bf Stage 4}: We again let the random walk roam freely until $T_4 \overset{\text{def.}}{=} \lfloor \gamma t_{\cov} \rfloor$. Note that the bound on the total length of the loops ensures that there is still time left after loop-insertion surgeries. We remark that although the whole torus has already been covered at the end of Stage 3, we still call this part a ``stage'' for the completeness of our cover strategy. 

\smallskip
To establish Theorem \ref{thm: rough lower bound}, we skip Stage 2 in the strategy above and proceed to loop insertion directly to cover late points at the end of Stage 1. As no coupling with random interlacements is involved, such a strategy works for all $\gamma\in(0,1)$, but yields only a lower bound that is sub-optimal in the pre-factor; see Section \ref{sec: proof of the rough lower bound} for details.

\medskip

We now make a few remarks on the proof strategy and a few related problems.

We first compare our strategy with that of the 2D case treated in \cite{2dLDP}. Both proofs crucially rely on ``soft local time'' techniques first introduced in \cite{PT15} to decouple trace of random walk. However, thanks to the Bernoulli nature of late points (see \cite{prevost2023phase} for more discussions) for $d\geq 3$, in our case we are able to obtain a sharp asymptotic deviation probability (for at least large values of $\gamma$), while obtaining sharp asymptotics for 2D deviation probability remains a difficult question. 
For the upper bounds, in both cases a decoupling of the trace of random walk in mesoscopic cubes is involved. However, in the 2D case, boxes of a different side-length\footnote{Of order $N^{\sqrt{\gamma}}$ instead of $N^\gamma$ in our case} are chosen and the trace of random walk are approximated by collection of random walk excursions instead of interlacements\footnote{Recall that there does not exist a translation-invariant interlacement model on $\Z^2$.}. For lower bounds the strategy employed in either case is very specific and not easily adaptable to the other case; see also the discussions in Section 1 of \cite{2dLDP}.

\smallskip 

It is worth noting that similar sharp asymptotics can be obtained for a counterpart problem regarding random interlacements which is reminiscent of the ``hard-wall'' conditioning of Gaussian free field \cite{BDZ95,bolthausen2001entropic}. More precisely, let $\I^u$ stand for the interlacement set at level $u>0$ (see Section \ref{sec: random interlacements preliminary} for precise definitions), and $\mathfrak{M}_N=\inf\{u \geq 0:Q(0,N) \subset \I^u\}$ for the cover level of the $d$-dimensional cube $Q(0,N)$ of side-length $N$. 
It is possible to establish that 
\begin{equation}
\label{eq: RI large deviation}
    \mathbb{P}(\mathfrak{M}_N \leq \gamma u_N) = \exp\left(-(1+o(1))N^{d(1-\gamma)}\right), \qquad \mbox{for all } \gamma \in \left(\frac{2}{d},1\right).
\end{equation}
The lower bound in \eqref{eq: RI large deviation} (which actually works for all $\gamma\in (0,1)$) follows readily from the Harris-FKG inequality for random interlacements.
For the upper bound, we again divide $Q(0,N)$ into ``well-separated'' cubes of mesoscopic side-length and 
apply \cite[Theorem 5.1]{prevost2023phase} to localize random interlacements, akin to the proof of upper bound in Theorem \ref{thm: large deviation}. Thanks to a better error bound for interlacements developed in \cite[Theorem 5.1]{prevost2023phase} (see also Proposition \ref{prop: RW/RI to independent RI coupling} in our work; readers may compare \eqref{eq: RW to independent RI} and \eqref{eq: RI to independent RI} therein), and a bootstrap argument, the proof strategy can work for a larger range of $\gamma$ than in the case of random walk\footnote{In fact, the main obstacle of extending the sharp asymptotics for random walk to smaller $\gamma$ 
lies in the fact that the probability of the event $U_{\gamma,N}$ becomes smaller than the error term in the coupling derived from \cite[Theorem 5.1]{prevost2023phase}. Hence, if one is able to improve the error bound for the random walk by (roughly speaking) removing the square root in the exponent in \eqref{eq: macroscopic coupling}, then \eqref{eq: large deviation bound} will work for $\gamma\in (\frac{2}{d},1)$.}. We will give a sketch of the proof in Appendix \ref{apx: proof of RI upper bound}.  

However, when $\gamma \in (0,\frac{2}{d})$, another strategy involving the so-called ``tilted interlacements'' kicks in, and changes the asymptotics completely. We now briefly explain  this strategy. To study the asymptotic probability that random interlacements with low intensity disconnects a macroscopic body from infinity, the authors of \cite{LiSzn14} introduce a Markov jump process with a specific spatially-regulated local drift towards $Q(0,N)$ and construct a spatially inhomogeneous variant of random interlacements (which they name by {\it tilted interlacements}) with the aforementioned Markov process serving as the underlying random walk in the construction of interlacements. 
In a neighborhood of $Q(0,N)$, tilted interlacements (locally) look like random interlacements of an uplifted intensity. In a forthcoming work \cite{LiShi2025}, the first two authors of this work obtain the following lower bound using the idea of tilted interlacements:
\begin{equation}\label{eq: RI small gamma}
 \mathbb{P}(\mathfrak{M}_N \leq \gamma u_N) \geq \exp\left(-\left(\left(\sqrt{\gamma}-\sqrt{\frac{2}{d}}\right)^2g(0){\rm cap}_{\mathbb{R}^d}([0,1]^d)+o(1)\right)  N^{d-2}\log N\right),\; \mbox{for all } \gamma \in \left(0,\frac{2}{d}\right),     
\end{equation}
where 
${\rm cap}_{\mathbb{R}^d}(\cdot)$ stands for the Brownian capacity. In the meantime, they 
obtain an upper bound that matches \eqref{eq: RI small gamma} in the principal exponential order. 

For a simple random walk $(Y_n)_{n\geq 0}$ on $\mathbb{Z}^d$,  by employing a strategy involving  ``tilted walk'' first introduced in \cite{li2017lower}, which resembles to some extent the strategy described above, a similar asymptotic lower bound for the probability that $Y_n$ covers the macroscopic cube $Q(0,N)$ emerges (see \cite{LiShi2025} for more discussions):
\begin{equation}\label{eq: RW cover box}
\mathsf{P}_0\big( Q(0,N) \subset Y_{[0,\infty)}\big) \geq \exp\left(-\left(\frac{2}{d}\cdot g(0){\rm cap}_{\mathbb{R}^d}([0,1]^d)+o(1)\right)  N^{d-2}\log N\right).
\end{equation}
By a natural coupling of $Y_{[0,\infty)}$ and $\I^u$ conditioned on $0\in \I^u$ (see the paragraph above \cite[(1.8)]{Szn17} for a similar argument) for arbitrarily small $u>0$, it is almost immediate that \eqref{eq: RW cover box} is asymptotically sharp. However, it should be noted that such a strategy fails to work for the original problem, as it  does not involve a constraint on the total time spent in $Q(0,N)$. 
This strongly suggests that ideas of this work are not sufficient for treating the case $\gamma \in (0,\frac{2}{d})$ and a new strategy is required.

\smallskip

Before ending this subsection, we briefly mention a ``dual'' problem of what we consider in this work, namely the downward large deviation of maximal local time for simple random walk on $\Z^d$ for $d \geq 3$. For simplicity we restrict to discrete-time walks. We still let $(Y_n)_{n\geq0}$ stand for the simple random walk on $\mathbb{Z}^d$, $d\geq 3$. Let $\xi(x,n)$ be its local-time process at site $x\in\mathbb{Z}^d$ at time $n\geq 0$ and $\xi^*(n):=\max_{x\in\mathbb{Z}^d}\xi(x,n)$ the maximal local time at time $n$, respectively. It is classical (see \cite[Theorem 1.3]{erdos1960some}) that 
\begin{equation}\label{eq:max}
\lim_{N\to\infty} \frac{\xi^*(N)}{\log N} =\alpha :=-\frac{1}{\log(1-{\rm Es}_d)},
\end{equation}
where ${\rm Es}_d:=\mathsf{P}_0[0\notin Y_{(1,\infty)}]$ is the escape probability. An interesting question is, for any $d\geq 3$ and all $\gamma\in(0,1)$, does one have
\begin{equation}\label{eq: rough xistar}
\mathsf{P}\big(\xi^*(N) \leq  \gamma \cdot \alpha \log N\big)=\exp\Big(-N^{1-\gamma+o(1)}\Big)?
\end{equation}
Moreover, does sharp asymptotics similar to \eqref{eq: large deviation bound} also hold? More precisely, for all $d\geq 3$ does there exist a $C=C(d)>0$ such that for all $\gamma\in (\gamma'_0,1)$ for some $\gamma'_0=\gamma'_0(d)<1$,
\begin{equation}\label{eq: sharp xistar}
P\big( \xi^*(N) \leq \gamma \cdot \alpha \log N\big)=\exp\Big(-\big(C+o(1)\big)N^{1-\gamma}\Big)?
\end{equation}
A simple strategy of chopping $Y_{[0,N]}$ into segments of length $N^{\gamma}$ yields an upper bound for both \eqref{eq: rough xistar} and \eqref{eq: sharp xistar}.  It would be very nice to develop a matching sharp lower bound for a largest possible range of $\gamma$.

\subsection{Organization of this article}
Section \ref{sec: notation} introduces the relevant notation, briefly discusses the model of random interlacements and its connection with random walks, and presents moment bounds and concentration results for late points. 
Section \ref{sec: proof of upper bound} establishes the upper bound in Theorem \ref{thm: large deviation}.
Section \ref{sec: proofs of lower bounds} demonstrates the lower bounds for Theorem \ref{thm: large deviation} and Theorem \ref{thm: rough lower bound}.
Section \ref{sec: Construction of the mapping Psi and its properties} is dedicated to the surgery of loop-insertion for the lower bounds. Its main goal is the proof of Proposition \ref{prop: loop insertion ineq}, which is relatively independent of other materials in Section \ref{sec: proofs of lower bounds}. 
Section \ref{sec: tables of symbols} collects the notation from Sections \ref{sec: proofs of lower bounds} and \ref{sec: Construction of the mapping Psi and its properties} for the reader's convenience. 
In \hyperref[apx: proof of rough upper bound]{Appendix}, we sketch the proofs of \eqref{eq:rough upper bound} and the upper bound of \eqref{eq: RI large deviation}.

\medskip

We conclude this section with our convention regarding constants.
In the rest of the article, constants without numeric subscripts may vary from place to place, while constants with numeric subscripts remain fixed within the same section.
All constants depend implicitly on the dimension $d$. Unless stated otherwise, all constants are positive.

\section{Notation and preliminaries}
\label{sec: notation}

We first set up some basic notation. For any positive sequence $a_N$, we say $a_N=o(1)$ if $\lim_{N \to \infty}a_N=0$.
Given two positive functions $f$ and $g$, we write $f \sim g$ if $\lim\frac{f}{g}=1$, where the limiting process can be inferred from the context. For any $a \in \mathbb{R}$, let $\lfloor a \rfloor$ denote the greatest integer less than or equal to $a$, and $\lceil a \rceil$ denote the smallest integer greater than or equal to $a$. 

We use $\phi_N$ to denote the canonical map from the lattice $\Z^d$ to the torus $\T_N$. For convenience, we always use $\x$ to denote the projection image $\phi_N(x)$ for $x\in \Z^d$. 
The addition and subtraction of vectors on $\T_N$ are addition and subtraction modulo $n$ in each coordinate.
For $x \in \Z^d$ and $R>0$, we let $Q(x,R)=x+([-\lfloor(R-1)/2\rfloor,\lceil(R-1)/2\rceil]\cap\Z)^d$ be the cube of side-length $R$ around $x$ in $\Z^d$, and $Q(\x,R)=\phi_N(Q(x,R))$ be the cube in the torus. 
For brevity, we also write
\begin{equation}\label{eq:Qdeltaboxes}
    Q_R(x)=Q(x,R),\qquad Q_\delta=Q(0, (1-\delta)N)\qquad\mbox{ and }\qquad \BQ_\delta=Q(\bs{0}, (1-\delta)N).
\end{equation}
We write $\varphi_N=\phi_N\mid_{Q(0,N)}$ for the bijection from $Q(0,N)$ to $\T_N$.

For $x \in \Z^d$, we use $|x|_\infty$ and $|x|_1$ to denote the $l^{\infty}$ norm and the $l^{1}$ norm respectively.
With slight abuse of notation, we also use them to denote the induced norms on $\T_N$.
For clarity, we let $d_\infty(x,y)=|x-y|_{\infty}$ and $d_1(x,y)=|x-y|_1$ for $x,y \in \Z^d$ (or $\T_N$). 
For $x\in \T_N$ and $K\subset \T_N$ (or $x\in \Z^d$ and $K\subset \Z^d$), we write $d_\infty(x,K)=\min\{d_\infty(x,y):y\in K\}$ for the $l^\infty$-distance from $x$ to $K$. Write $e_1=(1,0,\ldots,0)$ and define $e_2,\ldots,e_d$ similarly.

We write $\P_{x}$ for the canonical law of the discrete-time simple random walk on $\Z^d$ (or $\T_N$, depending on context) starting at $x\in \Z^d$ (or $\T_N$). Let $\P =N^{-d}\sum_{x\in \T_N}\P_{x}$ be the law of the random walk starting from the uniform distribution on $\T_N$.
We also write $\E$ or $\E_{x}$ for the expectation under $\P$ or $\P_x$ respectively.
We denote by $(X_n)_{n\geq 0}$ the corresponding random walk process under $\P_x$ or $\P$, and by $X\langle t_1,t_2\rangle=\{X_n:n \in \langle t_1,t_2\rangle\}$ the trace left by the random walk from time $t_1$ to $t_2$, where $\langle \in \{(,[\}$ and $\rangle \in \{),]\}$. 
We write $\CF_n$ for the natural filtration generated by $(X_k)_{0\leq k\leq n}$.

For $k \in \N$, we write 
\begin{equation}
\label{eq: random walk paths}
    \Omega_{\x,k}=\{(\x_0,\x_1,\ldots,\x_{k}) \subset (\T_N)^{k}:\x_0=\x,\,|\x_{i}-\x_{i-1}|_1=1 \text{  for all } 1 \leq i \leq k\}
\end{equation}
for the collection of all possible trajectories of random walk on $\T_N$ started from $\x$ up to time $k$. 
We also write $\Omega_{k}=\bigcup_{\x \in \T_N}\Omega_{\x,k}$ for all possible random walk trajectories up to time $k$ and $\Omega=\bigcup_{k \in \N}\Omega_k$ for the collection of trajectories of finite length.
For a path $\omega \in \Omega_k$, we let $L(\omega)=k$ be the length of $\omega$.
Moreover, for two paths $\omega,\omega' \in \Omega$ and $0 \leq t_1 \leq t_2 \leq \min\{L(\omega),L(\omega')\}$, we write $\omega=\omega' \mid_{ [t_1,t_2]}$ if $\omega(k) = \omega'(k)$ for all $t_1 \leq k \leq t_2$.
For simplicity, we write $\omega=\omega' \mid_{t_1}$ for $\omega=\omega' \mid_{[0,t_1]}$.
For $k\in \N$, we define a map $\pi_k$ from events in $\CF_k$ to related random walk paths of length $k$: 
\begin{equation}
\label{eq: paths correspond to event}
    \pi_k(A)=\left\{\omega=(\x_0,\x_1,\ldots,\x_k) \subset \Omega_k \mid
    \{X_{\cdot} = \omega \mid_{k} \} \subset A 
    \right\}, \quad \text{for }A\in \CF_k.
\end{equation}

For $K \subset \Z^d$ or $\T_N$, we define $H_K=\inf\{n\geq 0:X_n\in K\}$ and $\widetilde{H}_K=\inf\{n\geq 1:X_n\in K\}$ the respective entrance time and hitting time of $K$. We write $H_x$ and $\widetilde{H}_x$ for $H_{\{x\}}$ and $\widetilde{H}_{\{x\}}$ respectively. Let $\frC_N=\max_{x \in \T_N}H_{x}$ be the cover time of the torus $\T_N$.

For a finite $K\subset \Z^d$, we write $e_K(\cdot)$ for the equilibrium measure of $K$, 
\begin{equation*}
    e_K(x)=\P_x(\widetilde{H}_K=\infty)1_K(x), 
\end{equation*}
which is supported on the internal boundary of $K$. 
We denote by $\capacity(K)$ the total mass of $e_K$ and call it the capacity of $K$, and by $\overline{e}_K=\frac{e_K}{\capacity(K)}$ the normalized equilibrium measure. For $K\subset \T_N$ with $l^\infty$-diameter $\delta(K)<N$, we define its capacity as follows: if $K\subset Q(\bs{0},N-1)$, we let $\capacity(K)\overset{\text{def.}}{=}\capacity\big(\phi_N^{-1}(K)\big)$ (see the definition of $\phi_N$ at the beginning of this section); for other $K$'s with $\delta(K)<N$, we define $\capacity(K)$ by translation invariance.
Asymptotics of capacity are well known, and we will use the following bounds (see e.g., \cite[Section 2.2]{lawler2012intersections}),
\begin{equation}
\label{eq: capacity bound}
    cR^{d-2} \leq \capacity(Q(0,R)) \leq CR^{d-2}, \quad R \geq 1.
\end{equation}

For two vertices $x,y \in \Z^d$, we denote by $g(x,y)$ the Green's function of $X$, the simple random walk on $\Z^d$, under $\P_x$.
Write $g(x-y)\overset{\text{def.}}{=}g(x,y)$ for short. From {\cite[Lemma 2.1]{prevost2023phase}}, one has for all integers $n \geq 0$,
    \begin{equation}
    \label{eq: green function comparison}
        \sup_{|x|_1>n} g(x) < \sup_{|x|_1=n} g(x).
    \end{equation}

\subsection{Random interlacements}
\label{sec: random interlacements preliminary}
In this section, we present a concise overview of the model of random interlacements on $\Z^d$. We refer readers to \cite{DRS14} for a more detailed introduction.
Let $\omega$ denote the random interlacements process on $\Z^d$, with the law $\mathbb{P}$, as defined in \cite{Szn10}.
This process is a Poisson point process (with a certain translation-invariant intensity measure) of bi-infinite random walk trajectories, each associated with a positive label. The trace of the trajectories with labels less than $u>0$, called the {\it interlacement set} at level $u$, is denoted by $\I^u$. Its complement, the {\it vacant set} $\V^u$ at level $u$, satisfies the following characterization:
\begin{equation}
    \label{eq: prob: interlacements vacant set}
    \mathbb{P}(K\subset\V^u)= \exp\{- u \capacity(K)\}\mbox{ for any }K\subset\subset\Z^d.
\end{equation}
In particular, one has (see e.g., \cite[(1.3.7) and (1.3.8)]{DRS14})
\begin{equation}
    \label{eq: RI one point prob}
    \mathbb{P}(x\in \V^u) = \exp\Big(-\frac{u}{g(0)}\Big),
\end{equation}
and
\begin{equation}
    \label{eq: RI two points prob}
    \mathbb{P}(x, y \in \V^u) = \exp\Big(-\frac{2u}{g(0)+g(x-y)}\Big).
\end{equation}

The next lemma gives an upper bound on sums of two-point probability for the vacant set 
which will be useful in bounding certain sums of two-point probability for late points of random walk; see the proof of Lemma \ref{lem: second moment} for details. It follows directly from \cite[Lemma 1.5]{belius2011} by taking $a=N$ therein\footnote{In fact, there should be an extra constant factor of $C$ before $a^{d+1}$ in \cite[(1.39)]{belius2011}.} and noting that $|F|\leq (2N)^d$ if $F \subset Q(0,2N-1)$.
\begin{lemma}
\label{lem: RI second moment}
     There exist constants $C>0$ and $c_1>1$ such that for all $u\geq 0$, $N\in \N$ and $F \subset Q(0,2N-1)$, 
    \begin{equation*}
        \sum_{v\in F}\mathbb{P}\Big(0,v\in \V^{g(0)u}\Big) \leq
        \mathrm{e}^{-2u}\big(|F| +CuN^2\big)+C\exp(-c_1u).
    \end{equation*}
\end{lemma}

The next lemma shows that the probability that a cube of size $N^{\gamma}$ is covered by random interlacements with intensity close to $u_N(\gamma)$ (see \eqref{eq:uNgamma} for the definition) converges to $\mathrm{e}^{-1}$, which follows from \cite[Theorem 0.1]{Belcoverlevels} stating that the fluctuations of cover level for interlacements are governed by a Gumbel variable.
\begin{lemma}
    For $\gamma>0$, let $\displaystyle\lim_{N\to\infty}R_N/N^\gamma=1$ and $v_N = u_N(\gamma)(1+\rho_N)$ with \\ $\displaystyle\lim_{N\to\infty}\rho_N\log N=0$. 
    One has
    \begin{equation}
    \label{eq: RI cover level}
        \lim_{N\to \infty} \mathbb{P}\big(Q(0,R_N)\subset \I^{v_N}\big) = \mathrm{e}^{-1}.
    \end{equation}
\end{lemma}
\begin{proof}
    Let $$z_N=\gamma(1+\rho_N)\log N^d-\log (R_N)^d=\log\bigg(\frac{N^\gamma}{R_N}\bigg)^d+d\gamma(\rho_N\log N).$$ We apply \cite[Theorem 0.1]{Belcoverlevels} with the choice of $A=Q(0,R_N)$ and $z=z_N$, and get
    \begin{equation}
    \label{eq: cover level ineq}
        \Big|\mathbb{P}\Big(Q(0,R_N)\subset \I^{g(0)(\log (R_N)^d+z_N)}\Big)-\exp(-\mathrm{e}^{-z_N})\Big| \leq c(R_N)^{-c_2}, 
    \end{equation}
    where $c_2>0$ is a constant. 
    It is elementary that
    \begin{equation}
    \label{eq: elementary ineq}
        \big|\exp(-\mathrm{e}^{-z_N})-\mathrm{e}^{-1}\big|=\big|\exp(-\mathrm{e}^{-z_N})-\exp(-\mathrm{e}^0)\big|\leq c|z_N|.
    \end{equation}
    Note that $z_N\to 0$ as $N\to \infty$. 
    Combining \eqref{eq: cover level ineq} and \eqref{eq: elementary ineq} gives \eqref{eq: RI cover level}.
\end{proof}

It is classical that the FKG inequality holds for Poisson point process, and in particular random interlacements (see e.g., \cite[Theorem 3.1]{Teixeira2009}). The following lemma, which gives the lower bound on the probability that a set $F$ is covered by the random interlacements at level $u$, is a direct consequence of the FKG inequality. 
\begin{lemma}
\label{lem: FKG inequality}
    For all $u>0$ and $F \subset \Z^d$, 
    \begin{equation}
        \mathbb{P}(F\subset \I^{u}) \geq \mathbb{P}(0\in \I^u)^{|F|}\overset{\eqref{eq: RI one point prob}}{=}\big(1-\exp[-u/g(0)]\big)^{|F|}.
    \end{equation}
\end{lemma}

\subsection{Couplings of random walk and random interlacements}
\label{sec: Couplings results subsec}
In this section, we present couplings between the random walk on the torus and random interlacements, which is a crucial tool in understanding the picture left by the trace of random walk. Our coupling results are based on \cite[Theorem 5.1]{prevost2023phase}. Similar couplings at a mesoscopic scale can be found in \cite{Windisch2008, TW11, Bel13}, and a macroscopic coupling with less explicit error controls has been proved in \cite{CT16}. 

First, we couple the trace of the random walk (resp.\ the picture of random interlacements) intersected with disjoint cubes of side-length $R$, with separations at least of order $R$, with independent random interlacements. For $R>0$, we say a set $\F \subset \T_N$ (resp.\ $F \subset Q(0,N)$) is $R$-well separated if $\F \cap Q(\x,2R)=\{\x\}$ for all $\x \in \F$ (resp.\ if $F \cap Q(x,2R)=\{x\}$ for all $x \in F$). 
For $\F\subset\T_N$, we write $F$ for the image of $\F$ under the canonical map from $\T_N$ to $Q(0,N)$ in this subsection. 
The coupling result is as follows. 

\begin{prop}
\label{prop: RW/RI to independent RI coupling}
    Let $\delta > 0$. Assume that $R \in [1, \frac{N}{1+\delta}]$ and $\F \subset \T_N$ (resp.\ $F \subset Q(0,N)$) is $(1+\delta)R$-well separated. Then for every $u>0$, $\rho \in (0,1)$, there exists a probability measure $\WP$ (resp.\ $\widetilde{\mathbb{P}}$) extending $\P$ (resp.\ $\mathbb{P}$) and carrying independent random interlacements $\left(\I^{(x),u(1\pm\rho)}\right)_{x\in F}$ with the same law as $\I^{u(1\pm\rho)}$, such that
    \begin{multline}
    \label{eq: RW to independent RI}
        \WP\Big(\I^{(x), u(1-\rho)} \cap Q_R(x) \subset X[0, uN^d] \cap Q_R(\x) \subset \I^{(x), u(1+\rho)} \cap Q_R(x) \text{\quad for all } x \in F \Big) \\
        \geq 1-C|F|R^{2d}\lceil uR^{d-2}\rceil \exp\Big(-c\rho\sqrt{uR^{d-2}}\Big), 
    \end{multline}
    resp.
    \begin{multline}
    \label{eq: RI to independent RI}
        \widetilde{\mathbb{P}}\Big(\I^{(x), u(1-\rho)} \cap Q_R(x) \subset \I^u \cap Q_R(\x) \subset \I^{(x), u(1+\rho)} \cap Q_R(x) \text{\quad for all } x \in F \Big) \geq \\
        1-C|F|R^{2d}\exp\big(-c\rho^2 uR^{d-2}\big)
    \end{multline}
    for some constants $c,C \in (0, \infty)$ depending only on $\delta$ and $d$.
\end{prop}

The proof of Proposition \ref{prop: RW/RI to independent RI coupling} is of the same spirit as \cite[Proof of Theorem 1.2]{prevost2023phase}. Hence, we only sketch the proof here. 

\begin{proof}[Proof sketch]
    We begin the proof in the case of random walk under $\P$. By \cite[Theorem 5.1]{prevost2023phase}, there exists a probability measure $\WP_{\bs{0}}$ extending $\P_{\bs{0}}$, such that \cite[(5.1), (5.2) and (5.3)]{prevost2023phase} hold. Define $\WP = N^{-d}\sum_{x\in \T_N}\WP_x$, where $\WP_x$ is obtained from $\WP_{\bs{0}}$ through translation by $x\in \T_N$. It is easy to check that $\WP$ also carries a family of interlacements processes $(\omega^{(x)})_{x\in Q(0,N)}$ satisfying \cite[(5.1), (5.2) and (5.3)]{prevost2023phase} (with $\WP$ in place of $\widetilde{\mathbf{P}}_{\bs{0}}$ therein). 
    Let $\l^{(x)}_{\cdot,u}$ denote the field of local times associated with $\omega^{(x)}$ at level $u$; cf.\
    \cite[(2.12)]{prevost2023phase}. 
    Then we define $\I^{(x),u(1\pm\rho)}$ to be the random interlacements associated with $\l^{(x)}_{\cdot,2u(1\pm \rho/3)}-\l^{(x)}_{\cdot,u(1\mp \rho/3)}$, which has the same law as $\l^{(x)}_{\cdot,u(1\pm \rho)}$. Hence, $\I^{(x),u(1\pm\rho)}$ has the same law as $\I^{u(1\pm\rho)}$ for all $x\in Q(0,N)$.
    Since $\F$ is $(1+\delta)R$-well separated and by \cite[(5.2)]{prevost2023phase}, the random interlacements $\I^{(x),u(1\pm\rho)}$, $x\in F$, are independent. Let $\l_{\cdot, u}$ denote the field of local times of random walk $X$ under $\P$ up to time $uN^d$; see \cite[(2.2)]{prevost2023phase}.
    Since $\l_{\cdot, 2u}-\l_{\cdot, u}$ under $\WP$ has the same law as $\l_{\cdot, u}$ under $\P$, \eqref{eq: RW to independent RI} follows from \cite[(5.3)]{prevost2023phase}. The proof in the case of random interlacements is similar, and \eqref{eq: RI to independent RI} follows from \cite[(5.3')]{prevost2023phase}.
\end{proof}

We now present the macroscopic coupling \cite[Corollary 5.2]{prevost2023phase} (see also \cite[Theorem 4.1]{CT16}) between the trace of the random walk in $\BQ_\delta$ and that of random interlacements in $Q_\delta$, where $\delta \in (0,1)$. 

\begin{prop}[\text{\cite[Corollary 5.2]{prevost2023phase}, \cite[Theorem 4.1]{CT16}}]
\label{prop: macroscopic coupling}
    Let $\delta \in (0,1)$. For every $u>0$, $\rho \in (0,1)$ and $N \in \N$, there exists a coupling $\WP$ between the random walk under $\P$ and the random interlacements $\I^{u(1\pm\rho)}$ such that
    \begin{equation}
    \label{eq: macroscopic coupling}
        \WP\Big(\I^{u(1-\rho)} \cap Q_\delta \subset X[0, uN^d] \cap \BQ_\delta \subset \I^{u(1+\rho)} \cap Q_\delta \Big) \geq \;1-CN^{2d} \lceil uN^{d-2}\rceil \exp\Big(-c\rho\sqrt{uN^{d-2}}\Big)
    \end{equation}
    holds for some constants $c,C \in (0, \infty)$ depending only on $\delta$ and $d$.
\end{prop}

Combining Proposition \ref{prop: RW/RI to independent RI coupling} in the case of random interlacements and Proposition \ref{prop: macroscopic coupling}, we can give a stronger version of Proposition \ref{prop: RW/RI to independent RI coupling} in the case of random walk with the cost of restricting the cubes of size $R$ around $x \in F$ to be contained in $Q_\eta$. 

\begin{cor}
\label{cor: stronger RW to independent RI coupling}
    Let $\delta > 0, \eta \in (0,1) \text{ and } N \in \N$. Assume that $R \in (N^{\frac{1}{2}}, \frac{N}{1+\delta}]$, $\F \subset \T_N$ is $(1+\delta)R$-well separated, and $Q_R(x) \subset Q_\eta$ for all $x \in F$. Then for every $u \geq 1$ and $\rho \in (0,1)$, there exists a probability measure $\WP$ extending $\P$ and carrying independent random interlacements $\left(\I^{(x),u(1\pm\rho)}\right)_{x\in F}$ with the same law as $\I^{u(1\pm\rho)}$, such that
    \begin{multline}
    \label{eq: stronger RW to independent RI}
        \WP\Big(\I^{(x), u(1-\rho)} \cap Q_R(x) \subset X[0, uN^d] \cap Q_R(\x) \subset \I^{(x), u(1+\rho)} \cap Q_R(x) \text{\quad for all } x \in F \Big) \geq \\
        1-CuN^{3d}\exp\Big(-c\rho^2 \sqrt{uN^{d-2}}\Big)
    \end{multline}
    holds for some constants $c,C \in (0, \infty)$ depending only on $\delta$, $\eta$ and $d$.
\end{cor}

\subsection{Structure of late points}
\label{sec: structure of late points}
We first give some facts on the law of $\L^\alpha$, \cite[(6.4) and (6.5)]{prevost2023phase}, which can be proved by using the macroscopic coupling between the random walk on the torus and random interlacements. We generalize the regime of $\alpha$ therein and refer to \cite[Lemma 6.1]{prevost2023phase} for the proof, which still goes through under such generalization.

\begin{lemma}[\text{\cite[(6.4) and (6.5)]{prevost2023phase}}]
\label{lem: prob: late points}
    For all $\alpha_1>\alpha_0>0$ and $\beta_0>0$, there exists a constant $C=C(\alpha_0,\alpha_1,\beta_0)>0$ such that for all $N \in \N$ and $G\subset \T_N$ such that $\capacity(G) \leq \beta_0$,
    \begin{equation}
    \label{eq: prob: late points}
        \left|\frac{\P(G \subset \L^\alpha)}{N^{-d\alpha g(0)\capacity(G)}} - 1\right| \leq C\frac{(\log N)^{3/2}}{N^{(d-2)/2}} \quad 
         \forall\, \alpha \in [\alpha_0,\alpha_1].
    \end{equation}
    In particular, for $G=\{0\}$, we have uniformly in $\alpha \in [\alpha_0,\alpha_1]$, 
    \begin{equation}
    \label{eq: prob: single late point}
        \P(0 \in \L^\alpha) = \big(1+o(1)\big)N^{-d\alpha}.
    \end{equation}
\end{lemma}

Let $\L^\alpha_F = F\cap \L^\alpha$ denote\footnote{The readers should note that this differs from the definition \cite[(6.2)]{prevost2023phase}.} the set of $\alpha$-late points in $F \subset \T_N$. 
The following lemma characterizes the number of late points in a set $F \subset \T_N$ of size $\lambda N^d$, where $\lambda$ is a constant.

\begin{lemma}
\label{lem: late points concentration}
    Let $\varepsilon >0$. For every $\alpha_0, \lambda_0 \in (0,1)$, uniformly in $\alpha \in [\alpha_0,1)$, $\lambda \in [\lambda_0,1]$ and $F \subset \T_N$ satisfying $|F|=\lfloor\lambda N^d\rfloor$, 
    \begin{equation}
    \label{eq: late points concentration}
        \lim_{N \to \infty}\P\left(\Big||\L^\alpha_F|-\lambda N^{d(1-\alpha)}\Big|>\varepsilon \lambda N^{d(1-\alpha)}\right) = 0.
    \end{equation}
\end{lemma}

The proof is based on the second moment method, that is calculating the ``centered'' second moment $\E\big(|\L^\alpha_F|-\lambda N^{d(1-\alpha)}\big)^2$ using \eqref{eq: prob: single late point} and the following estimate on the sums of ``two-point probability'', Lemma \ref{lem: second moment}. We will come back to its proof after proving Lemma \ref{lem: second moment}. 

\begin{lemma}
\label{lem: second moment}
        For every $\alpha_0, \lambda_0 \in (0,1)$, uniformly in $\alpha \in [\alpha_0,1)$, $\lambda \in [\lambda_0,1]$ and $\F \subset \T_N$ satisfying $|\F|=\lfloor\lambda N^d\rfloor$, 
        \begin{equation}
        \label{eq: second moment estimate}
            \sum_{\x\neq \y \in \F}{\P(\x,\y \in \L^\alpha)} \leq \big(1+o(1)\big)\big(\lambda N^{d(1-\alpha)}\big)^2.
        \end{equation}
\end{lemma}

\begin{proof}
    We write $F$ for the image of $\F$ under the canonical map from $\T_N$ to $Q(0,N)$. 
    To simplify the statements below, we define $F_x = \phi_N^{-1}(\F)\cap Q(x,N) \setminus \{x\}$ for $x\in F$ (see the beginning of Section \ref{sec: notation} for the definition of $\phi_N$). 
    By \eqref{eq: prob: late points} and \eqref{eq: prob: interlacements vacant set}, we get uniformly in $\alpha \in [\alpha_0,1)$, 
    \begin{equation}
    \label{eq: second moment approximated by RI}
        \sum_{\x\neq \y \in \F}{\P(\x,\y \in \L^\alpha)} \leq 
        \big(1+o(1)\big)\sum_{x\in F}\sum_{y\in F_x}{\mathbb{P}\Big(x,y \in \V^{u_N(\alpha)}\Big)}.
    \end{equation}
    By Lemma \ref{lem: RI second moment} with $u=\alpha\log N^d$ and $F_x-x$ in place of $F$, we have
    \begin{equation*}
        \sum_{y\in F_x}{\mathbb{P}\Big(x,y \in \V^{u_N(\alpha)}\Big)} \leq 
        N^{-2d\alpha}\big(|F_x|+CN^2\log N^d\big)+CN^{-c_1d\alpha}.
    \end{equation*}
    Hence, noting that $|F|=\lfloor\lambda N^d\rfloor$ and $\sum_{x\in F}|F_x|\leq |F|^2\leq\big(\lambda N^d\big)^2$, we get
    \begin{equation*}
        \sum_{x\in F}\sum_{y\in F_x}{\mathbb{P}\Big(x,y \in \V^{u_N(\alpha)}\Big)} \leq
        \big(\lambda N^{d(1-\alpha)}\big)^2+ C\lambda N^{d(1-2\alpha)+2}\log N^d  +C\lambda N^{d(1-c_1\alpha)}. 
    \end{equation*}
    Plugging this bound into \eqref{eq: second moment approximated by RI}, we get 
    \eqref{eq: second moment estimate} uniformly in $\alpha \in [\alpha_0,1)$ and $\lambda \in [\lambda_0,1]$.
\end{proof}

We now turn to the proof of Lemma \ref{lem: late points concentration}. 
\begin{proof}[Proof of Lemma \ref{lem: late points concentration}]
    From now on, we always assume that $\alpha\in [\alpha_0,1)$ and $\lambda\in [\lambda_0,1)$. 
    It follows immediately from \eqref{eq: prob: single late point} that uniformly in $\alpha$ and $\lambda$, 
    \begin{equation}
    \label{eq: first moment of late points}
        \E |\L^\alpha_F| = |F|\P(0 \in \L^\alpha) = \big(1+o(1)\big)\lambda N^{d(1-\alpha)}.
    \end{equation}
    For the second moment of $|\L^\alpha_F|$, we have uniformly in $\alpha$ and $\lambda$, 
    \begin{equation}
    \label{eq: second moment of late points}
        \E |\L^\alpha_F|^2 = \sum_{x,y \in F}{\P\left(x,y \in \L^\alpha\right)} = \sum_{x \in F}{\P\left(x \in \L^\alpha\right)} + \sum_{x\neq y \in F}{\P\left(x,y \in \L^\alpha\right)} \\ \overset{\eqref{eq: prob: single late point}}{\underset{\eqref{eq: second moment estimate}}{\leq}} (1+o(1))\big(\lambda N^{d(1-\alpha)}\big)^2.
    \end{equation}
    With the estimates \eqref{eq: first moment of late points} and \eqref{eq: second moment of late points}, we can easily give 
    $\E \left(|\L^\alpha_F|-\lambda N^{d(1-\alpha)}\right)^2 = o\big((\lambda N^{d(1-\alpha)})^2\big)$, which implies \eqref{eq: late points concentration} by Markov's inequality.
\end{proof}

We conclude this section with the following remarks. 
\begin{rem}
\noindent
\begin{itemize}
    \item[1)] 
    We can get a better concentration result for $|\L^\alpha_F|$ by giving explicit estimates in \eqref{eq: prob: single late point} and \eqref{eq: second moment estimate} and then following the proof of Lemma \ref{lem: late points concentration}, cf.\ \cite[Lemma 4.2]{Bel13} for a similar concentration inequality but only for large $\alpha$. Notwithstanding, Lemma \ref{lem: late points concentration} is sufficient for our proof of large deviation bounds. 
    \item[2)] \label{rem: proof of upward} The estimate
    \eqref{eq: upward deviation} follows readily from Lemma \ref{lem: prob: late points}.
    Actually, the desired upper bound follows from Markov's inequality and \eqref{eq: prob: single late point}.
    For the lower bound, we first derive from \cite[(2.8)]{prevost2023phase} and \eqref{eq: prob: late points} that $\E|\L^{\gamma}|^2=\big(1+o(1)\big)N^{d(1-\gamma)}$.
    Combining this estimate with \eqref{eq: prob: single late point} and Cauchy-Schwarz inequality gives the desired lower bound.
\end{itemize}
\end{rem}

\section{Proof of the upper bound in Theorem \ref{thm: large deviation}}
\label{sec: proof of upper bound}
In this section, we prove the upper bound in Theorem \ref{thm: large deviation}, which can be paraphrased as:
\begin{equation}
\label{eq: upper bound}
 \frac{\log \P\big(U_{\gamma,N}\big)}{N^{d(1-\gamma)}} \leq -1+o(1)\quad \mbox{for $\gamma \in (\frac{d+2}{2d}, 1)$}; 
\end{equation}
see \eqref{eq: def: downward large deviation} for the definition of $U_{\gamma,N}$.

Pick $\delta \in (0,1)$ and assume $N$ is sufficiently large. 
Set $R_N=\lfloor N^\gamma \rfloor$ and $s_N=\lceil(1+\delta)R_N \rceil+1$.
Let $K_N = Q(0,\frac{N}{1+\delta}) \cap (s_N \Z)^d$ and write $\K_N$ for the image of $K_N$ under the canonical projection $\phi_N$ from $\Z^d$ to $\T_N$. 
It is easy to see that $\K_N$ is $(1+\delta)R_N$-well separated, and $Q(x, R_N)\subset Q_{\delta/2}$ for all $x \in K_N$. 

We now apply Corollary \ref{cor: stronger RW to independent RI coupling} with $\eta=\delta/2$, $R=R_N\in (N^\frac{1}{2}, \frac{N}{1+\delta}]$, $u=u_N(\gamma)$ and $\rho = \rho_N = (\log N)^{-2}$. Let $v_N=u_N(\gamma)(1+\rho_N)$. By \eqref{eq: stronger RW to independent RI}, we have
\begin{equation}
\label{eq: coupling of ub sets}
    \begin{aligned}
        \P\big(U_{\gamma,N}\big) &\leq \P\big(Q(\x,R_N)\subset X\big[0,u_N(\gamma)N^d\big] \text{ for all } \x \in \K_N\big) \\
        \overset{\eqref{eq: stronger RW to independent RI}}&{\leq} 
        \WP\big(Q(x,R_N)\subset \I^{(x),v_N} \text{ for all } x \in K_N\big) + CN^{4d}\exp\left(-c\rho_N^2\sqrt{u_N(\gamma)N^{d-2}}\right)
    \end{aligned}
\end{equation}
(recall that $c$ and $C$ depend only on $\delta$ and $d$), where we have used the fact that $u_N(\gamma)< CN$.
By the independence of $\I^{(x),v_N}$ ($x \in K_N$) and the translation invariance of random interlacements, 
\begin{equation}
\label{eq: independent events in upper bound}
    \WP\big(Q(x,R_N)\subset \I^{(x),v_N} \text{ for all } x \in K_N\big) =
    \mathbb{P}\big(Q(0,R_N)\subset \I^{v_N}\big)^{|K_N|}.
\end{equation}
By the definition of $K_N$, 
\begin{equation}
\label{eq: number of disjoint boxes}
    |K_N| \sim \bigg(\frac{N/(1+\delta)}{s_N}\bigg)^d 
    \sim \bigg(\frac{N^{1-\gamma}}{(1+\delta)^2}\bigg)^d, 
    \quad N\to\infty.
\end{equation}
Note that the exponent in the error term in \eqref{eq: coupling of ub sets} is smaller than $-N^{d(1-\gamma)}$ for large $N$ if $\gamma \in (\frac{d+2}{2d}, 1)$. Thus, we combine \eqref{eq: RI cover level}, \eqref{eq: coupling of ub sets}, \eqref{eq: independent events in upper bound} and \eqref{eq: number of disjoint boxes}, and obtain \eqref{eq: upper bound} by first taking $N\to \infty$ and then $\delta \to 0$.

\section{Proofs of lower bounds}
\label{sec: proofs of lower bounds}
In this section, we prove the lower bound in Theorem \ref{thm: large deviation} and will briefly discuss how Theorem \ref{thm: rough lower bound} can be proved in a similar fashion.
In Section \ref{sec: proof of the sharp lower bound}, we present in a precise fashion the four-stage strategy outlined in Section \ref{sec: outline and discussion} and prove the lower bound of Theorem \ref{thm: large deviation} assuming two key estimates Propositions \ref{prop: prob: lower bound on event E} and \ref{prop: loop insertion ineq} (note that the latter will be proved in Section \ref{sec: Construction of the mapping Psi and its properties}).  In Section \ref{sec: proof of the lower bound on event E prob}, we give the proof of Proposition \ref{prop: prob: lower bound on event E}. In Section \ref{sec: proof of the rough lower bound}, we sketch the proof of Theorem \ref{thm: rough lower bound}. 

\subsection{Proof of the lower bound in Theorem \ref{thm: large deviation}}
\label{sec: proof of the sharp lower bound}
We recall the parameters introduced in Section \ref{sec: outline and discussion}:
\begin{equation}
    \label{eq: def: parameters}
    \gamma\in (\gamma_0,1), \quad \delta \in (0,1], \quad K,\varepsilon,M>0 \text{ such that } K>6\varepsilon, \quad\g=\gamma-\frac{K}{u_N}, \quad \beta_N=\gamma-\frac{5\varepsilon}{u_N}
\end{equation}
(see \eqref{eq: sharp regime} and \eqref{eq:uNgamma} for the definitions of $\gamma_0$ and $u_N$). 
Recall also the time points dividing the stages:
\begin{equation}
    \label{eq: def: time points}
    T_1=\lfloor\g t_{\cov}\rfloor=\lfloor\gamma t_{\cov}-KN^d\rfloor, \; T_2=\lfloor\beta_N t_{\cov}\rfloor=\lfloor\gamma t_{\cov}-5\varepsilon N^d\rfloor, \; T_3=\lfloor\gamma t_{\cov}-\varepsilon N^d\rfloor. 
\end{equation}
Note that $\beta_N>\alpha_N>0$ for large $N$.

As discussed in Section \ref{sec: outline and discussion}, at \textbf{Stage 1} we require the numbers of $\g$-late points (formed by $X[0,T_1]$) in the bulk $\BQ_\delta$ (see \eqref{eq:Qdeltaboxes} for definition) and in the edge $\HH_\delta\overset{\text{def.}}{=}\T_N \setminus \BQ_\delta$ to be ``uniform''.  
To express this requirement formally, we introduce the following notions.

We start with the ``regularity'' for late points. For $\alpha\in(0,1)$, $F\subset \T_N$ and $\upsilon\in(0,1)$, we say $\L^\alpha$ is $(F,\upsilon)$-regular if  
\begin{equation*}
    (1-\upsilon)|F|N^{-d\alpha}\leq |\L^\alpha_F|\leq (1+\upsilon)|F|N^{-d\alpha}
\end{equation*}
(recall $\L^\alpha_F=F\cap \L^\alpha$). 
Lemma \ref{lem: late points concentration} ensures that with high probability, 
\begin{equation}
\label{eq: late points in the bulk and edge typical}
    \text{$\L^{\g}$ is $(\BQ_\delta, \varepsilon)$-regular and $(\HH_\delta,1/2)$-regular}.
\end{equation}

We also require that most of $\g$-late points in the edge to be separated in some sense, in order to control the total length of loops inserted at {Stage 3}. 
For $F\subset \T_N$, we define the {\it auxiliary graph} $G(F)=\big(F, E(F)\big)$ where vertices are the points in $F$, and 
\begin{equation}\label{eq:lNdef}
    \{x,y\}\in E(F)\mbox{ if }d_\infty(x,y) \leq 3l_N,\quad \mbox{ where }l_N\overset{\rm def.}{=}(\log N)^4.
\end{equation}
We denote by $\C(F)$ the collection of the connected components of $G(F)$. 
We label the points in $\T_N$ in a fixed order as $1,2,\ldots,N^d$ and let $\R(F)$ denote the set of (representative) vertices, i.e., the one with the smallest label in each connected component of $G(F)$. 
Specifically, we need the following separation requirement for $\g$-late points in the edge $\HH_\delta$:
\begin{equation}
\label{eq: late points in the edge separated}
    \big|\L^{\g}_{\HH_\delta}\big|-\big|\R\big(\L^{\g}_{\HH_\delta}\big)\big|\leq 
    N^{d(1-c_3\g)}, 
\end{equation}
where 
\begin{equation}
\label{eq: def: c_3,4}
    c_3=\frac{1}{2}(1+c_4)\in(1,c_4) \text{\quad with \quad} c_4=\frac{2g(0)}{g(0)+g(e_1)}>1
\end{equation}
(recall $g(\cdot)$ is the Green's function on $\Z^d$). 
Note that $N^{d(1-c_3\g)}$ is much smaller than the typical cardinality of $\L^{\g}_{\HH_\delta}$, which is of order $N^{d(1-\g)}$. 
We say 
\begin{equation}\label{eq:nice}
\mbox{$\L^{\g}$ is $(\delta,\varepsilon)$-nice if it satisfies \eqref{eq: late points in the bulk and edge typical} and \eqref{eq: late points in the edge separated}.}    
\end{equation}
We also write $\G_{\g, \varepsilon}$, which implicitly depends on $\delta$, for the collection of all $(\delta,\varepsilon)$-nice configurations of $\L^{\g}$. 
We will prove $\L^{\g}$ is $(\delta,\varepsilon)$-nice with high probability; see Lemma \ref{lem: prob: late points nice}. 

At \textbf{Stage 2}, conditioned on $\L^{\g}$ being in a $(\delta,\varepsilon)$-nice configuration, the walk is forced to cover $\L^{\g}_{\BQ_\delta}$, $\g$-late points in the bulk, between $T_1$ and $T_2$, whose probability we can bound from below by $\exp(-(1+o(1)) N^{d(1-\gamma)})$; 
see Proposition \ref{prop: prob: cover late points in a big box} for the precise statement.

At \textbf{Stage 3} 
where we manually modify the trace of random walk $X(T_2,T_3]$ to cover $\L^{\beta_N
}_{\HH_\delta}$ (i.e., $\beta_N$-late points in the edge), we make two requirements: 
\begin{itemize}
    \item [1)] that all points in the torus are within an $l_N$ distance from $X(T_2,T_3]$;
    \item [2)] there is an upper bound on the ``modified distance'' between points in $\L^{\g}_{\HH_\delta}$ and $X(T_2,T_3]$, 
which (roughly speaking) measures the total length of loops to be inserted at this stage.
\end{itemize}  
To express them more clearly, for $x \in \T_N$, we write
\begin{equation}
\label{eq: def: distance to random walk between T_2 and T_3}
    \bs{r}_x=
    d_\infty(x,X(T_2, T_3])
\end{equation}
(see \eqref{eq: def: time points} for the definitions of $T_2$ and $T_3$) and for $F\subset \T_N$, we define the \textit{modified total distance} of $F$ to $X(T_2,T_3]$ as
\begin{equation}
\label{eq: def: modified total distance}
    \mdist(F)=
    \sum_{x\in \R(F)}\big(\bs{r}_x+2\big)+3l_N\cdot (|F|-|\R(F)|). 
\end{equation}

To summarize,  we consider the following event, which depends on the parameters in \eqref{eq: def: parameters}: 
\begin{equation}
\label{eq: def: event E}
    E\overset{\rm def.}{=}E_1\cap E_2\cap E_3
\end{equation}
where
$$E_1\overset{\rm def.}{=}\left\{\L^{\g} \text{ is }(\delta,\varepsilon)\text{-nice}\right\},\qquad E_2\overset{\rm def.}{=}\left\{\L^{\g}_{\BQ_\delta} \subset X(T_1, T_2] \right\},$$
and
$$
E_3\overset{\rm def.}{=}\left\{
\bs{r}_x\leq l_N \text{ for all } x\in\T_N \text{\quad and \quad}
\mdist\big(\L^{\g}_{\HH_\delta}\big)\leq 
        M\cdot d\delta\cdot N^{d(1-\gamma)}\right\}.$$
These events correspond to requirements of the random walk needed in different stages.
We now present the lower bound on the probability of the event $E$.
\begin{prop}
\label{prop: prob: lower bound on event E}
For all $\delta, \eta>0$, there exist $K_0(\eta), \varepsilon_0(\eta)>0$ such that the following holds. 
For all $K>K_0$ and $0<\varepsilon<\varepsilon_0$, there exists $M=M(K,\varepsilon)$ such that, 
    \begin{equation}
    \label{eq: prob: E lower bound}
        \P(E)\geq \exp\Big(-(1+\eta)N^{d(1-\gamma)}\Big)
    \end{equation}
for sufficiently large $N$.
\end{prop}
Its proof and the proofs of all ingredients are deferred to
Section \ref{sec: proof of the lower bound on event E prob}. 

\smallskip

We now give the second necessary ingredient in the proof of lower bound in Theorem \ref{thm: large deviation}, which is a control on the probability cost of loop insertion at {\bf Stage 3}. 
  We define $\L_{F}^{\alpha}(\omega)$, $\bs{r}_{x}(\omega)$ and $\mdist(F(\omega))$ similarly for the finite path $\omega \in \Omega_{T_3}$ (see below \eqref{eq: random walk paths} for definition); cf.\ below \eqref{eq: prob: single late point}, \eqref{eq: def: distance to random walk between T_2 and T_3} and \eqref{eq: def: modified total distance} respectively. 
We also define the following collection of good paths associated with a set $F\subset \T_N$. Note that the event $E$ (see \eqref{eq: def: event E} for definition) corresponds to this path collection with $F=\HH_\delta$. However, we write the definition in a more general form to also adapt to the proof of Theorem \ref{thm: rough lower bound}; see Section \ref{sec: proof of the rough lower bound} for details.

\begin{definition}
\label{def: good paths prop}
    For $F \subset \T_N$, we say a path collection $\CA \subset \Omega_{T_3}$ is $(F,\varepsilon,M_0)$-good if for all $\omega \in \CA$,
    \begin{subequations}
    \label{good paths necessary properties}
        \begin{gather}
        \label{covering complement of F}
            \T_N \setminus F \subset \omega[0,T_3];\\
        \label{no too large late island}
            \bs{r}_{x}(\omega) \leq l_N \text{ for all } x \in  \L_{F}^{\beta_N}(\omega);\\
        \label{no too much cost for extra loops}
            2d\cdot\mdist\big(\L_{F}^{\beta_N}(\omega)\big)\leq J(F,M_0),
        \end{gather}
    \end{subequations}
    where 
\begin{equation}
    \label{cost of inserting loops}
        J(F,M_0)\overset{\rm def.}{=} 2d M_0\cdot|F|N^{-d\gamma}.
    \end{equation}
    For an event $A\in \CF_{T_3}$, we also say $A$ is $(F,\varepsilon,M_0)$-good if the path collection $\CA=\pi_{T_3}(A)$ (see \eqref{eq: paths correspond to event} for definition) is $(F,\varepsilon,M_0)$-good.
\end{definition}
We now quickly comment on the motivation of such a definition before stating the control (i.e.\ Proposition \ref{prop: loop insertion ineq}). In Section \ref{sec: Construction of the mapping Psi and its properties}, we will use \eqref{no too much cost for extra loops} to control the total length of loops inserted at Stage 3, and \eqref{no too large late island} to ensure that the loops inserted are mutually disjoint, so that the original trajectory can be recovered deterministically from the trajectory after insertion.

\begin{prop}
    \label{prop: loop insertion ineq}
    For any $F \subset \T_N$ and $M_0>0$ such that
    \begin{equation}
    \label{eq: cost of inserting loops not large}
        J(F,M_0)+2 \leq \varepsilon N^d
    \end{equation}
 holds, and for any $(F,\varepsilon,M_0)$-good event $A$, we have
    \begin{equation}
    \label{eq: loop insertion ineq}
        \P\big(U_{\gamma,N}\big) \geq \P(A) \cdot (2d)^{-J(F, M_0)}.
    \end{equation}
\end{prop}

Since the proof of Proposition \ref{prop: loop insertion ineq} relies solely on combinatorial arguments and is relatively independent of other parts of this section, we defer the proof to Section \ref{sec: Construction of the mapping Psi and its properties}.
We now provide the proof of the lower bound in Theorem \ref{thm: large deviation} assuming Propositions \ref{prop: prob: lower bound on event E} and \ref{prop: loop insertion ineq}.

\begin{proof}[Proof of the lower bound in Theorem \ref{thm: large deviation}]
    Pick $\eta, \delta>0$. By Proposition \ref{prop: prob: lower bound on event E}, there exist $K,\varepsilon,M>0$ such that \eqref{eq: prob: E lower bound} holds for large $N$. From now on, we fix $K$, $\varepsilon$ and $M$. 
    We claim that
    \begin{equation}
    \label{eq: E good}
        E \text{ is an $(\HH_\delta,\varepsilon,2M)$-good event, i.e., } \CE\overset{\text{def.}}{=}\pi_{T_3}(E) \text{ is $(\HH_\delta,\varepsilon,2M)$-good}.
    \end{equation}
    By the definition of the event $E$ (see \eqref{eq: def: event E}), it is clear that \eqref{covering complement of F} holds and 
    \begin{equation}
    \label{eq: path-distance domination}
        \bs{r}_{x}(\omega) \leq l_N \text{ for all } x \in \T_N, 
    \end{equation}
    which implies \eqref{no too large late island}. 
    For $\omega \in \CE$, 
    \begin{equation}
    \label{eq: mdist bound}
        \mdist\big(\L^{\g}_{\HH_\delta}(\omega)\big)\leq 
        M\cdot d\delta\cdot N^{d(1-\gamma)}\leq 2M\cdot|\HH_\delta|N^{-d\gamma}.
    \end{equation}
    Note that by \eqref{eq: path-distance domination}, $\bs{r}_x(\omega)+2\leq 3l_N$ for all $x\in \L^{\g}_{\HH_\delta}(\omega)$, which implies that $\mdist(F)$ decreases when we remove a point from $F \subset \L^{\g}_{\HH_\delta}(\omega)$. 
    Since $\L^{\beta_N}_{\HH_\delta}(\omega)\subset\L^{\g}_{\HH_\delta}(\omega)$, 
    we get
    \begin{equation}
    \label{eq: control mdist}
        \mdist\big(\L^{\beta_N}_{\HH_\delta}(\omega)\big)\leq \mdist\big(\L^{\g}_{\HH_\delta}(\omega)\big).
    \end{equation}
    Combining \eqref{eq: mdist bound} and \eqref{eq: control mdist} gives \eqref{eq: E good}. 
    We now take $F=\HH_\delta$, $M_0=2M$ and $N$ large such that \eqref{eq: cost of inserting loops not large} holds. Then we can apply Proposition \ref{prop: loop insertion ineq} with $A=E$ and obtain the lower bound in Theorem \ref{thm: large deviation} with the help of \eqref{eq: prob: E lower bound} and by sending $\eta$ and $\delta$ to 0. 
\end{proof}

\subsection{Proof of Proposition \ref{prop: prob: lower bound on event E}}
\label{sec: proof of the lower bound on event E prob}
To prove Proposition \ref{prop: prob: lower bound on event E}, we deal with the events $E_1$, $E_2$ and $E_3$ (see below \eqref{eq: def: event E} for definition) separately in Lemma \ref{lem: prob: late points nice}, Propositions \ref{prop: prob: cover late points in a big box} and \ref{prop: bounds on distances to random walk} below.  
Proposition \ref{prop: prob: lower bound on event E} then follows directly from the Markov property.
Note that for the clarity of presentation, Propositions \ref{prop: prob: cover late points in a big box} and \ref{prop: bounds on distances to random walk} are written for random walks after a time shift of $T_1$ and $T_2$ respectively rather than the original walk. We also mention that the parameters in this subsection implicitly satisfy \eqref{eq: def: parameters}. 

\smallskip

We start with $E_1=\{\L^{\g} \text{ is $(\delta,\varepsilon)$-nice}\}$.
\begin{lemma}
    \label{lem: prob: late points nice}
    $\displaystyle\lim_{N\to\infty} \P\big(E_1\big) = 1.$
\end{lemma}
\begin{proof}
    Note that for large $N$, $\g$ is bounded below from 0. By Lemma \ref{lem: late points concentration}, 
    \begin{equation*}
        \lim_{N\to \infty}\P\big(\text{$\L^{\g}$ is $(\BQ_\delta, \varepsilon)$-regular and $(\HH_\delta,1/2)$-regular}\big) = 1.
    \end{equation*}
    Recall the notation on the auxiliary graph $G(F), F\subset \T_N$ below \eqref{eq: late points in the bulk and edge typical}. 
    It is a basic graph property that $|E(F)|\geq |F|-|\C(F)|$. Noting $|\C(F)|=|\R(F)|$ and taking $F=\L^{\g}_{\HH_\delta}$, we get
    $\big|E\big(\L^{\g}_{\HH_\delta}\big)\big|\geq \big|\L^{\alpha_N}_{\HH_\delta}\big|-\big|\R\big(\L^{\alpha_N}_{\HH_\delta}\big)\big|$. 
    Thus, it suffices to prove
    \begin{equation}
    \label{eq: prob: pairs relatively small}
    \lim_{N\to\infty}
    \P\Big(\big|E\big(\L^{\g}_{\HH_\delta}\big)\big| \leq 
    N^{d(1-c_3\cdot\g)}\Big) = 1
    \end{equation}
    (see \eqref{eq: def: c_3,4} for the definitions of $c_3$ and $c_4$ and note that $c_3<c_4$). 
    By the translation invariance of $\L^{\g}$, for $x\neq y \in \T_N$,
    \begin{equation*}
        \P(x,y \in \L^{\g}) = \P(0,x-y \in \L^{\g}) \overset{\eqref{eq: prob: late points}}{\underset{\eqref{eq: green function comparison}}{\leq}} 
        CN^{-d\g\cdot \frac{2g(0)}{g(0)+g(e_1)}}\overset{\eqref{eq: def: c_3,4}}{=}
        CN^{-c_4\cdot d\g}.
    \end{equation*}
    Hence, 
    \begin{equation}
    \label{eq: expectation of pairs}
\E\big[\big|E\big(\L^{\g}_{\HH_\delta}\big)\big|\big] \leq
        \sum_{x\neq y\in \T_N, |x-y|\leq 3l_N}\P(x,y \in \L^{\g})\leq 
        CN^{d(1-c_4\cdot\g)}l_N^d, 
    \end{equation}
    which implies \eqref{eq: prob: pairs relatively small} by Markov's inequality.
\end{proof}

We now turn to $E_2$. 
For $\F\subset \T_N$, we let $\F_{\bulk}=\F\cap \BQ_\delta$ and $\F_{\edge}=\F\cap \HH_\delta$. We write $F_{\bulk}$ for the image of $\F_{\bulk}$ under the canonical map from $\T_N$ to $Q(0,N)$. Recall also $\G_{\alpha_N,\varepsilon}$ is the collection of $(\delta,\varepsilon)$-nice configurations of $\L^{\alpha_N}$. 
\begin{prop}
    \label{prop: prob: cover late points in a big box}
    For any $\F\in \G_{\g, \varepsilon}$ and $x\in \T_N$, 
    \begin{equation}
    \label{eq: prob: cover late points in a big box}
        \P_x\big(E_2'\big) \geq
        \exp\left(-\xi(K,\varepsilon,N)\cdot N^{d(1-\gamma)}\right), 
    \end{equation}
    where 
    $$
    E_2'\overset{\rm def.}{=}\Big\{\F_{\bulk} \subset X\big[0,(K-5\varepsilon)N^d-1\big]\Big\},
    $$
    and $\xi(K,\varepsilon, N)>0$ satisfies $\displaystyle\lim_{K\to \infty, \varepsilon\to 0}\lim_{N\to\infty} \xi(K,\varepsilon, N)=1$.
\end{prop}
We will return to its proof after we prove Proposition \ref{prop: prob: lower bound on event E}.

Finally, we deal with $E_3$. With the definition of $\bs{r}_\cdot$ and $\mdist(\cdot)$ in \eqref{eq: def: distance to random walk between T_2 and T_3} and \eqref{eq: def: modified total distance} in mind, for $x \in \T_N$ and $\upsilon>0$, we write
\begin{equation}
\label{eq: def: distance to random walk at the beginning}
    \bs{r}_{x}^{\upsilon}=d_{\infty}\big(x, X\big[0,\upsilon N^d\big]\big)
\end{equation}
and 
\begin{equation}
\label{eq: def: modified total distance at the beginning}
    \mdist'(F,\upsilon)=\sum_{x\in \R(F)}(\bs{r}_x^{\upsilon}+2)+3l_N\cdot (|F|-|\R(F)|).
\end{equation}
\begin{prop}
    \label{prop: bounds on distances to random walk}
    There exists a constant $M=M(K,\varepsilon)>0$ such that
    \begin{equation}
    \label{eq: bounds on distances to random walk}
        \lim_{N\to \infty}\ 
        \inf_{y\in \T_N, \F\in \G_{\g, \varepsilon}}
        \P_y\left(E_3' \right) = 1,
    \end{equation}
    where
    $$
    E_3'\overset{\rm def.}{=}\left\{\bs{r}_{x}^{3\varepsilon}\leq l_N \text{ for all }x\in\T_N \right\}\bigcap
\left\{        \mdist'\big(\F_{\edge}, 3\varepsilon\big)\leq 
        M\cdot d\delta\cdot N^{d(1-\gamma)}\right\}.
    $$
\end{prop}
Again, we postpone its proof till we prove Proposition \ref{prop: prob: lower bound on event E}.
\begin{proof}[Proof of Proposition \ref{prop: prob: lower bound on event E} assuming Propositions \ref{prop: prob: cover late points in a big box} and \ref{prop: bounds on distances to random walk}]
    We apply the Markov property of the random walk at time $T_1, T_2$ and $T_3$ (see \eqref{eq: def: time points} for definitions) combined with Lemma \ref{lem: prob: late points nice}, Propositions \ref{prop: prob: cover late points in a big box} and \ref{prop: bounds on distances to random walk}, and immediately get \eqref{eq: prob: E lower bound} by taking $K$ sufficiently large and $\varepsilon$ sufficiently small. 
    In particular, we use the weaker version of Proposition \ref{prop: bounds on distances to random walk} with $d_{\infty}\big(x, X\big[0,4\varepsilon N^d-1\big]\big)$ in place of $\bs{r}_x^{3\varepsilon}$ for the time scale 
    $(T_2, T_3]$. 
\end{proof}

We now use Proposition \ref{prop: macroscopic coupling} and Lemma \ref{lem: FKG inequality} to prove Proposition \ref{prop: prob: cover late points in a big box}. To apply Proposition \ref{prop: macroscopic coupling}, we need the starting point of the random walk to be uniform on the torus, and we will use a mixing argument to do this. 
\begin{proof}[Proof of Proposition \ref{prop: prob: cover late points in a big box}]
    Let $T$ be a random time taking the value $\lfloor\varepsilon N^d/2\rfloor -1$ or $\lfloor\varepsilon N^d/2\rfloor$ with equal probability. By \cite[Proposition 4.7 and Theorem 5.6]{LP17}, there exists a coupling $\mathsf{Q}$ between $(X_n)_{n\geq 0}$ with law $\P_x$ and $(Z_n)_{n\geq 0}$ with law $\P$, under which $(X_{n+T})_{n\geq 0}$ coincides with $(Z_n)_{n\geq 0}$ with $\mathsf{Q}$-probability at least $1-\exp\big(-c\varepsilon N^{d-2}\big)$, which implies
    \begin{equation}
    \label{eq: mixing coupling}
        \P_x\big(\F_{\bulk} \subset X\big[0,(K-5\varepsilon)N^d-1\big]\big) \geq \P\big(\F_{\bulk} \subset Z\big[0,(K-6\varepsilon) N^d\big]\big)-\exp\big(-c\varepsilon N^{d-2}\big).
    \end{equation}
    We then apply Proposition \ref{prop: macroscopic coupling} with $u=K-6\varepsilon$ and $\rho=\rho_N=(\log N)^{-1}$, and get
    \begin{equation}
    \label{eq: cover a box by RW to RI}
    \begin{split}
        &\P\big(\F_{\bulk} \subset Z\big[0,(K-6\varepsilon) N^d\big]\big)\\ \geq \;&
        \mathbb{P}\Big(F_{\bulk} \subset \I^{(K-6\varepsilon)(1-\rho_N)}\Big)-
        CKN^{3d}\exp\Big(-c\rho_N\sqrt{(K-6\varepsilon)N^{d-2}}\Big)
    \end{split}
    \end{equation}
    (recall that $c$ and $C$ here only depend on $\delta$ and $d$). 
    By Lemma \ref{lem: FKG inequality}, 
    \begin{equation*}
    \label{eq: FKG inequality}
        \mathbb{P}\Big(F_{\bulk} \subset \I^{(K-6\varepsilon)(1-\rho_N)}\Big) \geq 
        \big(1-\exp[-(K-6\varepsilon)(1-\rho_N)/g(0)]\big)^{|F_{\bulk}|}.
    \end{equation*}
    Since $\F\in \G_{\g,\varepsilon}$, 
    \begin{equation*}
    \label{eq: F is a nice configuration}
        |F_{\bulk}|=|\F_{\bulk}|\leq (1+\varepsilon)|Q_\delta|N^{-d\g}
        \overset{\eqref{eq: def: parameters}}{\leq}
        (1+\varepsilon)\exp[K/g(0)]\cdot N^{d(1-\gamma)}. 
    \end{equation*}
    Combining the above four inequalities, and noting that the exponents in the error terms in \eqref{eq: mixing coupling} and \eqref{eq: cover a box by RW to RI} are smaller than $-N^{d(1-\gamma)}$ for large $N$ if $\gamma\in (\frac{d+2}{2d},1)$, we obtain the bound \eqref{eq: prob: cover late points in a big box}
    with 
    \begin{equation*}
        \xi(K,\varepsilon, N)=-(1+\varepsilon)\cdot 
        \exp[K/g(0)]\cdot
        \log\big(1-\exp[-(K-6\varepsilon)(1-\rho_N)/g(0)]\big).
    \end{equation*}
    It is straightforward to check that $\displaystyle\lim_{K\to \infty, \varepsilon\to 0}\lim_{N\to\infty} \xi(K,\varepsilon, N)=1$. 
\end{proof}

In the rest of this subsection, we use Propositions \ref{prop: RW/RI to independent RI coupling} and \ref{prop: macroscopic coupling} to control the distances of points in the torus to the random walk and prove Proposition \ref{prop: bounds on distances to random walk}. 
To use the coupling results, we apply the same mixing argument as in the proof of Proposition \ref{prop: prob: cover late points in a big box}. Then, Proposition \ref{prop: bounds on distances to random walk} is reduced to the following Proposition \ref{prop: uniform case: bounds on distances to random walk} which has a similar claim but is for the random walk starting from the uniform distribution on the torus. 
We will turn to the proof of Proposition \ref{prop: uniform case: bounds on distances to random walk} after showing Proposition \ref{prop: bounds on distances to random walk}. 
For clarity, we define two events (which depend on the parameters in \eqref{eq: def: parameters}):
\begin{equation}
\label{eq: def: E_1,2}
    V_1=\big\{\bs{r}_{x}^{2\varepsilon}\leq l_N \text{ for all }x\in\T_N\big\}, \quad 
    V_2(\F,M)=\left\{\mdist'\big(\F_{\edge}, 2\varepsilon\big)\leq 
    M\cdot d\delta\cdot N^{d(1-\gamma)}\right\}.
\end{equation}

\begin{prop}
\label{prop: uniform case: bounds on distances to random walk}
    There exists a constant $M=M(K,\varepsilon)>0$ such that
    \begin{equation}
    \label{eq: uniform case: bounds on distances to random walk}
        \lim_{N\to \infty}\inf_{\F\in \G_{\g, \varepsilon}}
        \P[V_1\cap V_2(\F,M)] = 1.
    \end{equation}
\end{prop}

\begin{proof}[Proof of Proposition \ref{prop: bounds on distances to random walk} assuming \Cref{prop: uniform case: bounds on distances to random walk}]
    The coupling $\mathsf{Q}$ of random walks $(X_n)_{n\geq 0}$ with law $\P_x$ and $(Z_n)_{n\geq 0}$ with law $\P$ in the proof of Proposition \ref{prop: prob: cover late points in a big box} implies
    \begin{equation}
    \label{eq: same mixing coupling}
        \mathsf{Q}\big(Z\big[0,2\varepsilon N^d\big]\subset X\big[0,3\varepsilon N^d\big]\big) \geq 
        1-\exp\big(-c\varepsilon N^{d-2}\big).
    \end{equation}
    Combining \eqref{eq: uniform case: bounds on distances to random walk} and \eqref{eq: same mixing coupling} gives \eqref{eq: bounds on distances to random walk}.
\end{proof}

Finally, we prove Proposition \ref{prop: uniform case: bounds on distances to random walk}.
Proposition \ref{prop: macroscopic coupling} indicates that to control $\bs{r}_{\x}^{2\varepsilon}$, $\x\in \T_N$ it suffices to control 
\begin{equation} \label{eq: def: distance to random interlacement}
    r_x^{\varepsilon} \overset{\text{def.}}{=} d_{\infty}(x,\I^\varepsilon), \quad x\in \Z^d, 
\end{equation}
which is smaller than $l_N=(\log N)^4$ with high probability. 
Noting that for $\F \subset \T_N$, the $l^\infty$-distance of distinct points in $\R\big(\F_{\edge}\big)$ is greater than $3l_N$, 
we can use Proposition \ref{prop: RW/RI to independent RI coupling} to construct an independent family of random variables with the same law as $r_0^{\varepsilon}$, 
which dominate the distances $\bs{r}_{\x}^{2\varepsilon}, \x\in \R\big(\F_{\edge}\big)$ with high probability. 
Then, we can easily control the total distance $\mdist'\big(\F_{\edge}, 2\varepsilon\big)$ for $\F\in \G_{\g,\varepsilon}$ (recall that $\G_{\g,\varepsilon}$ is the collection of $(\delta,\varepsilon)$-nice configurations of $\L^{\g}$, see \eqref{eq:nice} and below for the definition). 

\begin{proof}[Proof of Proposition \ref{prop: uniform case: bounds on distances to random walk}]We apply Proposition \ref{prop: macroscopic coupling} picking $u,\delta,\rho$ therein as $2\varepsilon, \frac{1}{2}, \frac{1}{2}$, respectively. 
    By \eqref{eq: macroscopic coupling}, 
    \begin{equation}
    \label{eq: distance to RW to RI}
        \P\big(\bs{r}_{\bs{0}}^{2\varepsilon}> l_N\big)\leq
        \mathbb{P}(r_0^{\varepsilon}> l_N)+
        C\varepsilon N^{3d}\exp\left(-c\sqrt{\varepsilon N^{d-2}}\right).
    \end{equation}
    For $r>1$, we have the following tail bound for $r_0^\varepsilon$:
    \begin{equation}
    \label{eq: tail bound of the distance to RI}
        \mathbb{P}(r_0^{\varepsilon}>r)\leq \mathbb{P}(Q(0,r)\subset \V^\varepsilon)\overset{\eqref{eq: prob: interlacements vacant set}}{\underset{ \eqref{eq: capacity bound}}{\leq}}\exp\big(-c\varepsilon r^{d-2}\big), 
    \end{equation}
    which directly implies that $r_0^{\varepsilon}$ has finite moments. 
    By the translation invariance of the random walk under $\P$ and a union bound, we get
    \begin{equation}
    \label{eq: domination of the distance}
        \P(V_1^c) \leq
        \err_1 \overset{\text{def.}}{=}
        N^d\cdot\P\big(\bs{r}_{\bs{0}}^{2\varepsilon}> l_N\big) 
        \xrightarrow{\eqref{eq: distance to RW to RI}, \eqref{eq: tail bound of the distance to RI}} 0, \text{ as } N\to\infty.
    \end{equation}
    
    We now assume that $\F \in \G_{\g,\varepsilon}$. By the definition of the auxiliary graph (see below \eqref{eq: late points in the bulk and edge typical}), for large $N$, $\R(\F_{\edge})$ is $\frac{35}{12}l_N$-well separated. 
    We then use Proposition \ref{prop: RW/RI to independent RI coupling} picking $\delta, R, \F, u, \rho$ therein as $1/6, 5l_N/2, \R\big(\F_{\edge}\big), 2\varepsilon, 1/2$. 
    Let $\WP$ stand for the extended probability measure and $\big(\I^{(x),\varepsilon}\big)_{\x \in \R(\F_{\edge})}$ for the constructed independent random interlacements with the same law as $\I^{\varepsilon}$. We now define the event
    \begin{equation*}
        V_3 = \left\{\I^{(x),\varepsilon}\cap Q(x, 5l_N/2) \subset X\big[0,2\varepsilon N^d\big]\cap Q(\x, 5l_N/2) \text{ for all } \x\in \R\big(\F_{\edge}\big)\right\}.
    \end{equation*}
    By \eqref{eq: RW to independent RI}, 
    \begin{equation}
    \label{eq: RI contained in RW in small boxes}
        \WP(V_3^c) \leq \err_2 \overset{\text{def.}}{=} C\varepsilon N^{4d}\exp\left(-c\sqrt{\varepsilon\cdot l_N^{d-2}}\right)\to 0, \text{ as }N\to\infty. 
    \end{equation}
    We write $r^{\varepsilon}(x)=d_\infty\left(x, \I^{(x),\varepsilon}\right)$ for $\x \in \R\big(\F_{\edge}\big)$. It is clear that 
    \begin{equation}
    \label{eq: distances to RI are i.i.d.}
        \text{$r^{\varepsilon}(x)$, $\x\in \R\big(\F_{\edge}\big)$, are i.i.d.\ with the same law as $r_0^{\varepsilon}$.}
    \end{equation}
    Let $V_4=\left\{\bs{r}_{\x}^{2\varepsilon}\leq r^{\varepsilon}(x) \text{ for all } \x\in \R\big(\F_{\edge}\big)\right\}$. 
    Noting that $V_1\cap V_3$ implies $V_4$ for large $N$, we have
    \begin{equation}
    \label{eq: E_4 high prob}
        \WP(V_4^c)\overset{\eqref{eq: domination of the distance}}{\underset{\eqref{eq: RI contained in RW in small boxes}}{\leq}} \err_1+\err_2. 
    \end{equation}
    Moreover, on $V_4$,
    \begin{equation}
    \label{eq: total distances to RW bounded by to RI}
        \sum_{\x\in \R(\F_{\edge})}\bs{r}_{\x}^{2\varepsilon}\leq 
        \sum_{\x\in \R(\F_{\edge})}r^{\varepsilon}(x). 
    \end{equation}

    Since $\F \in \G_{\g, \varepsilon}$ (see below \eqref{eq:nice} for the definition) and $|H_\delta|\sim \big(1-(1-\delta)^d\big) N^d\sim d\delta N^d$ as $N\to\infty$ and $\delta\to 0$, we have the following bounds on $\big|\F_{\edge}\big|$ and $\big|\F_{\edge}\big|-\big|\R\big(\F_{\edge}\big)\big|$:
    \begin{equation*}
        c(K)d\delta\cdot N^{d(1-\gamma)} \overset{\eqref{eq: def: parameters}}{\leq}
        \frac{1}{2}|H_\delta|N^{-d\g} \leq
        \big|\F_{\edge}\big| \leq 
        \frac{3}{2}|H_\delta|N^{-d\g} \overset{\eqref{eq: def: parameters}}{\leq}
        C(K)d\delta\cdot N^{d(1-\gamma)}
    \end{equation*}
    and 
    \begin{equation}
    \label{eq: small number of badly separated late points}
        \big|\F_{\edge}\big|-\big|\R\big(\F_{\edge}\big)\big|\leq 
        N^{d(1-c_3\cdot \g)}=C(K)N^{d(1-c_3\gamma)}, 
    \end{equation}
    where $c_3>1$ (see \eqref{eq: def: c_3,4} for definition), which imply
    \begin{equation}
    \label{eq: large number of well separated late points}
        c(K,\delta)N^{d(1-\gamma)}\leq\big|\R\big(\F_{\edge}\big)\big|\leq C(K)d\delta\cdot N^{d(1-\gamma)}.
    \end{equation}
    Since $r_0^{\varepsilon}$ has finite moments (see below \eqref{eq: tail bound of the distance to RI}), by \eqref{eq: distances to RI are i.i.d.}, \eqref{eq: large number of well separated late points} and Chebyshev's inequality, we have
    \begin{equation}
    \label{eq: total distances to RI}
        \WP\Big(\sum_{\x\in \R(\F_{\edge})}r^{\varepsilon}(x)> 2\big|\R\big(\F_{\edge}\big)\big|\cdot \mathbb{E}r_0^{\varepsilon}\Big) \leq\frac{C(\varepsilon)}{\big|\R\big(\F_{\edge}\big)\big|}
         \leq
        \err_3\overset{\text{def.}}{=}C(K,\varepsilon,\delta)N^{-d(1-\gamma)}. 
    \end{equation}
    Recall $V_2(F,M)$ defined in \eqref{eq: def: E_1,2} and the fact that
    $$\mdist'(\F_{\edge}, 2\varepsilon)
    =\sum_{\x\in \R(\F_{\edge})}\big(\bs{r}_{\x}^{2\varepsilon}+2\big)
    +3l_N\cdot\big(\big|\F_{\edge}\big|-\big|\R\big(\F_{\edge}\big)\big|\big).$$ 
    The claims \eqref{eq: total distances to RW bounded by to RI}, \eqref{eq: small number of badly separated late points}, \eqref{eq: large number of well separated late points} and \eqref{eq: total distances to RI} together imply that there exists $M=M(K,\varepsilon)>0$ such that
    \begin{equation*}
        \WP[V_2(\F,M)]\geq 
        1-\left(\WP(V_4^c)+\err_3\right) \overset{\eqref{eq: E_4 high prob}}{\geq}
        1-\textstyle\sum_{i=1}^3 \err_i \to 1, \text{ as } N\to\infty, 
    \end{equation*}
    where the bound $1-\sum_{i=1}^3 \err_i$ does not depend on the choice of $\F \in \G_{\g,\varepsilon}$. 
    We conclude by combining this bound with \eqref{eq: domination of the distance}.
\end{proof}

\subsection{Sketch of the proof of Theorem \ref{thm: rough lower bound}}
\label{sec: proof of the rough lower bound}
We explain in this subsection how to modify the arguments in Sections \ref{sec: proof of the sharp lower bound} and \ref{sec: proof of the lower bound on event E prob} to prove Theorem \ref{thm: rough lower bound}. 
Note that in this subsection $\gamma\in(0,1)$, which is different from the range of $\gamma$ in Theorem \ref{thm: large deviation}. 
Recall also the requirements on $\varepsilon$ and $\beta_N$ in \eqref{eq: def: parameters}. Here we simply take $\varepsilon=1$. We define an analogous version of the event $E$ (as defined in \eqref{eq: def: event E}):
\begin{equation}\label{eq: def: event E for rlb}
    E'=E'(M')=\bigg\{
    \begin{array}{c}
    \L^{\beta_N} \text{ is $(1,1/2)$-nice}, \\
    \bs{r}_x\leq l_N \text{ for all } x\in\T_N \text{\quad and \quad}
    \mdist(\L^{\ai})\leq 
    M'\cdot N^{d(1-\gamma)}
    \end{array}
    \bigg\}.
\end{equation}
We can imitate the proof of Lemma \ref{lem: prob: late points nice} and Proposition \ref{prop: bounds on distances to random walk}, and obtain the following result. We omit the proof.  
\begin{lemma}
\label{lem: event E' prob}
    There exists a constant $M'>0$ such that 
    \begin{equation*}
    \label{eq: event E' prob}
        \lim_{N \to \infty}\P(E')=1.
    \end{equation*}
\end{lemma}

We now sketch the proof of Theorem \ref{thm: rough lower bound}. 

\begin{proof}[Sketch of the proof of Theorem \ref{thm: rough lower bound}]
    By Lemma \ref{lem: event E' prob}, we can take $M'>0$ such that $\P(E') \geq \frac{1}{2}$ for large $N$.
    Similar to \eqref{eq: E good}, we can prove
    \begin{equation*}
        E' \text{ is $(\T_N,1,M')$-good}.
    \end{equation*}
    We now take $F=\T_N$, $M_0=M'$ and $N$ large enough such that \eqref{cost of inserting loops} holds. 
    Applying Proposition \ref{prop: loop insertion ineq} with $A=E'$, we obtain that for sufficiently large $N$, 
    \begin{equation*}
        \P(U_{\gamma,N})\geq \frac{1}{2}\exp\big(-C'M'\cdot N^{d(1-\gamma)}\big).
    \end{equation*}
    We hence conclude by taking some $C>C'M'$.
\end{proof}

\section{Covering via loop insertion}
\label{sec: Construction of the mapping Psi and its properties}
In this section, we give the proof of Proposition \ref{prop: loop insertion ineq}. Following the approach outlined in Section \ref{sec: outline and discussion}, for each path $\omega$ in the $(\F,\varepsilon,M_0)$-good path collection $\CA$, we construct in Proposition \ref{prop: Psi properties} a path collection $\Psi(\omega)$ such that each path in $\Psi(\omega)$ has length $T_4=\lfloor \gamma t_{\cov} \rfloor$ and visits every vertex in $\T_N$ at least once (see \eqref{eq: covering all vertices Psi} below). 
Consequently, we have $\bigcup_{\omega \in \CA}\Psi(\omega) \subset \mathcal{U}_{\gamma,N} \overset{\mathrm{def.}}{=} \pi_{T_4}(U_{\gamma,N})$ (see \eqref{eq: def: downward large deviation} and \eqref{eq: paths correspond to event} for the definitions of $U_{\gamma, N}$ and $\pi_k(\cdot)$).
To establish a lower bound for $|\bigcup_{\omega \in \CA}\Psi(\omega)|$, we also require two key conditions:
\begin{enumerate}
    \item $\Psi(\omega)$ contains sufficiently many paths for any $\omega \in \CA$;
    \item $(\Psi(\omega))_{\omega \in \CA}$ are mutually disjoint.
\end{enumerate}
These requirements correspond precisely to \eqref{eq: enough covering paths Psi} and \eqref{eq: injective mapping Psi} in the following proposition (recall $T_3=\lfloor \gamma t_{\cov}-\varepsilon N^d \rfloor$). 

Recall that $\Omega_k$ stand for the collection of random walk paths of length $k$; see below \eqref{eq: random walk paths}. For any $\omega \in \Omega_k$ and $0 \leq t_1 \leq t_2 \leq k$, let $\omega[t_1,t_2]=\{\omega(j):  t_1 \leq j \leq t_2 \}$, where $\omega(j)$ is the $(j+1)$-th element in $\omega$. The sub-path
$\omega(t_1,t_2]$ can be defined similarly.

\begin{prop}
\label{prop: Psi properties} 
    For any $\F \subset \T_N$ and $M_0>0$ such that 
    \eqref{eq: cost of inserting loops not large} holds, and
    for any $(\F,\varepsilon,M_0)$-good path collection $\CA$, there exists a mapping $\Psi$ from $\CA$ to $2^{\Omega_{T_4}}$ such that:
    \begin{subequations}
    \label{eq: Psi properties}
        \begin{align}
            & \text{For any } \omega \in \CA \text{ and } \widetilde{\omega} \in \Psi(\omega) \text{, we have } \T_N \subset \widetilde{\omega}[0,T_4] \label{eq: covering all vertices Psi};\\ 
            &\text{For any } \omega \in \CA\text{, we have }|\Psi(\omega)| \geq (2d)^{T_4-T_3 - J(\F,M_0)};   \label{eq: enough covering paths Psi}\\
            &\text{For any } \omega,\omega' \in \CA \text{ and } \omega \neq \omega' \text{, we have }\Psi(\omega) \cap \Psi(\omega') = \emptyset \label{eq: injective mapping Psi},
        \end{align}
    \end{subequations}
    where $J(\F,M_0)$ is defined in \eqref{cost of inserting loops}.
\end{prop}

\begin{proof}[Proof of Proposition \ref{prop: loop insertion ineq} 
 assuming Proposition \ref{prop: Psi properties}]
    Since $A$ is an $(\F,\varepsilon,M_0)$-good event, the path collection $\CA \overset{\mathrm{def.}}{=}\pi_{T_3}(A)$ is $(\F,\varepsilon,M_0)$-good.
    Then, we can apply Proposition \ref{prop: Psi properties} and obtain a mapping $\Psi$ satisfying properties \eqref{eq: covering all vertices Psi}, \eqref{eq: enough covering paths Psi} and \eqref{eq: injective mapping Psi}.
    By \eqref{eq: covering all vertices Psi}, we have $\Psi(\omega) \subset \CU_{\gamma,N}$ for any $\omega \in \CA$. 
    Therefore, 
    \begin{equation}
    \label{eq: lower bound for covering paths}
        |\,\CU_{\gamma,N}| \geq \big| \bigcup_{\omega \in \CA} \Psi(\omega) \big|
        \overset{\eqref{eq: injective mapping Psi}}{=} \sum_{\omega \in \CA} |\Psi(\omega)|
        \overset{\eqref{eq: enough covering paths Psi}}{\geq}|\CA| \cdot (2d)^{T_4-T_3- J(\F,M_0)}.
    \end{equation}
    Note that $|\Omega_{k}|=(2d)^{k}\cdot |\T_N|$. 
    Therefore, by the definitions of $\CU_{\gamma,N}$ and $\CA$, we obtain
    \begin{equation}
    \label{eq: above loop insertion description}
        \frac{\P(U_{\gamma,N})}{\P(A)}\overset{\eqref{eq: paths correspond to event}}{=} \frac{|\CU_{\gamma,N}|/\left((2d)^{T_4}\cdot |\T_N| \right)}{|\CA|/\left((2d)^{T_3}\cdot |\T_N|\right)} \overset{\eqref{eq: lower bound for covering paths}}{\geq} (2d)^{-J(\F,M_0)}, 
    \end{equation}
    which implies \eqref{eq: loop insertion ineq}.
\end{proof}

The remainder of this section is devoted to the proof of Proposition \ref{prop: Psi properties}, which roughly corresponds to \textbf{Stage 3} and \textbf{Stage 4} of the four-stage strategy. 

Recall the definition of the auxiliary graph $G(\cdot)$ and the subset $\R(\cdot)$ below \eqref{eq: late points in the bulk and edge typical}.
Let $\omega \in \CA$. 
The ``surgeries'' performed on the path $\omega$ at \textbf{Stage 3} proceed as follows. 
First, for each vertex $\x\in \R(\L_{\F}^{\ai}(\omega))$, we insert a loop (which we refer to as $\beta(\x,\y)$) rooted at $\y$, the vertex closest\footnote{Recall that all points in $\T_N$ are ordered (see below \eqref{eq:lNdef}), thus allowing us to break ties by choosing the vertex with the smallest label. \label{ft:1}} to $\x$ in the trace of random walk between $T_2$ and $T_3$, to cover $\x$. 
This ensures that each component $\sfC$ in $\C(\L_{\F}^{\ai}(\omega))$ is connected to the path $\omega$.
Second, we insert loops $\beta^{\mathrm{Tree}}(\mathsf{C},\x)$ to cover the remaining late points $\L_{\F}^{\ai}(\omega) \setminus \R(\L_{\F}^{\ai}(\omega))$ ``cluster by cluster".
Specifically, the loop $\beta^{\mathrm{Tree}}(\mathsf{C},\x)$ starts at a vertex $\x$ in $\mathsf{C}$ and visits all vertices in $\mathsf{C}$, where $\mathsf{C}$ is a subset of $\T_N$ such that its auxiliary graph $G(\mathsf{C})$ is connected.
These two steps constitute the strategy at \textbf{Stage 3}. 

Let us remark that the main purpose of introducing two types of loops 
is to ensure that the loops corresponding to different connected components of the auxiliary graph do not intersect each other, which is a crucial property used in the proof of \eqref{eq: injective mapping Psi} where we recover the original path from the path after loop insertion; see the proof of Lemmas \ref{lem: inherit local information} and \ref{lem: final prep for recover first step} for reference.

At \textbf{Stage 4}, we allow the path after the insertion of loops to roam freely up to length $T_4$. All such paths are collected into the set $\Psi(\omega)$; see \eqref{eq: def Psi} for reference.

We now turn to the claims of Proposition \ref{prop: Psi properties}. The first claim \eqref{eq: covering all vertices Psi} follows directly from the fact that all late points are covered by the inserted loops. 

For \eqref{eq: enough covering paths Psi}, the total length of the inserted loops is smaller than the LHS of \eqref{no too much cost for extra loops}.
Specifically, the length of each loop $\beta(\x,\y)$ is bounded by a constant multiple of $d_\infty(\x,\y)\vee 1$ and the length of each $\beta^{\mathrm{Tree}}(\mathsf{C},\x)$ loop is bounded by $6l_N(|\mathsf{C}|-1)$.
Since $\CA$ is $(\F,\varepsilon,M_0)$-good, this ensures that $\Psi$ satisfies \eqref{eq: enough covering paths Psi}.

Finally, for \eqref{eq: injective mapping Psi}, we demonstrate that the original path $\omega$ can be recovered from $\wt{\omega} \in \Psi(\omega)$.
Specifically, loops $\beta(\x,\y)$ are inserted the first time the path hits $\y$ after $T_2$, and $\beta^{\mathrm{Tree}}(\mathsf{C},\x)$ are inserted the first time the path hits $\x$ after $T_2$, where we recall that $T_2=\lfloor \ai t_{\cov} \rfloor=\lfloor \gamma t_{\cov}-5\varepsilon N^d\rfloor$, so the locations of the inserted loops in $\widetilde{\omega}$ can be deduced from $\widetilde{\omega}$ and deleted to obtain $\omega$.
For the loops $\beta^{\mathrm{Tree}}(\mathsf{C},\x)$, they are determined by the first $T_2$ steps of $\omega$, which are identical to those of $\widetilde{\omega}$.
For the loops $\beta(\x,\y)$, our construction ensures that $\beta(\x,\y)$ is self-avoiding (except at the starting point), which enables us to deduce $\beta(\x,\y)$ from $\widetilde{\omega}$; see Section \ref{sec: proof of proposition rfs} for details.

We conclude this part with the organization of this section.
In Section \ref{sec: Construction of loops to be inserted}, we construct the loops $\beta(\x,\y)$ and $\beta^{\mathrm{Tree}}(\sfC,\x)$ and demonstrate their key properties.
In Section \ref{sec: Loop insertion subsec}, we construct the mapping $\Psi(\cdot)$ by inserting loops defined in Section \ref{sec: Construction of loops to be inserted} and then extending the path to length $T_4$.
In Section \ref{sec: proof of psi props}, we prove that the mapping $\Psi$ possesses the properties specified in \eqref{eq: Psi properties}.
The proof of Proposition \ref{prop: recover first step}, which concerns the removal of loops $\beta(\x,\y)$, is more intricate and is provided in Section \ref{sec: proof of proposition rfs}.

\subsection{Construction of loops}
\label{sec: Construction of loops to be inserted}

We begin with the definition of loops. For each pair of vertices $\x, \y \in \T_N$, 
we write $\x=(x_1,x_2,\ldots,x_d)$ and $\y=(y_1,y_2,\ldots,y_d)$,
where $d$ is the dimension of the torus.
Define the difference vector $\bs{\Delta}=\y-\x=(\Delta_1,\Delta_2,\ldots,\Delta_d)$.
Since $\bs{\Delta}$ can be viewed as a vector in $(\Z/N\Z)^d$, we can assume without loss of generality that $| \Delta_i | \leq \frac{N+1}{2}$ for $1 \leq i \leq d$.
Thus, $d_{\infty}(\x,\y)=\max_{1 \leq i \leq d}|\Delta_i|$.

We categorize the vector $\bs{\Delta}=\y-\x$ into four {\bf types} based on the values of its components viewed as a vector in $(\Z/N\Z)^d$.
\begin{enumerate}[label=(\alph*)]
    \item There exists $1 \leq i < j \leq d$ such that $\Delta_i \neq 0$ and $\Delta_j \neq 0$.
    \item There exists $1 \leq i \leq d-1$ such that $\Delta_i \neq 0$, and for any $j \neq i$, $\Delta_j =0$.
    \item $\Delta_d \neq 0$, and for any $1 \leq i \leq d-1$, $\Delta_i=0$.
    \item $\bs{\Delta}=\bs{0}$.
\end{enumerate}
We also specify the type of each vertex $\x\in \T_N$ as the type of the vector $\bs{\Delta}=\near(\omega,\x)-\x$, where $\near(\omega,\x)$ refers to the vertex in the path segment $\omega(T_2,T_3]$  closest\footnote{See Footnote \ref{ft:1}.} to $\x$. 

We now construct the first kind of loops $\beta(\x,\y)$ that start and end at $\y$, according to the type of $\bs{\Delta}$.  
For types (a), (b) and (c), the loop consists of two parts: the $\y \to \x$ part and the $\x \to \y$ part.
For clarity, we include a figure (Figure \ref{fig: loops inserted in the first step}) that illustrates the loops we construct in each type.

\begin{figure}[tb!]
    \centering
    \subfigure[$\Delta_i \neq 0$ and $\Delta_j \neq 0$, where $1 \leq i<j \leq d$.]{
        \includegraphics[width=0.45\linewidth]{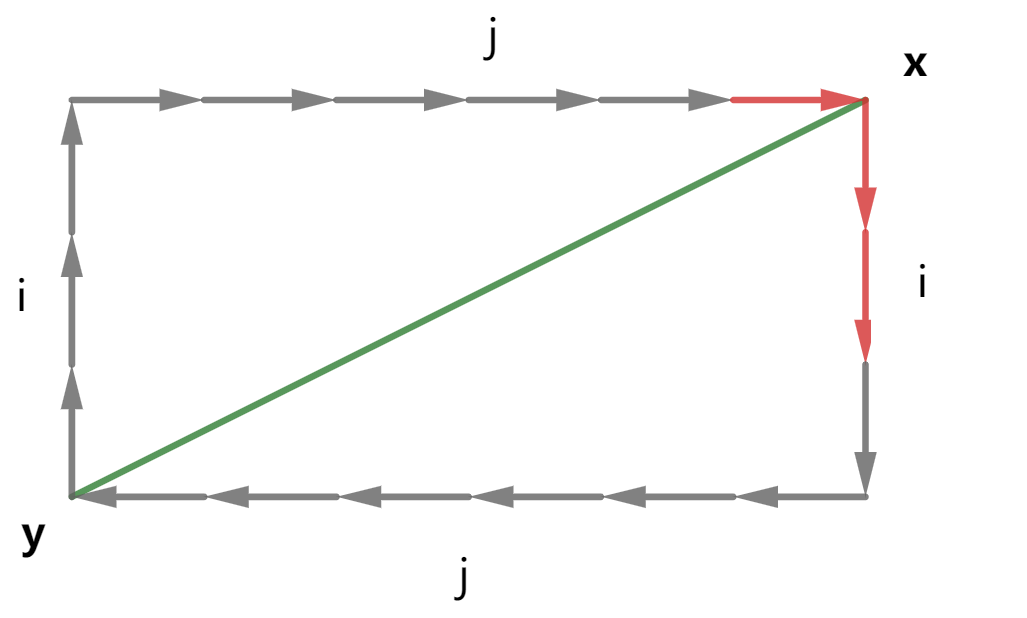}
        \label{Case 1 image}
    }
    \quad
    \subfigure[$\Delta_i \neq 0$ for some $i <d$. $\Delta_j =0$ for any $j \neq i$.]{
        \includegraphics[width=0.45\linewidth]{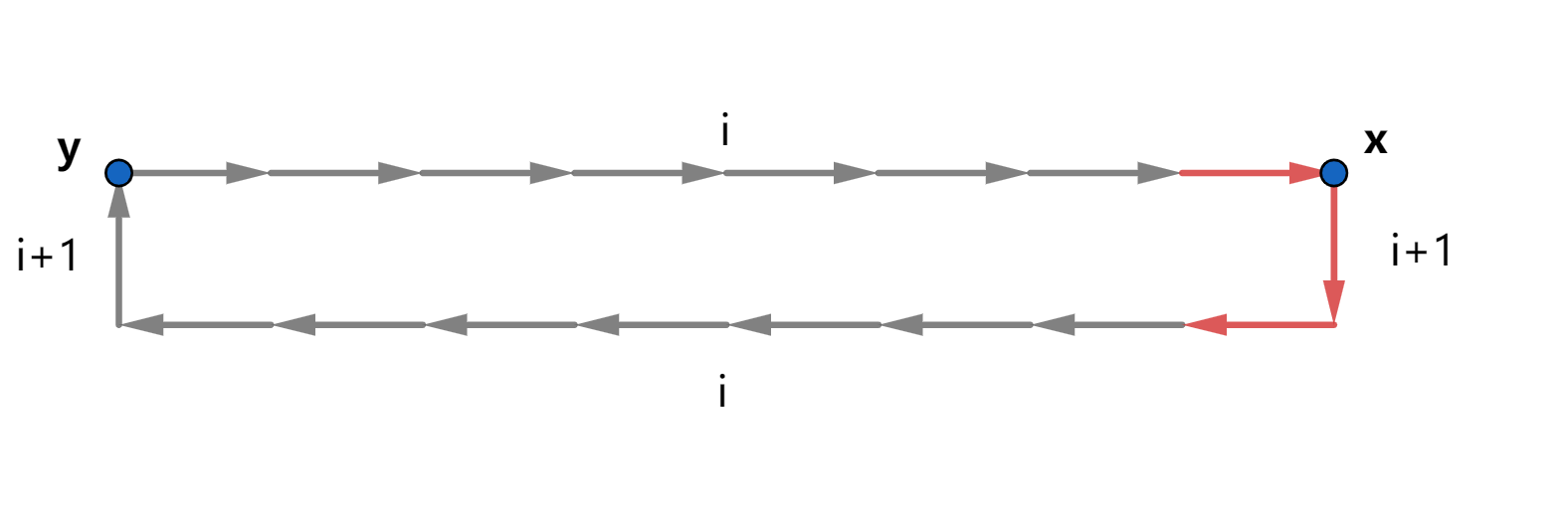}
        \label{Case 2 image}
    }
    \\
    \subfigure[$\Delta_d \neq 0$. $\Delta_j =0$ for any $j \neq d$]{
        \includegraphics[width=0.45\linewidth]{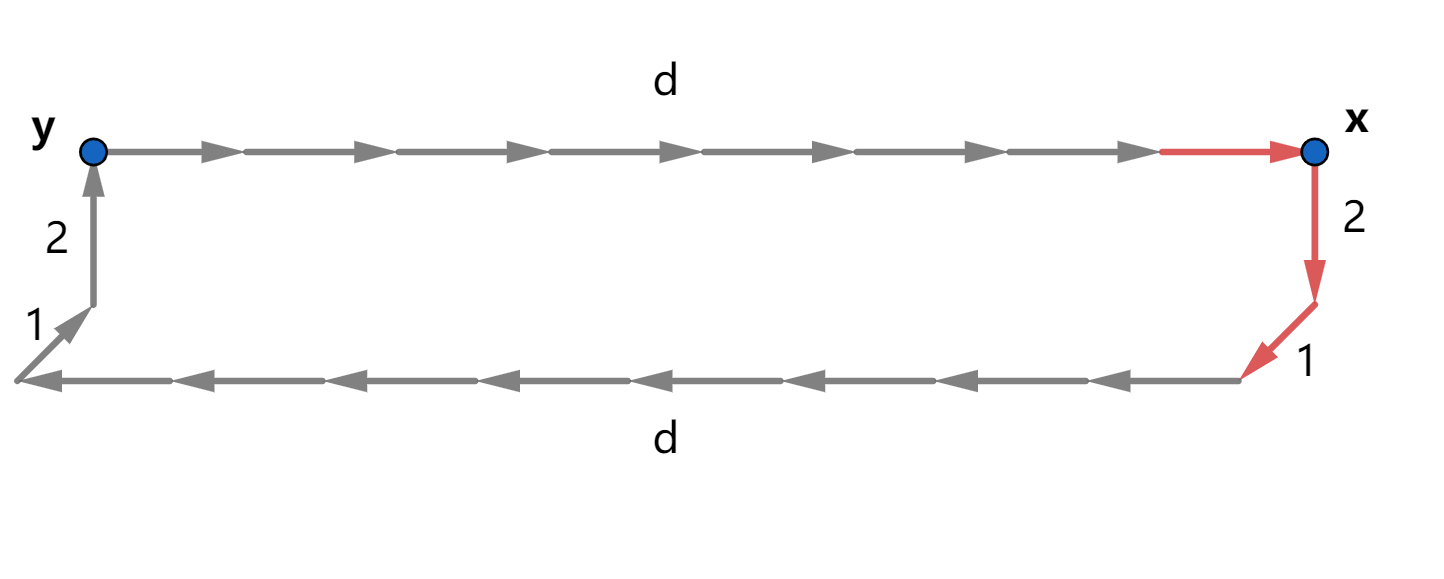}
        \label{Case 3 image}
    }
    \quad
    \subfigure[$\bs{\Delta}=\bs{0}$]{
        \includegraphics[width=0.45\linewidth]{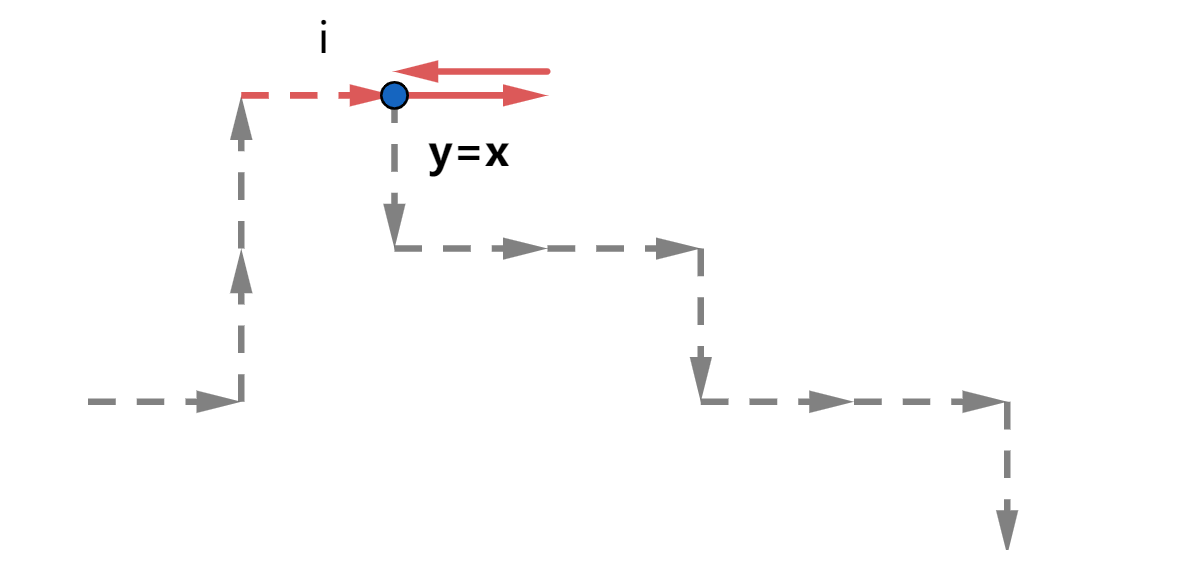}
        \label{Case 4 image}
    }
    \caption{The loops $\beta(\x,\y)$. A solid short arrow represents a single step of the loop. A number $i$ beside the arrows indicates that this step moves in the $i$-th coordinate. To emphasize the \emph{local} picture at $\x$ for each type, we highlight specific regions with red color. The local structure at $x$ plays a critical role in the proof of Proposition \ref{prop: Psi properties}.} 
    \label{fig: loops inserted in the first step}
    \smallskip
\end{figure}

\vspace{10px}

If the type of the vector $\bs{\Delta}=\y-\x$ is (a), we construct $\beta=\beta(\x,\y)$ as follows. We refer readers to Figure \ref{fig: loops inserted in the first step}\subref{Case 1 image} for reference. 
\begin{itemize}
    \item Let $\beta(0)=\y$.
    \item $\y \to \x$: For the $k$-th step, if the path $\beta$ has not visited $\x$ in the first $(k-1)$ steps, let $i$ be the minimal index such that $\beta(k-1)_i \neq x_i$, and let the $k$-th step $\beta(k)-\beta(k-1)=e_i \cdot \sgn (\Delta_i)$.
    \item $\x \to \y$: For the $k$-th step, if the path $\beta$ has visited $\x$ in the first $(k-1)$ steps and $\beta(k-1) \neq \y$, let $i$ be the minimum index such that $\beta(k)_i \neq y_i$, and let the $k$-th step $\beta(k)-\beta(k-1)=-e_i \cdot \sgn (\Delta_i)$. If $\beta(k) = \y$ for some $k>0$, we end the path here. 
\end{itemize}
Since $d_1(\beta(k+1),\x)=d_1(\beta(k),\x)-1$ before the path visits $\x$ and $d_1(\beta(k+1),\y)=d_1(\beta(k),\y)-1$ after the path visits $\x$, the path $\beta$ is a loop connecting $\y$ and $\x$.

\smallskip

If the type of the vector $\bs{\Delta}=\y-\x$ is (b), suppose that $\Delta_i \neq 0$, where $1 \leq i \leq d-1$.
We construct $\beta=\beta(\x,\y)$ as follows.
 We refer readers to Figure \ref{fig: loops inserted in the first step}\subref{Case 2 image} for reference.
\begin{itemize}
    \item Let $\beta(0)=\y$.
    \item $\y \to \x$: Let the path $\beta$ move in the direction of $e_i \cdot \sgn (\Delta_i)$ for $|\Delta_i|$ steps, and the path $\beta$ visits $\x$.
    \item $\x \to \y$: First, let the path $\beta$ move in the direction of $e_{i+1}$ for one step. Next, let the path $\beta$ move in the direction of $-e_i \cdot \sgn (\Delta_i)$ for $|\Delta_i|$ steps. Finally, let the path $\beta$ move in the direction of $-e_{i+1}$ for one step. 
\end{itemize}

\smallskip

If the type of the vector $\bs{\Delta}=\y-\x$ is (c), we construct $\beta=\beta(\x,\y)$ as follows.
 We refer readers to Figure \ref{fig: loops inserted in the first step}\subref{Case 3 image} for reference.
\begin{itemize}
    \item Let $\beta(0)=\y$.
    \item $\y \to \x$: Let the path $\beta$ move in the direction of $e_d \cdot \sgn (\Delta_d)$ for $|\Delta_d|$ steps, and the path $\beta$ visits $\x$.
    \item $\x \to \y$: First, let the path $\beta$ move in the direction of $e_{2}$ for one step and then move in the direction of $e_{1}$ for one step. Next, let the path $\beta$ move in the direction of $-e_d \cdot \sgn (\Delta_d)$ for $|\Delta_d|$ steps. Finally, let the path $\beta$ move in the direction of $-e_{1}$ for one step and then move in the direction of $-e_{2}$ for one step. 
\end{itemize}

\smallskip

If the type of the vector $\bs{\Delta}=\y-\x$ is (d), we have $\x=\y$. 
We let 
\begin{equation} \label{eq: beta-i}
    \beta_i(\x)= (\x, \x+e_i, \x)
\end{equation}
and choose the loop to be inserted from $\beta_i$, $1 \leq i \leq d$  
according to the path; see the second line of \eqref{eq: loops to be inserted in STEP1 for each vertex}.
 We refer readers to Figure \ref{fig: loops inserted in the first step}\subref{Case 4 image} for reference.

\begin{remark}
\begin{enumerate}[label=(\roman*)]
    \item When the type of $\bs{\Delta}$ is (b) or (c), we introduce extra steps ($\pm e_{i+1}$ for type (b) and $\pm e_1,\pm e_2$ for type (c)) to ensure that the loop $\beta(\x,\y)$ is non-self-intersecting, except at the starting point. 
    Furthermore, these extra steps allow us to deduce the type of $\x$ from $\widetilde{\omega}$, see Lemma \ref{lem: identify the case Psi}.
    \item When the type of $\bs{\Delta}$ is (d), loops in the form of $\beta_i(\x)$ are inserted to ensure that the path after loop insertion does not change direction at $\x$ when first hitting $\x$.
    This construction enables us to deduce the type of $\x$ from $\widetilde{\omega}$; see Lemmas \ref{lem: identify the case Psi} and \ref{lem: inherit local information}.
\end{enumerate}
\end{remark}

\smallskip
\begin{figure}[tb!]
    \centering
    \subfigure[Vertices in the set $\mathsf{C}$. The side-length of the cubes is $l_N$. Vertices are labeled with black numbers.]{
\includegraphics[width=0.45\linewidth]{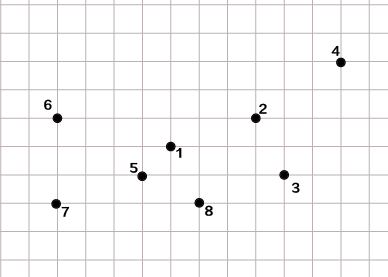}
        \label{STEP 1}
    }
    \quad
    \subfigure[Auxiliary graph $G(\mathsf{C})$ and $\mathsf{Tree}(\mathsf{C})$. Edges of $\mathsf{Tree}(\mathsf{C})$ are represented by solid red lines, while the remaining edges in the graph $G(\mathsf{C})$ are depicted using black dashed lines.]{
        \includegraphics[width=0.45\linewidth]{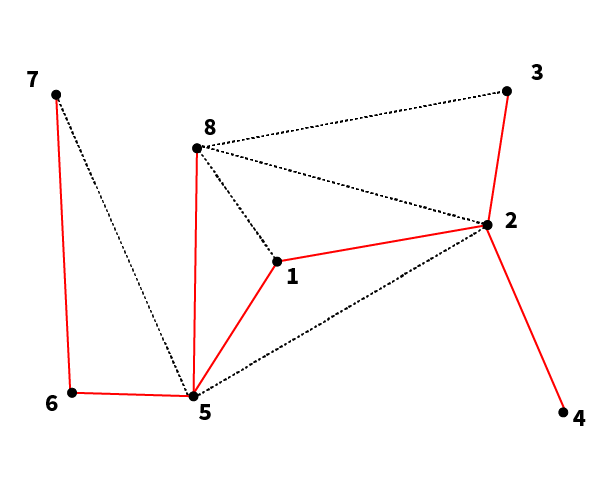}
        \label{STEP 2}
    }
    \\
    \subfigure[Loop $\beta^{\mathrm{Tree},\mathrm{Aux}}(\mathsf{C})$. Steps of loop $\beta^{\mathrm{Tree},\mathrm{Aux}}(\mathsf{C})$ are represented by the blue arrows connecting the two vertices of an edge. The smaller blue numbers next to the arrow indicates which step in the loop this arrow corresponds to.]{
        \includegraphics[width=0.45\linewidth]{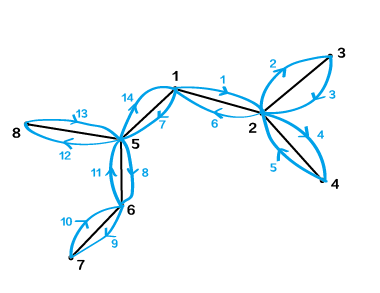}
        \label{STEP 3}
    }
    \quad
    \subfigure[Loop $\beta^{\mathrm{Tree}}(\mathsf{C})$. Steps of $\beta^{\mathrm{Tree}}(\mathsf{C})$ are represented by the blue arrows. The blue arrows next to the number $k$ correspond to the $k$-th step of loop $\beta^{\mathrm{Tree},\mathrm{Aux}}(\mathsf{C})$.]{
        \includegraphics[width=0.45\linewidth]{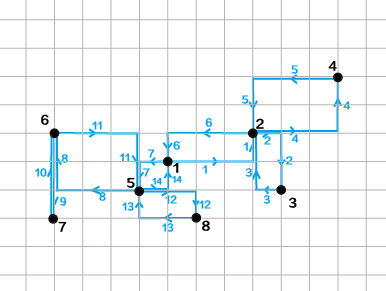}
        \label{STEP 4}
    }
    \caption{Construction of loop $\beta^{\mathrm{Tree}}(\mathsf{C},\x)$. }
    \label{fig: loops inserted in the second step}
\end{figure}

We now turn to the construction of the second kind of loops $\beta^{\mathrm{Tree}}(\sfC,\x)$, where $\x \in \sfC$ and $\sfC$ is a subset of $\T_N$ with a connected auxiliary graph $G(\sfC)$.
The construction proceeds as follows.
First, we find an auxiliary loop $\beta^{\Tree,\,\mathrm{Aux}}(\sfC,\x)$ on the auxiliary graph of $\sfC$, which visits each vertex at least once. 
Then, we connect the vertices in $\sfC$ on $\T_N$ according to the loop $\beta^{\Tree,\,\mathrm{Aux}}(\sfC,\x)$.
We refer readers to Figure \ref{fig: loops inserted in the second step} for reference.

We begin by introducing $\mathsf{Path}(\y,\x)$ to convert one step on the auxiliary graph into part of a path on $\T_N$.
\begin{definition}
    For two distinct vertices $\x\neq\y \in \T_N$, we define $\mathsf{Path}(\y,\x)$ as the $\y \to \x$ part in the construction of loop $\beta(\x,\y)$ when the type of $\bs{\Delta}=\y-\x$ is (a), (b) or (c).
\end{definition}
In fact, in the definition above one may also choose some other shortest path connecting $\x$ and $\y$ using another a priori fixed rule that ensures that the loop $\beta^{\Tree}(\sfC,\x)$ is completely determined by the component $\sfC$ and root $\x$; see the proof of Lemma \ref{lem: recover second step}.
Note that $\mathsf{Path}(\y,\x)$ is a non-self-intersecting path connecting $\y$ and $\x$, with a length no greater than $d|\x-\y|_{\infty}$, and each vertex $\z$ in the path satisfies $|\z-\x|_\infty \leq |\y-\x|_\infty$.

We now define the loop $\beta^{\mathrm{Tree}}(\mathsf{C},\x)$ by first constructing the auxiliary loop $\beta^{\Tree,\,\mathrm{Aux}}(\sfC,\x)$, and then replacing each step $\y\to\z$ in $\beta^{\Tree,\,\mathrm{Aux}}(\sfC,\x)$ with $\mathsf{Path}(\y,\z)$ to obtain a loop on $\T_N$.
To construct the auxiliary loop $\beta^{\Tree,\,\mathrm{Aux}}(\sfC,\x)$, we perform a depth-first-search (DFS) algorithm to explore the graph $G(\sfC)$ and record the order each vertex is visited. 
We refer readers to Figures \ref{fig: loops inserted in the second step}\subref{STEP 2}, \ref{fig: loops inserted in the second step}\subref{STEP 3} and \ref{fig: loops inserted in the second step}\subref{STEP 4} for reference.
\begin{definition} \label{def: loops to be inserted in STEP2}
    For each $\mathsf{C} \subset \T_N$ and $\x \in \sfC$, we apply DFS\footnote{Recall that all points in $\T_N$ are a priori ordered. See also Footnote \ref{ft:1}.} to generate a spanning tree $\Tree(\sfC)$ rooted at $\x$. The sequence of vertices visited during the exploration forms a loop \emph{on the graph} $G(\sfC)$, which is denoted by $\beta^{\Tree,\,\mathrm{Aux}}(\sfC,\x)$.
    We now write $\beta^{\mathrm{Tree},\mathrm{Aux}}(\mathsf{C},\x)=(\z_0,\z_1, \ldots, \z_L)$ and concatenate the paths $\mathsf{Path}(\z_i,\z_{i+1})$, $0 \leq i \leq L-1$ in sequence to obtain a loop \emph{on the graph} $\T_N$, which is denoted by $\beta^{\mathrm{Tree}}(\mathsf{C},\x)$.
\end{definition}
It is elementary that the loop $\beta^{\mathrm{Tree}}(\mathsf{C},\x)$ is completely determined by the subset $\mathsf{C}$ and the vertex $\x$.

\smallskip

We conclude this subsection by stating some key properties of loops $\beta(\x,\y)$ and $\beta^{\mathrm{Tree}}(\mathsf{C},\x)$, which are used in the proof of Proposition \ref{prop: Psi properties}. We omit their proofs as they can be easily verified from the definition.

The first property provides upper bounds for the length of the loops constructed above. Recall that we define $L(\omega)$ as the length of the path $\omega$; see below \eqref{eq: random walk paths} for reference.
\begin{lemma}
\begin{enumerate}[label=(\roman*)]
    \item For two vertices $\x ,\y \in \T_N$, we have 
    \begin{equation}
    \label{eq: upper bound for length of beta}
        L(\beta(\x,\y)) \leq 2d \cdot (d_{\infty}(\x,\y)+2).
    \end{equation}
    \item For each $\mathsf{C} \subset \T_N$ and $\x \in \sfC$, we have
    \begin{equation}
    \label{eq: upper bound for length of betatree}
        L(\beta^{\mathrm{Tree}}(\sfC,\x)) \leq 6dl_N(|\mathsf{C}|-1).
    \end{equation}
\end{enumerate} 
\end{lemma}
The second key property is that the part of the loop $\beta(\x,\y)$ before hitting $\x$ and the part of the loop $\beta(\x,\y)$ after hitting $\x$ are disjoint.
This property allows us to deduce $\beta(\x,\y)$ from $\widetilde{\omega} \in \Psi(\omega)$, enabling the recovery of the path $\omega$ from $\wt{\omega}$.

\begin{lemma}
    \label{lem: loops no self intersecting}
    For any distinct $\x \neq \y \in \T_N$, the loop $\beta=\beta(\x,\y)$ satisfies  
    \begin{equation}
    \label{eq: loops no self intersecting}
        \beta[1,H_{\x}-1] \cap \beta[H_{\x}+1,L(\beta)-1]=\emptyset,
    \end{equation}
    where $H_{\x}$ is the time when the loop $\beta$ first hits vertex $\x$.
\end{lemma}

\subsection{Loop insertion}
\label{sec: Loop insertion subsec}
In this subsection, we will construct the mapping $\Psi$ in the statement of Proposition \ref{prop: Psi properties} by first inserting loops $\beta(\x,\y)$ and $\beta^{\mathrm{Tree}}(\sfC,\x)$ into the path $\omega$ (at \textbf{Stage 3} of the strategy), and then allowing the path to roam freely (at \textbf{Stage 4}). 
 To describe the exact time (and location) to insert these loops, we introduce the hitting time of the path $\omega$ for  a vertex $\z \in \T_N$ as follows:
\begin{equation} \label{eq: def: hitting time after lambda}
    H_{\lambda}(\omega,\z)=\inf \{k \geq \lambda t_{\cov}\, : \,\omega(k) = \z\},
\end{equation}
which is the first time that the path $\omega$ hits the vertex $\z$ after time $\lambda t_{\cov}$.

\smallskip
We are ready to define the mapping $\Psi$.  
From now on in this subsection, we fix the path $\omega \in \CA$, where $\CA$ is an $(\F,\varepsilon,M_0)$-good path collection; see Definition \ref{def: good paths prop} for the definition of $(\F,\varepsilon,M_0)$-good path collections.

The first step is to insert loops of the form $\beta_{\x}$ defined below to connect the vertices in $\R(\L^{\ai}_{\F}(\omega))$ to the path $\omega$.
Recall the definition of $\near(\cdot,\cdot)$ and the types of $\x$ at the beginning of Section \ref{sec: Construction of loops to be inserted}.
\begin{definition}
\label{def: loops to be inserted in STEP1}
    For any $\omega \in \CA$ and $\x\in \R(\L_{\F}^{\ai}(\omega))$, we define
    \begin{equation}
    \label{eq: loops to be inserted in STEP1 for each vertex}
        \beta_{\x}=\beta_{\x}(\omega)=
        \begin{cases}
            \beta(\x,\near(\omega,\x))&\text{if $\x \neq \near(\omega,\x)$},\\
            \beta_{i}(\x)&\text{if $\x = \near(\omega,\x)$ and $\Dir(\omega,H_{\ai}(\omega,\near(\omega,\x))-1)=i$},
        \end{cases}
    \end{equation}
    where $\Dir(\omega,k)=i$ if $\omega(k+1)-\omega(k)=\pm e_i$ for a path $\omega$ and $k \geq 0$. 
    Moreover, we define 
    \begin{equation}
    \label{eq: loops to be inserted in STEP1}
        B(\omega)=\{\beta_{\x}(\omega)\,:\, \x \in \R(\L^{\ai}_{\F}(\omega))\}
    \end{equation}
    as the collection of loops to be inserted in the first step.
 
\end{definition}
For each $\x \in \R(\L^{\ai}_{\F}(\omega))$, the loop $\beta_{\x}(\omega)$ is inserted one-by-one to connect $\x$ to the path $\omega$ at time $H_{\ai}(\omega,\near(\omega,\x))$. It remains 
 to determine the order of insertion. We now claim that the mapping $\near(\omega,\cdot)\mid_{\R(\L^{\ai}_{\F}(\omega))}$ is injective\footnote{In fact, by definition of $\bs{r}_{\x}(\omega)$ and \eqref{no too large late island}, we get $d_{\infty}(\x,\near(\omega,\x)) = \bs{r}_{\x}(\omega) \leq l_n$ for each $\x \in \R(\L^{\ai}_{\F}(\omega))$.
Recall that $d_\infty(x,x')>3l_N$ for any $x\neq x'\in \R(\L^{\ai}_{\F}(\omega))$, so $\near(\omega,\cdot)\mid_{\R(\L^{\ai}_{\F}(\omega))}$ is injective.\label{footnote: loops mutually disjoint}}, which 
ensures that 
$H_{\ai}(\omega,\near(\omega,\x)) \neq H_{\ai}(\omega,\near(\omega,\x'))$ for any $\x \neq \x' \in \R(\L^{\ai}_{\F}(\omega))$.
Hence, we can label the elements of $\R(\L_{\F}^{\ai}(\omega))$ as $\x_1,\x_2,\ldots \x_n$, where 
\begin{equation}
\label{eq: number of roots}
    n=|\R(\L_{\F}^{\ai}(\omega))|,
\end{equation} 
in such a way that $H_{\ai}(\omega,\near(\omega,\x_j))$ is strictly increasing with respect to $j$.
For simplicity, we let $\beta^{(j)}=\beta^{(j)}(\omega)=\beta_{\x_j}(\omega)$, which gives an order on $B(\omega)$. 
Furthermore, we denote $H_{\ai}(\omega,\near(\omega,\x_j))$ by $H_j(\omega)$ for the remainder of this section.

We now recursively insert the loops in $B(\omega)$ to connect the vertices in $\R(\L^{\ai}_{\F}(\omega))$ to the path $\omega$. 
Recall again the definition of $\Omega_k$ below \eqref{eq: random walk paths}. 
For a path $\omega \in \Omega_{L_1}$, a loop $\beta\in \Omega_{L_2}$ 
and $0 \leq k \leq L_1$ such that $\omega(k)=\beta(0)=\beta(L_2)$, we define the insertion map $\insmap(\omega,\beta,k)\in\Omega_{L_1+L_2}$ whose output is the new path with $\beta$ inserted into $\omega$ at time $k$.
    
We now set $\omega^{(0)}=\omega$ and $S^{(0)}(\omega)=0$. 
Using the mapping $\insmap$, 
we recursively define the result and total duration after $j$-th insertion, denoted by $\omega^{(j)}$ and $S^{(j)}$ resp.\ for $1 \leq j \leq n$ as follows:
\begin{equation}
\label{eq: inserting loops STEP 1}
    \omega^{(j)}=\insmap(\omega^{(j-1)},\beta^{(j)},H_j(\omega)+S^{(j-1)}(\omega)), \quad S^{(j)}(\omega)=S^{(j-1)}(\omega)+L(\beta^{(j)}).
\end{equation}

We now turn to the second step at \textbf{Stage 3}, which involves inserting loops in the form of $\beta^{\mathrm{Tree}}(\sfC,\x)$ to cover the remaining late points $\L_{\F}^{\ai}(\omega) \setminus \R(\L_{\F}^{\ai}(\omega))$.

Recall that $\C(\L_{\F}^{\ai}(\omega))$ represents the collection of connected components of $G(\L_{\F}^{\ai}(\omega))$. 
There exists a natural bijection $\x \mapsto \sfC_{\x}$ between $\R(\L_{\F}^{\ai}(\omega))$ and $\C(\L_{\F}^{\ai}(\omega))$ such that $\x \in \sfC_{\x}$.

\begin{definition}
    For each $\omega \in \CA$ and $\x\in \R(\L_{\F}^{\ai}(\omega))$, we define the collection of loops to be inserted in the second step as follows:
    \begin{equation}
    \label{eq: loops to be inserted in STEP2}
        B^{\Tree}(\omega)=\big\{\beta^{\Tree}(\sfC_{\x},\x)\,:\,\x \in \R(\L_{\F}^{\ai}(\omega)) \big\}.
    \end{equation}
\end{definition}
For clarity, we denote $\beta^{(j),\text{Tree}}=\beta^{\mathrm{Tree}}(\mathsf{C}_{\x_j},\x_j)$, which provides an order on $B^{\Tree}(\omega)$.
For simplicity, let $H_j(\omega^{(n)})=H_{\ai}(\omega^{(n)},\x_j)$. 

We now recursively insert the loops in $B^{\Tree}(\omega)$. 
We initialize by setting $S^{(0),\text{Tree}}(\omega)=0$.
For each $1 \leq j \leq n$, write
\begin{equation}
\label{eq: inserting loops STEP 2}
\begin{split}
    \omega^{(n+j)}&=\insmap(\omega^{(n+j-1)},\beta^{(j),\text{Tree}},H_j(\omega^{(n)})+S^{(j-1),\text{Tree}}(\omega));\\
    S^{(j),\text{Tree}}(\omega)&=S^{(j-1),\text{Tree}}(\omega)+L(\beta^{(j),\text{Tree}})
\end{split}
\end{equation}
for the result and total duration after $j$-th insertion, respectively.

The final step of constructing $\Psi$ involves converting the path $\omega^{(2n)}$ into paths of length $T_4$, corresponding to \textbf{Stage 4}.
By estimating the total length of the loops inserted, we can show that $L(\omega^{(2n)}) \leq T_4$ for any $\omega \in \CA$; see Lemma \ref{lem: upper bound for path with loops} below.
This ensures that the path $\omega^{(2n)}$ can continue to roam freely after the time $L(\omega^{(2n)})$ until it reaches the total length  $T_4$. 
We collect all such extended paths into the set defined as $\Psi(\omega)$.

\begin{definition}
    For any $\omega \in \CA$, we let 
    \begin{equation}
    \label{eq: def Psi}
        \Psi(\omega)=\Big\{\widetilde{\omega} \in \Omega_{T_4}: \widetilde{\omega}=\omega^{(2n)}|_{L(\omega^{(2n)})} \Big\}.
    \end{equation}
\end{definition}

The construction of the mapping $\Psi$ is now complete.

\subsection{Proof of Proposition \ref{prop: Psi properties}}
\label{sec: proof of psi props}
The claim \eqref{eq: covering all vertices Psi} is a direct consequence of the construction of $\Psi$ and \eqref{covering complement of F}. We hence omit the proof.


We now proceed to the proof of \eqref{eq: enough covering paths Psi}.
The key to the proof lies in establishing an upper bound for $L(\omega^{(2n)})$, which is provided in Lemma \ref{lem: upper bound for path with loops}.

\begin{lemma}
\label{lem: upper bound for path with loops}
    For any $\omega \in \CA$, where $\CA$ is an $(\F,\varepsilon,M_0)$-good path collection, we have 
    \begin{equation}
        \label{eq: upper bound for path with loops}
        L(\omega^{(2n)})\leq T_3 + J(\F,M_0),
    \end{equation}
    where $J(\F,M_0)$ is defined in \eqref{cost of inserting loops}.
    In particular, by \eqref{eq: cost of inserting loops not large}, we have $L(\omega^{(2n)})\leq T_4$. 
\end{lemma}

\begin{proof}[Proof of \eqref{eq: enough covering paths Psi} assuming Lemma \ref{lem: upper bound for path with loops}]
    From the definition of $\Psi$ in \eqref{eq: def Psi}, we know that
    \begin{equation}
        |\Psi(\omega)|=(2d)^{T_4 - L(\omega^{(2n)})}.
    \end{equation}
    Combining this with \eqref{eq: upper bound for path with loops} gives \eqref{eq: enough covering paths Psi}.
\end{proof}

We now prove Lemma \ref{lem: upper bound for path with loops} by  estimating the total length of the loops inserted at \textbf{Stage 3}.

\begin{proof}[Proof of Lemma \ref{lem: upper bound for path with loops}]
    By the definition of $\Psi$, we have
    \begin{equation}
    \label{eq: reprentation of L(omega2n)}
        L(\omega^{(2n)})=T_3+S^{(n)}(\omega) + S^{(n),\text{Tree}}(\omega).
    \end{equation}
    Thus, it suffices to bound $S^{(n)}(\omega)$ and $S^{(n),\text{Tree}}(\omega)$ separately.
    
    For $S^{(n)}(\omega)$, we recall the definition of $\bs{r}_{x}(\omega)$ (see \eqref{eq: def: distance to random walk between T_2 and T_3}) and $\near(\cdot,\cdot)$ (see the beginning of Section \ref{sec: Construction of loops to be inserted}), and obtain by \eqref{eq: upper bound for length of beta}
    \begin{equation}
    \label{eq: upper bound for total length in STEP 1}
        S^{(n)}(\omega) = \sum_{\x \in \R(\L_{\F}^{\ai}(\omega))}L(\beta(\x,\near(\omega,\x))) \leq \sum_{\x \in \R(\L_{\F}^{\ai}(\omega))} 2d(\bs{r}_{\x}(\omega)+2).
    \end{equation}
    For $S^{(n),\text{Tree}}(\omega)$, we recall $n=|\R(\L_{\F}^{\ai}(\omega))|$ (see \eqref{eq: number of roots}) and obtain by \eqref{eq: upper bound for length of betatree} 
    \begin{equation}
    \label{eq: upper bound for total length in STEP 2}
    \begin{aligned}
        S^{(n),\Tree}(\omega)=\sum_{j=1}^{n}L(\beta^{\Tree}(\sfC_j,\x_j))\leq \sum_{\sfC \in \C(\L_{\F}^{\ai}(\omega))} 6d l_N(|\sfC|-1)
        =6dl_N(|\L_{\F}^{\ai}(\omega)|-|\R(\L_{\F}^{\ai}(\omega))|),
    \end{aligned}
    \end{equation}
    where the last equation follows from $|\C(\L_{\F}^{\ai}(\omega))|=n$ and the fact that $\C(\L_{\F}^{\ai}(\omega))$ is a partition of $\L_{\F}^{\ai}(\omega)$.
    Hence, the total length of the loops inserted at \textbf{Stage 3} is
    \begin{equation}
    \label{eq: total length of loops}
    \begin{aligned}
        S^{(n)}(\omega) + S^{(n),\text{Tree}}(\omega)\overset{\eqref{eq: upper bound for total length in STEP 1}}&{\underset{ \eqref{eq: upper bound for total length in STEP 2}}{\leq}}  \sum_{\x \in \R(\L_{\F}^{\ai}(\omega))} 2d(\bs{r}_{\x}(\omega)+2) + 6dl_n(|\L_{\F}^{\ai}(\omega)|-|\R(\L_{\F}^{\ai}(\omega))|) \\
        \overset{\eqref{no too much cost for extra loops}}&{\underset{\eqref{cost of inserting loops}}{\leq}} J(\F,M_0).
    \end{aligned}
    \end{equation}
    Note that \eqref{no too much cost for extra loops} holds here because $\omega$ is a path in the $(\F,\varepsilon,M_0)$-good path collection $\CA$.
    Combining this with \eqref{eq: reprentation of L(omega2n)} gives \eqref{eq: upper bound for path with loops}.
\end{proof}

We now turn to the proof of \eqref{eq: injective mapping Psi}.
Suppose that $\widetilde{\omega} \in \Psi(\CA)$, where $\Psi(\CA)\overset{\text{def.}}{=}\bigcup_{\omega \in \CA}\Psi(\omega)$.
Intuitively, if all loops and their insertion locations can be uniquely identified from $\widetilde{\omega}$, we can remove these loops to recover the original path.
After deleting all such loops from $\wt{\omega}$, the first $T_3$ steps of the resulting path gives $\omega$. 

For two paths $\omega,\wh{\omega} \in \Omega$, we say that $\wh{\omega}$ is an \emph{\ext}of $\omega$ if $L(\wh{\omega}) \geq L(\omega)$ and $\wh{\omega}=\omega \mid_{L(\omega)}$.
It is obvious that if $\wh{\omega}$ is an \ext of $\omega$ and $L(\omega) \geq \lambda t_{\cov}$, then $H_{\lambda}(\omega,\x)=H_{\lambda}(\wh{\omega},\x)$ for any $\x \in \omega[\lambda t_{\cov}, L(\omega)]$. Given a path $\omega \in \Omega$ and indices $0 \leq k_1 < k_2 \leq L(\omega)$ such that $\omega(k_1)=\omega(k_2)$, we define $\dltmap(\omega,k_1,k_2)\in \Omega_{L(\omega)-k_2+k_1}$ as the result of removing the loop $\omega(k_1,\ldots,k_2)$ from $\omega$.

\medskip

We now recover $\omega$ from $\wt{\omega}$. We first show that we can recover the path $\omega^{(n)}$ from $\wt{\omega}$ by deleting the loops inserted in the second step. Then, we show that $\omega$ can be recovered from $\omega^{(n)}$.

In the following lemma, we show that an \ext of path $\omega^{(n)}$ can be obtained by identifying the set $B^{\Tree}(\omega)$ and deleting the loops it contains.

\begin{lemma}
\label{lem: recover second step}
    For any $\widetilde{\omega} \in \Psi(\CA)$,  where $\CA$ is an $(\F,\varepsilon,M_0)$-good path collection, there exists a path $\ov{\omega}$ such that for any $\omega \in \CA$ satisfying $\widetilde{\omega} \in \Psi(\omega)$, the path
    $\ov{\omega}$ is an \ext of $\omega^{(n)}$.  
\end{lemma}

For clarity in the following proof, we say that a mapping $\Obj(\cdot)$ with domain $\CA$ (e.g., $B^{\Tree}(\cdot)$, $\L_{\F}^{\ai}(\cdot)$, etc.) is \emph{determined} by $\wt{\omega}$ if $\Obj(\omega)$ remains the same for all $\omega \in \CA$ such that $\widetilde{\omega} \in \Psi(\omega)$.
With slight abuse of notation, we also say that $\Obj(\omega)$ is determined by $\wt{\omega}$ if it is presupposed that $\wt{\omega} \in \Psi(\omega)$ and the mapping $\Obj(\cdot)$ is determined by $\wt{\omega}$.

\begin{proof}
    Let $\wt{\omega} \in \Psi(\CA)$, and fix $\omega \in \CA$ such that $\wt{\omega} \in \Psi(\omega)$.
    We aim to establish that the set $B^{\Tree}(\omega)$ is determined by $\wt{\omega}$.
    Since the insertion of loops occurs only after $T_2$ (by the definition of $\Psi$), it follows that $\wt{\omega}=\omega\mid_{T_2}$.
    Hence, $\L_{\F}^{\ai}(\widetilde{\omega})=\L_{\F}^{\ai}(\omega)$.
    The definition of $B^{\Tree}(\omega)$ depends solely on the set $\L_{\F}^{\ai}(\omega)$, and thus $B^{\Tree}(\omega)$ is determined by $\wt{\omega}$.

    Note that $n$ (see \eqref{eq: number of roots} for definition) is also determined by $\wt{\omega}$.
    To delete the loops in $B^{\Tree}(\omega)$ via the mapping $\dltmap$, it is necessary to demonstrate that $H_j(\omega^{(n)})+S^{(j-1),\text{Tree}}(\omega)$, the insertion times of loops into the original path, is determined by $\wt{\omega}$.
    Recall that we have written the set $\R(\L_{\F}^{\ai}(\omega))$ as $\x_1,\x_2, \ldots, \x_n$.
    Although the representation may depend on $\omega$, the set $\L_{\F}^{\ai}(\omega)$ is determined by $\wt{\omega}$.
    Note that\footnote{Otherwise, $\x_j \in \mathsf{Path}(\z_i,\z_{i+1})$ for some $\z_i\sim \z_{i+1} \in \mathsf{C}_{j'} $, leading to $|\z_i-\x|_{\infty} \leq |\z_i-\z_{i+1}|_{\infty}$, implying $\x_j \in \mathsf{C}_{j'}$, which contradicts the assumption. Similarly, $\x_j \notin \beta^{(j')}$.} $\x_j \notin \beta^{(j'),\mathrm{Tree}}$ and $\x_j \notin \beta^{(j')}$ when $j' \neq j$,
    so\footnote{We only consider the first time a path visits a vertex after time $T_2$ here.} $\omega$ visits the vertices $\near(\omega,\x_j)$ after time $T_2$ in the same sequence as $\wt{\omega}$ visits vertices $\x_j$ after time $T_2$.
    This ensures that the expression $\x_1,\x_2, \ldots, \x_n$ is determined by $\wt{\omega}$.
    Furthermore, by the definition of $\omega^{(n+j)}$ and $S^{(j)}$, we obtain
    \begin{equation}
    \label{eq: changing of hitting time}
        H_{\ai}(\widetilde{\omega},\x_j)=H_{\ai}(\omega^{(n+j')},\x_j)=H_j(\omega^{(n)})+S^{(j-1),\text{Tree}}(\omega) \,\,\,\,\mbox{   for all $1 \leq j \leq j'\leq n$},
    \end{equation}
    which implies that $H_j(\omega^{(n)})+S^{(j-1),\text{Tree}}(\omega)$ is determined by $\wt{\omega}$. 
    
    Consider the following deletion process.
    We start with the loop $\wt{\omega}^{(2n)}=\wt{\omega}$ and define recursively
    \begin{equation}
    \label{eq: deletion process to recover STEP 2}
        \wt{\omega}^{(n+j-1)}=\dltmap(\wt{\omega}^{(n+j)}, H_j(\omega^{(n)})+S^{(j-1),\text{Tree}}(\omega),H_j(\omega^{(n)})+S^{(j-1),\text{Tree}}(\omega)+L(\beta^{\Tree,(j)})),
    \end{equation}
    for all $1 \leq j \leq n$. 
    By induction, we can derive 
    that $\dltmap$ can be applied as shown in \eqref{eq: deletion process to recover STEP 2} and $\wt{\omega}^{(n+j)}$ is an \ext of $\omega^{(n+j)}$ for all $0 \leq j \leq n$. 
    Since the deletion process is determined by $\wt{\omega}$ from the discussion above, the path $\wh{\omega}=\wt{\omega}^{(n)}$ is also determined by $\wt{\omega}$. This finishes the proof of Lemma \ref{lem: recover second step}.
\end{proof}

To recover $\omega$ from $\wt{\omega}^{(n)}$, we show that it is possible to recover $\omega$ from any extension of $\omega^{(n)}$. 
Define $\overline{\Psi}(\omega)$ as the collection of all extensions of $\omega^{(n)}$, and let $\overline{\Psi}(\CA)\overset{\text{def.}}{=}\bigcup_{\omega \in \CA}\overline{\Psi}(\omega)$. 
We state the following proposition:
\begin{prop}
\label{prop: recover first step}
    Suppose that $\overline{\omega} \in \ov{\Psi}(\CA)$, where $\CA$ is an $(\F,\varepsilon,M_0)$-good path collection. There exists a \emph{unique} path $\omega \in \CA$ such that $\ov{\omega} \in \ov{\Psi}(\omega)$.
\end{prop}

The proof is deferred to Section \ref{sec: proof of proposition rfs}.
The central idea of the proof involves identifying the set $B(\omega)$, followed by deletion of the loops.
\begin{proof}[Proof of \eqref{eq: injective mapping Psi} assuming Proposition \ref{prop: recover first step}]
    For an $(\F,\varepsilon,M_0)$-good path collection $\CA$, suppose that there exists $\omega, \omega' \in \CA$ such that $\Psi(\omega) \cap \Psi(\omega')\neq \emptyset$. Then there exists a path $\wt{\omega} \in \Psi(\omega) \cap \Psi(\omega')$.
    For simplicity, we let $n=|\R(\L_{\F}^{\ai}(\omega))|$ and $n'=|\R(\L_{\F}^{\ai}(\omega'))|$.
    By Lemma \ref{lem: recover second step}, there exists a path $\ov{\omega}$ such that $\ov{\omega}$ is both an \ext of $\omega^{(n)}$ and an \ext of $(\omega')^{(n')}$, which implies $\ov{\omega} \in \ov{\Psi}(\omega) \cap \ov{\Psi}(\omega')$.
.    We now apply Proposition \ref{prop: recover first step} to the path $\ov{\omega}$ and conclude that $\omega=\omega'$.
\end{proof}

\subsection{Proof of Proposition \ref{prop: recover first step}}
\label{sec: proof of proposition rfs}
In this subsection, we prove Proposition \ref{prop: recover first step}.
The main approach parallels the proof of Lemma \ref{lem: recover second step}, but demands more intricate technical arguments.
For the remainder of this subsection, we fix $\ov{\omega} \in \ov{\Psi}(\CA)$.
For ease of exposition, we formally define what it means for $\Obj(\omega)$ to be \emph{determined} by $\ov{\omega}$, where $\Obj(\cdot)$ is a mapping with domain $\CA$; see Definition \ref{def:determined} below for examples.
Note that the meaning of ``$\Obj(\omega)$ is determined by $\ov{\omega}$" in this subsection slightly differs from that of ``$\Obj(\omega)$ is determined by $\wt{\omega}$" in the previous subsection. 

\begin{definition}
\label{def:determined}
    For a path $\ov{\omega} \in \ov{\Psi}(\CA)$ and a mapping $\Obj(\cdot)$ with domain $\CA$ (e.g., $\L_{\F}^{\ai}(\cdot)$, $B(\cdot)$ in \eqref{eq: loops to be inserted in STEP1}, or the mapping that maps $\omega\in \CA$ to the loop insertion process described in \eqref{eq: inserting loops STEP 1} etc.), we say that $\Obj(\cdot)$ is \emph{determined} by $\ov{\omega}$ if $\Obj(\omega)$ is identical for all $\omega \in \CA$ satisfying $\ov{\omega} \in \ov{\Psi}(\omega)$.
    With slight abuse of notation, we also say that $\Obj(\omega)$ is determined by $\ov{\omega}$ if we assume $\ov{\omega} \in \ov{\Psi}(\omega)$ and the mapping $\Obj(\cdot)$ is determined by $\ov{\omega}$.
\end{definition}

Recall that at the beginning of Section \ref{sec: Loop insertion subsec}, we have fixed the path $\omega \in \CA$ where $\CA$ is an $(\F,\varepsilon,M_0)$-good path collection. Also recall that $B(\omega)$, defined in Definition \ref{def: loops to be inserted in STEP1}, is the set of loops inserted to construct $\omega^{(n)}$. 
We begin by showing that the set $B(\omega)$ can be identified.
    
\begin{lemma}
\label{lem: identify the loops inserted Psi}
    For all $\overline{\omega}\in \ov{\Psi}(\CA)$ and $\omega \in \CA$ such that $\ov{\omega} \in \ov{\Psi}(\omega)$, the loop collection $B(\omega)$ is determined by $\ov{\omega}$.  
\end{lemma}  

To establish Lemma \ref{lem: identify the loops inserted Psi}, we begin by identifying the late points set $\L_{\F}^{\ai}(\omega)$.

\begin{lemma}
\label{lem: identify the late points Psi}
    For all $\overline{\omega}\in \ov{\Psi}(\CA)$ and $\omega \in \CA$ such that $\ov{\omega} \in \ov{\Psi}(\omega)$, the late points set $\L_{\F}^{\ai}(\omega)$ is determined by $\ov{\omega}$.
    In particular, $\R(\L_{\F}^{\ai}(\omega))$ is also determined by $\ov{\omega}$.
\end{lemma}

\begin{proof}
    Note that the insertion of loops occurs only after $T_2$ (by definition of $\Psi$), implying $\omega=\omega^{(n)} \mid_{T_2}$.
    Since $\ov{\omega}$ is an \ext of $\omega^{(n)}$, it follows that $\omega=\ov{\omega} \mid_{T_2}$.
    Therefore, $\L_{\F}^{\ai}(\omega)=\L_{\F}^{\ai}(\overline{\omega})$, which implies Lemma \ref{lem: identify the late points Psi}.
\end{proof}

After identifying the set $\R(\L_{\F}^{\ai}(\omega))$, we only need to identify $\beta_{\x}(\omega)$, the loop associated with each vertex $\x \in \R(\L_{\F}^{\ai}(\omega))$.
This process will be divided into two steps.
First, we identify the type of $\x$, as demonstrated in Lemma \ref{lem: identify the case Psi}. 
Next, assuming that the type of $\x$ is known, we identify $\beta_{\x}(\omega)$ in Lemma \ref{lem: identify the loop with case Psi}.

\smallskip
Recall the definition of the type of $\x$, introduced at the beginning of Section \ref{sec: Construction of loops to be inserted}.

\begin{lemma}
\label{lem: identify the case Psi}
    For all $\overline{\omega}\in \ov{\Psi}(\CA)$ and $\omega \in \CA$ such that $\ov{\omega} \in \ov{\Psi}(\omega)$, the type of each vertex $\x \in \R(\L_{\F}^{\ai}(\omega))$ is determined by $\ov{\omega}$.
\end{lemma}

Note that the type of each vertex $\x \in \R(\L_{\F}^{\ai}(\omega))$ may depend on the path $\omega$, as $\near(\omega,\x)$ is defined in terms of $\omega$. Recall that we have given each vertex in $\R(\L_{\F}^{\ai}(\omega))$ a label $j \in [1,n]$ (possibly depending on $\omega$) above \eqref{eq: number of roots}.
To prove Lemma \ref{lem: identify the case Psi}, we first show that the local structure of $\beta^{(j)}$ near $\x_j$, where $\x_j \in \R(\L_{\F}^{\ai}(\omega))$, can be deduced from $\ov{\omega}$.
Recall the definition of $\Dir(\omega,k)$ above \eqref{eq: loops to be inserted in STEP1}. 

\begin{lemma}
\label{lem: inherit local information}
    Suppose $\omega \in \CA$ and $\ov{\omega} \in \overline{\Psi}(\omega)$. 
    For all $\x_j \in \R(\L_{\F}^{\ai}(\omega))$: 
    \begin{enumerate}[label=(\roman*)]
        \item If the type of $\x_j$ is (d), we have 
        \begin{equation}
        \label{eq: case iv)}
            \Dir(\overline{\omega},H_{\ai}(\overline{\omega},\x_j))=\Dir(\overline{\omega},H_{\ai}(\overline{\omega},\x_j)-1).
        \end{equation}
        \item If the type of $\x_j$ is not (d), we have
        \begin{equation}
            \label{eq: path to loop inseted}
            \Dir(\overline{\omega},H_{\ai}(\overline{\omega},\x_j)+i)=\Dir(\beta^{(j)},H_0(\beta^{(j)},\x_j)+i), \quad i\in\{0,\pm 1\}.
        \end{equation}
    \end{enumerate}
\end{lemma}

\begin{proof}
    Recall $n=|\R(\L_{\F}^{\ai}(\omega))|$ in \eqref{eq: number of roots} 
    and the definitions of $H_j(\omega)$, $S^{(j)}(\omega)$ and $\omega^{(j)}$ in \eqref{eq: inserting loops STEP 1} for $1 \leq j \leq n$.
    We first show that loops in $B(\omega)$ (see \eqref{eq: loops to be inserted in STEP1}) are mutually disjoint. Recall that $d_{\infty}(\x,\near(\omega,\x))\leq l_n$ for each $\x \in \R(\L^{\ai}_{\F}(\omega))$ and $d_\infty(\x,\x')>3l_N$ for all $\x\neq \x'\in \R(\L^{\ai}_{\F}(\omega))$. Moreover, for all $\z \in \beta(\x,\y)$, we have $d_\infty(\x,\z) \leq d_{\infty}(\x,\y)+1$. Therefore, we obtain:  
    \begin{equation}
    \label{eq: loops mutually disjoint}
       \mbox{Loops in $B(\omega)$ are mutually disjoint.}
    \end{equation}
    In particular, we have
    $\x_j \notin \beta^{(j')}$ for all $j' \neq j$, which implies that the path $\omega^{(j)}$ hits $\x_j$ in the part corresponding to $\beta^{(j)}$.
    Furthermore, when $j'>j$, since the loop $\beta^{(j')}$ is inserted after $H_j(\omega) +S^{(j)}(\omega)$, the moment when the part of the loop $\beta^{(j)}$ terminates, the part of the path before $H_j(\omega) +S^{(j)}(\omega)$ remains unchanged.
    Thus, by the definition of extension, 
    we obtain
    \begin{equation}
        \label{eq: hitting moment}
        H_{\ai}(\overline{\omega},\x_j)=H_j(\omega)+S^{(j-1)}(\omega)+H_0(\beta^{(j)},\x_j).
    \end{equation}
    It is straightforward to verify \eqref{eq: case iv)} using \eqref{eq: loops to be inserted in STEP1} and \eqref{eq: inserting loops STEP 1}. 
    \eqref{eq: hitting moment} directly implies \eqref{eq: path to loop inseted} by observing that the loop $\beta^{(j)}$ has at least one step before hitting $\x_j$ and at least two steps after hitting $\x_j$.
\end{proof}

We now prove Lemma \ref{lem: identify the case Psi} with Lemma \ref{lem: inherit local information}. 
In fact, we can identify the type of $\x \in \R(\L_{\F}^{\ai}(\omega))$ from $\Dir(\overline{\omega},H_{\ai}(\overline{\omega},\x)+i)$, $i=0,\pm 1$, i.e.\ the red parts in Figure \ref{fig: loops inserted in the first step}. Recall the definition of the loop $\beta_{\x}(\omega)$ associated with the vertex $\x$ in \eqref{eq: loops to be inserted in STEP1}.

\begin{proof}[Proof of Lemma \ref{lem: identify the case Psi}]
    Suppose that $\x$ belongs to $\R(\L_{\F}^{\ai}(\omega))$ (which is determined by $\ov{\omega}$).
    From the definition of the loops $\beta(\x,\y)$ when the type of $\bs{\Delta}=\y-\x$ is (a), (b) and (c) (see also Figure \ref{fig: loops inserted in the first step}), the following hold (for simplicity, write $\beta=\beta_{\x}(\omega)$ and $H_0=H_{0}(\beta_{\x}(\omega),\x)$ for short):
    \begin{itemize}
        \item If the type of $\x$ is (a), we have $\Dir(\beta,H_{0}-1)>\Dir(\beta,H_{0})$ and $\Dir(\beta,H_{0})\leq \Dir(\beta,H_0+1)$.
        \item If the type of $\x$ is (b), we have $\Dir(\beta,H_{0}-1)<\Dir(\beta,H_{0})$.
        \item If the type of $\x$ is (c), we have $\Dir(\beta,H_0-1)>\Dir(\beta,H_0)$ and $\Dir(\beta,H_0) > \Dir(\beta,H_{0}+1)$.
    \end{itemize}
    Combining the observations above with \eqref{eq: path to loop inseted} and \eqref{eq: case iv)}, the type of $\x$ can be identified as follows (for simplicity, write $\ov{H}=H_{\ai}(\overline{\omega},\x)$ for short), for instance,
    \begin{itemize}
        \item If $\Dir(\overline{\omega},\ov{H}-1)>\Dir(\overline{\omega},\ov{H})$ and $\Dir(\overline{\omega},\ov{H})\leq \Dir(\overline{\omega},\ov{H}+1)$, then the type of $\x$ is (a).
    \end{itemize}
    Since the set $\R(\L_{\F}^{\ai}(\omega))$ and the criteria above are determined by $\ov{\omega}$, it follows that the type of each $\x \in \R(\L_{\F}^{\ai}(\omega))$ is also determined by $\ov{\omega}$.
\end{proof}

We now turn to the second step: identifying the loop $\beta_{\x}(\omega)$ under the assumption of the type of $\x$. 

\begin{lemma}
\label{lem: identify the loop with case Psi}
     For all $\overline{\omega}\in \ov{\Psi}(\CA)$, if $\omega \in \CA$ satisfies $\ov{\omega} \in \ov{\Psi}(\omega)$ and $\x \in \R(\L_{\F}^{\ai}(\omega))$, the loop $\beta_{\x}(\omega)$ is determined by $\overline{\omega}$. 
\end{lemma}

To prove Lemma \ref{lem: identify the loop with case Psi}, we provide an approach for identifying the loop $\beta_{\x}(\omega)$, as described in the following lemma.
For convenience, for each $\x \in \R(\L_{\F}^{\ai}(\omega))$ such that the type of $\x$ is not (d), we let 
\begin{equation}
    \Extra(\x)=
    \begin{cases}
        0, \quad \mbox{if the type of $\x$ is (a)};\\
        2, \quad \mbox{if the type of $\x$ is (b)};\\
        4, \quad \mbox{if the type of $\x$ is (c)}.\\
    \end{cases}
\end{equation}

 \begin{lemma}
 \label{lem: appch give the loop 22}
     For all $\overline{\omega}\in \ov{\Psi}(\CA)$ and $\omega \in \CA$ such that $\ov{\omega} \in \ov{\Psi}(\omega)$, and for all $\x \in \R(\L_{\F}^{\ai}(\omega))$, we have $\beta_{\x}(\omega)=\wt{\beta}_{\x}(\omega)$, where
     \begin{enumerate}[label=(\roman*)]
         \item If the type of $\x$ is (d), the loop $\wt{\beta}_{\x}(\omega)$ is given by $$(\overline{\omega}(H_{\ai}(\overline{\omega},\x)),\overline{\omega}(H_{\ai}(\overline{\omega},\x)+1),\overline{\omega}(H_{\ai}(\overline{\omega},\x)+2)).$$
         \item If the type of $\x$ is not (d), the loop $\wt{\beta}_{\x}(\omega)$ consists of the part of  $\overline{\omega}$ between times $H_{\ai}(\overline{\omega},\x)-k$ and $H_{\ai}(\overline{\omega},\x)+k+\Extra(\x)$, where $k$ is the smallest positive integer satisfying $H_{\ai}(\overline{\omega},\x)-k=H_{\ai}(\overline{\omega},\x)+k+\Extra(\x)$.
     \end{enumerate}
\end{lemma}

\begin{proof}[Proof of Lemma \ref{lem: identify the loop with case Psi} assuming Lemma \ref{lem: appch give the loop 22}]
    By Lemma \ref{lem: identify the case Psi}, the type of each $\x\in\R(\L_{\F}^{\ai}(\omega))$ is determined by $\ov{\omega}$.
    Thus, the loop $\wt{\beta}_{\x}(\omega)$, as defined in Lemma \ref{lem: appch give the loop 22}, is also determined by $\ov{\omega}$.
    Since Lemma \ref{lem: appch give the loop 22} establishes that $\beta_{\x}(\omega)=\wt{\beta}_{\x}(\omega)$, it follows that $\beta_{\x}(\omega)$ is determined by $\overline{\omega}$, which is exactly the claim of Lemma \ref{lem: identify the loop with case Psi}. 
\end{proof}

We now prove Lemma \ref{lem: appch give the loop 22}.
The statement in Lemma \ref{lem: appch give the loop 22}(i) follows directly from the definition of $B(\omega)$ and \eqref{eq: inserting loops STEP 1}.
We now focus on the case that the type of $\x$ is not (d). We refer readers to Figure \ref{fig: loops inserted in the first step} for reference.

\begin{proof}[Proof of Lemma \ref{lem: appch give the loop 22}(ii)]
    For all $\x \in \R(\L_{\F}^{\ai}(\omega))$ such that the type of $\x$ is not (d), observe that 
    \begin{equation*}
        H_{0}(\beta_{\x}(\omega),\x)=L(\beta_{\x}(\omega))-H_{0}(\beta_{\x}(\omega),\x)-\Extra(\x)\quad \mbox{for all }\x \in \R(\L_{\F}^{\ai}(\omega)).
    \end{equation*}
    Combining this with Lemma \ref{lem: loops no self intersecting} and \eqref{eq: hitting moment} establishes the statement in Lemma \ref{lem: appch give the loop 22} (ii).
\end{proof}

We now derive Lemma \ref{lem: identify the loops inserted Psi} from Lemmas \ref{lem: identify the late points Psi} and \ref{lem: identify the loop with case Psi}.

\begin{proof}[Proof of Lemma \ref{lem: identify the loops inserted Psi}]
    Suppose $\omega$ is a path in $\CA$ such that $\overline{\omega} \in \overline{\Psi}(\omega)$.
    From Lemma \ref{lem: identify the late points Psi}, it follows that $\R(\L_{\F}^{\ai}(\omega))$ is determined by $\ov{\omega}$.
    Additionally, by Lemma \ref{lem: identify the loop with case Psi}, we know that for all $\x \in \R(\L_{\F}^{\ai}(\omega))$,  $\beta_{\x}(\omega)$ is determined by $\ov{\omega}$.
    Therefore, the loop collection $B(\omega)$ (see \eqref{eq: loops to be inserted in STEP1} for definition) is determined by $\ov{\omega}$.
\end{proof}

We now turn to the proof of Proposition \ref{prop: recover first step}.
As in the proof of Lemma \ref{lem: recover second step}, we remove the loops in $B(\omega)$ using the mapping $\dltmap$.
To apply $\dltmap$, we first establish that the moments at which the loops in $B(\omega)$ are inserted are determined by $\ov{\omega}$.
Note that by Lemma \ref{lem: identify the late points Psi}, the quantity $|\R(\L_{\F}^{\ai}(\omega))|$ is determined by $\ov{\omega}$ .

\begin{lemma}
\label{lem: final prep for recover first step}
    For all $\overline{\omega}\in \ov{\Psi}(\CA)$ and $\omega \in \CA$ such that $\ov{\omega} \in \ov{\Psi}(\omega)$, the loop $\beta^{(j)}(\omega)$ and the moment $H_j(\omega)+S^{(j-1)}(\omega)$ are determined by $\ov{\omega}$ for all $1 \leq j \leq |\R(\L_{\F}^{\ai}(\omega))|$.
    In particular, $L(\beta^{(j)}(\omega))$ is determined by $\ov{\omega}$ for all $1 \leq j \leq |\R(\L_{\F}^{\ai}(\omega))|$.
\end{lemma}

\begin{proof}
    By Lemma \ref{lem: identify the late points Psi}, the set $\R(\L_{\F}^{\ai}(\omega))$ is determined by $\ov{\omega}$.
    Recall that $n=|\R(\L_{\F}^{\ai}(\omega))|$. 
    By \eqref{eq: loops mutually disjoint}, for each vertex $\x\in\R(\L_{\F}^{\ai}(\omega))$, there exists exactly one loop $\beta_{\x}(\omega)$ in $B(\omega)$ that contains $\x$.
    By Lemma \ref{lem: identify the loops inserted Psi}, the loop $\beta_{\x}(\omega)$ is determined by $\ov{\omega}$ for each $\x\in\R(\L_{\F}^{\ai}(\omega))$.

    Next, we show that $\beta^{(j)}(\omega)$ is determined by $\ov{\omega}$ for all $1 \leq j \leq n$.
    Recall that each vertex in $\R(\L_{\F}^{\ai}(\omega))$ is assigned a label (possibly depending on $\omega$) so that $H_j(\omega)=H_{\ai}(\omega,\near(\omega,\x_j))$ is increasing with respect to $j$ as described below \eqref{eq: loops to be inserted in STEP1}.
    By \eqref{eq: loops mutually disjoint}, $H_{\ai}(\overline{\omega},\x_l)$ are increasing with respect to $j$.
    Therefore, the loop $\beta^{(j)}(\omega)$ is also determined by $\ov{\omega}$ for all $1 \leq j \leq n$.

    Finally, we demonstrate that $H_j(\omega)+S^{(j-1)}(\omega)$ are determined by $\ov{\omega}$ for all $1 \leq j \leq n$.
    Since the insertion of $\beta^{(j')}(\omega)$ occurs before the path hitting $\x_j$ when $1 \leq j' \leq j-1$ (see \eqref{eq: inserting loops STEP 1}), the part of the path following $\beta^{(j')}$ shifts forward by $L(\beta^{(j')})$ steps. 
    Using \eqref{eq: loops mutually disjoint} and the definition of $S^{(j)}(\omega)$, we have 
    \begin{equation}
    \label{eq: changing hitting time earlier}
        H_{\ai}(\omega^{(j)},\near(\omega,\x_j))=H_j(\omega) +S^{(j-1)}(\omega).
    \end{equation}
    Using an argument similar to that used to derive \eqref{eq: hitting moment},  we obtain 
    \begin{equation}
    \label{eq: changing hitting time later}
         H_{\ai}(\ov{\omega},\near(\omega,\x_j))=H_{\ai}(\omega^{(j')},\near(\omega,\x_j))
    \end{equation}
    for all $j' \geq j$. Combining \eqref{eq: changing hitting time earlier} and  \eqref{eq: changing hitting time later} gives that $H_j(\omega)+S^{(j-1)}(\omega)$ are determined by $\ov{\omega}$ for all $1 \leq j \leq n$.
    Since $\beta^{(j)}(\omega)$ is determined by $\ov{\omega}$ for all $1 \leq j \leq n$, the starting point $\near(\omega,\x_j)$ of $\beta^{(j)}(\omega)$ is also determined by $\ov{\omega}$.
    Hence, $H_j(\omega) +S^{(j-1)}(\omega)$ is determined by $\ov{\omega}$.
\end{proof}

We now derive Proposition \ref{prop: recover first step} from Lemmas \ref{lem: identify the late points Psi} and \ref{lem: final prep for recover first step}.

\begin{proof}[Proof of Proposition \ref{prop: recover first step}]
    The existence of the path $\omega \in \CA$ such that $\ov{\omega} \in \Psi(\omega)$ is trivial. 
    
    We now focus on the uniqueness of such a path.
    By Lemma \ref{lem: identify the late points Psi}, $\R(\L_{\F}^{\ai}(\omega))$ is determined by $\ov{\omega}$.
    Therefore, $n=|\R(\L_{\F}^{\ai}(\omega))|$ is also determined by $\ov{\omega}$.
    Let $\omega$ be an arbitrary path such that $\ov{\omega} \in \Psi(\omega)$. 
    Consider the following deletion process.
    Define $\wt{\omega}^{(n)}=\ov{\omega}$, and recursively set
    \begin{equation}
    \label{eq: deletion process STEP 1}
        \wt{\omega}^{(j)}=\dltmap(\wt{\omega}^{(j+1)}(\omega),H_j(\omega)+S^{(j-1)}(\omega),H_j(\omega)+S^{(j-1)}(\omega)+L(\beta^{(j)}(\omega))), \;\, \mbox{for $0 \leq j \leq n-1$.}
    \end{equation}
    By induction, we can show that the mapping $\dltmap$ can be applied as described at each step $0 \leq j \leq n-1$, and that $\wt{\omega}^{(j)}$ is an \ext of $\omega^{(j)}$ when $0 \leq j \leq n-1$.
    Therefore, by Lemmas \ref{lem: identify the late points Psi} and \ref{lem: final prep for recover first step}, the paths $\wt{\omega}^{(j)}$ ($0 \leq j \leq n$) are determined by $\ov{\omega}$.
    Since all paths in $\CA$ are of length $T_3$, we obtain that
    \begin{equation}
        \wt{\omega}^{(0)}=\omega\mid_{T_3}.
    \end{equation}
    Now the uniqueness follows from the fact that $\wt{\omega}^{(0)}$ is determined by $\ov{\omega}$.
\end{proof}

\section{Tables of Symbols}
\label{sec: tables of symbols}
In this section, we gather some symbols from Sections \ref{sec: proofs of lower bounds} and \ref{sec: Construction of the mapping Psi and its properties} for reference. These symbols are presented in the following six tables, which compile symbols for Points and Subsets of $\T_N$, notation related to the Auxiliary Graph, Paths and Loops involved in the loop insertion process, Path and Loop Collections, Functions of Paths, and Mappings between Paths, respectively.

\begin{table}[h!]
\centering
\begin{tabular}{c c c} 
 \hline
Symbols & First appearance & Description \\ [0.5ex] 
 \hline
  $\BQ_\delta$ & \eqref{eq:Qdeltaboxes} & The bulk of the torus; its complement is denoted by $\HH_\delta$ \\
 $\bs{\Delta}$ & The beginning of Section \ref{sec: Construction of loops to be inserted} & The difference vector between two points. Also note the type of $\bs{\Delta}$ below the definition. \\
 $\near(\omega,\x)$ & The beginning of Section \ref{sec: Construction of loops to be inserted} & The vertex in the path segment $\omega(T_2,T_3]$ closest to $\x$\\ [1ex]
 \hline
\end{tabular}
\caption{Points and Subsets of $\T_N$}
\label{table:Torus}
\end{table}

\begin{table}[h!]
\centering
\begin{tabular}{c c c} 
 \hline
Symbols & First appearance & Description \\ [0.5ex] 
 \hline
 $G(F)$ & Above \eqref{eq:lNdef} & The auxiliary graph of $\T_N$'s subset $F$. \\ 
 $\C(F)$ & Below \eqref{eq:lNdef} & The collection of the connected components of $G(F)$  \\
 $\R(F)$ & Below \eqref{eq:lNdef} & The set of (representative) vertices of connected components in $G(F)$  \\ [1ex] 
 \hline
\end{tabular}
\caption{Auxiliary Graph}
\label{table:AuxGraph}
\end{table}

\begin{table}[h!]
\centering
\begin{tabular}{c c c} 
 \hline
Symbols & First appearance & Description \\ [0.5ex] 
 \hline
 $\beta(\x,\y)$ & Section \ref{sec: Construction of the mapping Psi and its properties}, below the proof of Proposition \ref{prop: loop insertion ineq} & A self-disjoint loop that connects two points $\x$ and $\y$.\\ 
 $\beta_i(\x)$ & \eqref{eq: beta-i} & $\beta_i(\x)= (\x, \x+e_i, \x)$.\\
 $\beta_{\x}(\omega)$ & Definition \ref{def: loops to be inserted in STEP1} & Loops inserted in the first step, constructed from $\beta(\x,\y)$ and $\beta_i(\x)$.\\
 $\beta^{(j)}(\omega)$ & Below \eqref{eq: number of roots} & The $j$-th loop inserted in the first step.\\
 $\omega^{(j)}$ & \eqref{eq: inserting loops STEP 1} & The path obtained by inserting $j$ loops in the first step ($1 \leq j \leq n$).\\
 $\beta^{\mathrm{Tree}}(\mathsf{C},\x)$ & Definition \ref{def: loops to be inserted in STEP2} & A loop that starts at $\x$ visits every vertex in $\mathsf{C}$.\\
 $\beta^{(j),\text{Tree}}$ & Below \eqref{eq: loops to be inserted in STEP2} & The $j$-th loop inserted in the second step, constructed from $\beta^{\mathrm{Tree}}(\mathsf{C},\x)$.\\
 $\omega^{(n+j)}$ & \eqref{eq: inserting loops STEP 2} & The path obtained by inserting $j$ loops in the second step ($1 \leq j \leq n$).\\ [1ex]
 \hline
\end{tabular}
\caption{Paths and Loops}
\label{table:PL}
\end{table}

\begin{table}[h!]
\centering
\begin{tabular}{c c c} 
 \hline
Symbols & First appearance & Description \\ [0.5ex] 
 \hline
 $\Omega_{t}$, $\Omega_{x,t}$ & \eqref{eq: random walk paths} and below & The collection of all random walk paths up to time $t$ (resp. with starting point $\x$).\\ 
 $B(\omega)$  & Definition \ref{def: loops to be inserted in STEP1} & The collection of loops to be inserted in the first step.\\
 $B^{\Tree}(\omega)$ & \eqref{eq: loops to be inserted in STEP2} & The collection of loops to be inserted in the second step.  \\
 $\Psi(\omega)$ & \eqref{eq: def Psi} & The collection of ``good paths" constructed through loop insertion.\\
 $\overline{\Psi}(\omega)$ & Above Proposition \ref{prop: recover first step} & The collection of extensions of $\omega^{(n)}$.\\ [1ex] 
 \hline
\end{tabular}
\caption{Paths and Loops Collections}
\label{table:PLC}
\end{table}

\begin{table}[h!]
\centering
\begin{tabular}{c c c} 
 \hline
Symbols & First appearance & Description \\ [0.5ex] 
 \hline
 $L(\cdot)$ & Below \eqref{eq: random walk paths} & The length of a path.\\ 
 $H_{\lambda}(\omega,\z)$  & \eqref{eq: def: hitting time after lambda} & The first hitting time of path $\omega$ after time $\lambda t_{\cov}$.\\
 $\Dir(\omega,k)$ &Definition \ref{def: loops to be inserted in STEP1} & The direction of the $k$-th step of the path $\omega$.\\
 $S^{(j)}(\omega)$ & \eqref{eq: inserting loops STEP 1} & The total length of the first $j$ loops inserted in the first step. \\
 $S^{(j),\text{Tree}}(\omega)$ & \eqref{eq: inserting loops STEP 2} & The total length of the first $j$ loops inserted in the first step.\\ [1ex] 
 \hline
\end{tabular}
\caption{Functions of Paths}
\label{table:func-of-path}
\end{table}

\begin{table}[h!]
\centering
\begin{tabular}{c c c} 
 \hline
Symbols & First appearance & Description \\ [0.5ex] 
 \hline
 $\insmap(\omega,\beta,k)$ & Above \eqref{eq: inserting loops STEP 1}& The insertion map.\\ 
 $\dltmap(\omega,k_1,k_2)$  & Below the proof of Lemma \ref{lem: upper bound for path with loops} & The deletion map. \\
 \emph{extensions} of $\omega$ & Below the proof of Lemma \ref{lem: upper bound for path with loops} & Paths with the first $L(\omega)$ steps the same as $\omega$.\\ [1ex] 
 \hline
\end{tabular}
\caption{Mappings between Paths}
\label{table:maps}
\end{table}

\begin{appendix}
\section{Proof of (\ref{eq:rough upper bound})}
\label{apx: proof of rough upper bound}

In this appendix, we sketch the proof of the rough upper bound \eqref{eq:rough upper bound}. Here we follow an argument similar to that of the upper bound in \cite[Theorem 1.1]{2dLDP} instead of modifying \cite{Goodman2014}.

We start with using the soft local times method to decouple the traces left by random walk in disjoint cubes. 
Let $A_1,A_2,\ldots,A_{k_0}, A'_1,A'_2,\ldots,A'_{k_0}$ be subsets of $\T_N$ such that $A_j \subset A'_j$ and $A_j \cap \partial A'_j = \emptyset$ for all $j=1,2,\ldots, k_0$ and $A_i' \cap A_j' = \emptyset$ for all $i \neq j$.
We write $A=\cup_{j=1}^{k_0}A_j$ and $A'=\cup_{j=1}^{k_0}A'_j$.
We denote by $Z_i^{(j)}$  the $i$-th excursion of random walk $(X_k)_{k \geq 0}$ from $\partial A_j$ to $\partial A_j'$.
For two subsets $F \subset F' \subset \T_N$ such that $F \cap \partial F'=\emptyset$ and a probability measure $\nu(\cdot)$ supported on $\partial F$, let $\mathbf{P}^{(\nu)}$ be the law of the first excursion from $\partial F$ to $\partial F'$ of random walk $(X_k)_{k \geq 0}$ under the measure $\P_{\nu}(\cdot)$.

The same argument\footnote{Note that although \cite[Lemma 2.1]{2dLDP} only deals with the 2D torus, the same proof goes through without modification.} as \cite[Lemma 2.1]{2dLDP} gives the following coupling between excursions of a random walk on torus and i.i.d.\ excursions.

\begin{lemma} 
\label{lem: SFT decouple lemma}
    Assume that $\widetilde{e}_j(\cdot)$ ($j=1,2,\ldots,k_0$) is a probability measure supported on $\partial A_j$ satisfying that 
    \begin{equation}
    \label{eq: small fluctuation of hitting measure}
        1-\frac{v}{3} \leq \frac{\P_y(X_{H_{A}}=x|X_{H_{A}}\in A_j)}{\widetilde{e}_j(x)} \leq 1+\frac{v}{3}
    \end{equation}
    for all $x \in \partial A_j$ and $y \in \partial A'_j$.
    Then, there exists a probability measure $\mathbf{P}$ extending $\mathsf{\P}$ and a family of events $(G_{j}^{m_0})_{1 \leq j \leq k_0}$ independent with random walk $(X_k)_{k \geq 0}$ such that
    \begin{itemize}
        \item $(G_{j}^{m_0})_{1 \leq j \leq k_0}$ are independent from each other.
        \item $\P(G_{j}^{m_0}) \geq 1-C \exp(-cvm_0)$ for all $j=1,2,\ldots,k_0$.
        \item For all $j=1,2,\ldots, k_0$ and $m>m_0$, on the event $U_{j}^{m_0}$, we have
        \begin{equation}
            \begin{aligned}
                \{Z_1^{(j)},Z_2^{(j)},\ldots, Z_{(1-v)m}^{(j)} \} &\subset \{\widetilde{Z}_1^{(j)},\widetilde{Z}_2^{(j)},\ldots, \widetilde{Z}_{(1+3v)m}^{(j)} \},\mbox{ and}\\
                \{\widetilde{Z}_1^{(j)},\widetilde{Z}_2^{(j)},\ldots, \widetilde{Z}_{(1-v)m}^{(j)} \} &\subset 
                \{Z_1^{(j)},Z_2^{(j)},\ldots, Z_{(1+3v)m}^{(j)} \},
            \end{aligned}
        \end{equation}
        where $(\widetilde{Z}_{i}^{(j)})_{i \geq 0}$ is a sequence of i.i.d.\ excursions with law $\mathbf{P}^{(\widetilde{e}_j)}$.
        Moreover, $(\widetilde{Z}_{i}^{(j)})_{i \geq 0}$ and $(\widetilde{Z}_{i}^{(j')})_{i \geq 0}$ are independent if $j \neq j'$.
    \end{itemize}
\end{lemma}
The events $G_{j}^{m_0}$ are defined explicitly in \cite[(2.6)]{2dLDP}, where they denote by $U_j^{m_0}$, but here we only need their properties.

We now use Lemma \ref{lem: SFT decouple lemma} to prove \eqref{eq:rough upper bound}.
Let $0<\gamma<\gamma_4<\gamma_3<\gamma_2<\gamma_1<1$.
Set
\begin{equation}
    s_N=N^{\gamma_2},\quad s_N'=\frac{N}{\lfloor N^{(1- \gamma_1)} \rfloor},\mbox{ and }\quad k_N= \lfloor N^{(1- \gamma_1)} \rfloor^d.
\end{equation}
We tile the (continuous) torus $\mathbb{R}_N \overset{\text{def.}}{=} \mathbb{R}^d / N \Z^d$ with $k_N$ cubes of side-length $s_n'$, and choose $x_j$ as the site in $\T_N$ closest to the center of the $j$-th cube.
Let 
\begin{equation}
\label{eq: boxes considered}
    A_j=B(x_j, s_N), \quad  A'_j=B(x_j,bs_N') 
\end{equation}
for some small $b>0$ (to be fixed at \eqref{eq: concentration of total number excursions}), where $B(x,r)$ is the Euclidean ball with center $x$ and radius $r$.
For each $v \in (0,1)$, there exists $N_0>0$ such that for all $N>N_0$, the inequality \eqref{eq: small fluctuation of hitting measure} holds for $\widetilde{e}_j(\cdot)=\overline{e}_{A_j}(\cdot)$ for any $j=1,2,\ldots,k_N$; see \cite[Proposition 1.5]{Szn17} for a similar result.
We let $v>0$ sufficiently small such that $(1+4v)\gamma_3<\gamma_2$ and set
\begin{equation}
    m_1=\gamma_3 g(0) \log N^d \capacity(B(0, s_N));\quad m_2=(1+4v)\gamma_3 g(0) \log N^d \capacity(B(0, s_N)).
\end{equation}
We define the events
\begin{equation}
\label{eq: def: one box event}
    \Lambda_1^{(j)}=G_j^{m_1},\quad \Lambda_2^{(j)}=\big\{\mbox{there exists $\x \in A_j$ such that $\x \notin \widetilde{Z}_j$ for all $j=1,2,\ldots,m_2$}\big\}
\end{equation}
for each $j=1,2,\ldots,k_N$.
By coupling i.i.d.\ excursions $\widetilde{Z}^{(j)}_i$ with excursions of random interlacements, we can show $\Lambda_2^{(j)}$ occurs with high probability.
Combining this with Lemma \ref{lem: SFT decouple lemma} gives
\begin{equation}
\label{eq: cover failure whp}
    \mathbf{P}(\Lambda^{(j)}_1 \cap \Lambda^{(j)}_2) \to 1.
\end{equation}
Let $\Lambda_3=\{\sum_{j=1}^{k_N} 1\{\Lambda^{(j)}_1 \cap \Lambda^{(j)}_2\} \geq \widetilde{\lambda} k_N+1\}$, where $\widetilde{\lambda}\in (\frac{\gamma_4}{\gamma_3},1)$.
Note that $\Lambda^{(j)}_1 \cap \Lambda^{(j)}_2$ ($j=1,2,\ldots,k_N$) are independent, so by \eqref{eq: cover failure whp} and a standard large deviation bound, we obtain 
\begin{equation}
\label{eq: prob: abnormal coverage by excursions}
    \mathbf{P}\left(\Lambda_3\right) \geq 1-\exp\big(-cN^{d(1-\gamma_1)}\big).
\end{equation}
Let $\zeta_j$ denote the number of excursions of random walk $(X_k)_{k \geq 0}$ between $\partial A_j$ and $\partial A_j'$ up to $\gamma t_{\cov}$ and $\zeta=\sum_{j=1}^{k_N}\zeta_j$.
By Lemma \ref{lem: SFT decouple lemma} and \eqref{eq: def: one box event}, on the event $\Lambda_1^{(j)} \cap \Lambda_2^{(j)}$, the cube $A_j$ is not completely covered by the first $m_1$ excursions of random walk $(X_k)_{k \geq 0}$ between $\partial A_j$ and $\partial A_j'$.
Therefore, on the event $U_{\gamma,N} \cap \Lambda_3$, there exist at least $\widetilde{\lambda} k_N$ cubes among $A_1,A_2,\ldots,A_{k_N}$ such that $\zeta_j \geq m_1$, which implies:
\begin{equation}
\label{eq: lb for num of excursions}
    \mbox{On the event $U_{\gamma,N} \cap \Lambda_3$, we have } \zeta \geq \widetilde{\lambda}k_N m_1=\widetilde{\lambda}\gamma_3 k_N \cdot g(0) \log N^d \capacity(B(0, s_N)).
\end{equation}

Furthermore, by replacing \cite[Lemma 2.1]{Dembo2004} with \cite[Lemma 2.2]{3dBMconcentration} and imitating the proof of \cite[Lemma 3.4]{2dLDP}, by choosing $b$ in \eqref{eq: boxes considered} sufficiently small\footnote{We will use \cite[Lemma 2.2]{3dBMconcentration} with $R=2b$, which requires $b$ to be sufficiently small so that the term $\eta(R)$ in \cite[(2.3))]{3dBMconcentration} remains sufficiently small; see also \cite[Lemma 3.4]{3dBMconcentration} and its proof for reference.}, one arrives at 
\begin{equation}
\label{eq: concentration of total number excursions}
    \P\left(\Lambda^c\right) \leq \exp \big(-c k_N \cdot g(0) \log N^d \capacity(B(0, s_N)) \big)=\exp(-k_N\cdot N^{(d-2)\gamma_2+o(1)}),
\end{equation}
where 
$$
\Lambda=\Big\{\zeta \leq \gamma_4 k_N \cdot g(0) \log N^d \capacity(B(0, s_N))\Big\}.
$$
Recall that $\widetilde{\lambda}>\frac{\gamma_4}{\gamma_3}$ and $m_1=\gamma_3 g(0) \log N^d \capacity(B(0, s_N))$, so by \eqref{eq: lb for num of excursions}, one has $U_{\gamma,N} \subset \Lambda_3^c \cup \Lambda^c$.
Combining \eqref{eq: prob: abnormal coverage by excursions}, \eqref{eq: concentration of total number excursions} and a union bound gives \eqref{eq:rough upper bound}.

\section{Proof of the upper bound of (\ref{eq: RI large deviation})}
\label{apx: proof of RI upper bound}
We show here the upper bound of \eqref{eq: RI large deviation} for all $\gamma \in (\frac{2}{d},1)$ using a bootstrap argument. Let 
\begin{equation}
    \label{eq: phase transition point}
    \gamma_*=\inf\left\{\overline{\gamma}\in (0,1]: \lim_{N\to\infty}\sup_{\gamma\in [\overline{\gamma},1]}
    \frac{\log\bbP(\frM_N\leq\gamma u_N)}{N^{d(1-\gamma)}}\leq -1\right\}.
\end{equation}
It follows from \cite[Theorem 0.1]{Belcoverlevels} (by taking $A=Q(0,N)$ and $z=0$ therein) that 
\begin{equation*}
\lim_{N\to\infty}\log\bbP(\frM_N\leq u_N) = -1, 
\end{equation*}
which implies the finiteness of $\gamma_*$. 
It now suffices to prove $\gamma_*\leq \frac{2}{d}$. 

Pick $\gamma, a, \delta\in (0,1)$. 
We set\footnote{The notation defined here is similar to that in Section \ref{sec: proof of upper bound} except that we replace $\gamma$ with $a$.}
\begin{equation*}
    R_N=\lfloor N^a \rfloor, \quad s_N=\lceil(1+\delta)R_N \rceil+1, \text{ and} \quad K_N=Q\big(0,\tfrac{N}{1+\delta}\big)\cap(s_N \Z)^d. 
\end{equation*}
It is clear that for large $N$, $K_N$ is $(1+\delta)R_N$-well separated, and $Q(x,R_N)\subset Q(0,N)$ for all $x\in K_N$. 
Thus, we can apply Proposition \ref{prop: RW/RI to independent RI coupling} with $R=R_N$, $F=K_N$, $u=\gamma u_N$ and $\rho=\rho_N=(\log N)^{-2}$. More precisely, writing $v_N=\gamma u_N(1+\rho_N)$, by \eqref{eq: RI to independent RI}, we get
\begin{equation}
\label{eq: bootstrap step}
\begin{aligned}
    \bbP(\frM_N\leq\gamma u_N) &\,\leq\; \bbP(Q(x,R_N)\subset \I^{\gamma u_N},\,\forall\, x\in K_N) \\
    \overset{\eqref{eq: RI to independent RI}}&{\leq} 
    \widetilde{\bbP}\big(Q(x,R_N)\subset \I^{(x),v_N}, \,\forall\, x \in K_N\big) + CN^{3d}\exp\big(-c\gamma N^{a(d-2)}(\log N)^{-3}\big)
\end{aligned}
\end{equation}
(recall that $C$ and $c$ depend only on $\delta$ and $d$). 
By the independence of $\I^{(x),v_N}$ ($x \in K_N$) and the translation invariance of random interlacements, 
\begin{equation}
\label{eq: boxes covered by independent RI}
    \widetilde{\bbP}\big(Q(x,R_N)\subset \I^{(x),v_N}, \,\forall\,x \in K_N\big) =
    \bbP\big(Q(0,R_N)\subset \I^{v_N}\big)^{|K_N|}
    =\bbP\big(\frM_{R_N}\leq v_N\big)^{|K_N|}.
\end{equation}
A simple calculation gives $v_N=f(\gamma, a, N)\cdot u_{R_N}$, where 
\begin{equation*}
\label{eq: bootstrap parameter}
    f(\gamma, a, N)=\gamma (1+\rho_N)\frac{\log N}{\log R_N}\to \frac{\gamma}{a},\text{\quad as }N\to\infty.
\end{equation*}
Let $\gamma_*<\lambda_1<\lambda_2<1$. We now fix $a\in(0,1)$ and take $\gamma\in[a\lambda_1,a\lambda_2]$.
By the definition of $\gamma_*$, 
\begin{equation}
\label{eq: assumption}
    \lim_{N\to\infty}\sup_{\frac{\gamma}{a}\in[\lambda_1,\lambda_2]}\frac{\log\bbP\big(\frM_{R_N}\leq v_N\big)}{(R_N)^{d(1-f(\gamma,a,N))}}\leq -1. 
\end{equation}
Through a careful calculation, \eqref{eq: assumption} implies
\begin{equation}
\label{eq: prior estimate}
     \lim_{N\to\infty}\sup_{\gamma\in[a\lambda_1,a\lambda_2]}\frac{\log\bbP\big(\frM_{R_N}\leq v_N\big)}{N^{d(a-\gamma)}}\leq -1. 
\end{equation}
Moreover, by the definition of $K_N$, 
\begin{equation}
\label{eq: number of boxes}
    |K_N| \sim \bigg(\frac{N^{1-a}}{(1+\delta)^2}\bigg)^d, 
    \quad N\to\infty.
\end{equation}
We now assume that $d(1-a\lambda_1)<a(d-2)$. By combining \eqref{eq: bootstrap step}, \eqref{eq: boxes covered by independent RI}, \eqref{eq: prior estimate} and \eqref{eq: number of boxes}, and sending $\delta\to 0$, we obtain
\begin{equation}
\label{eq: bootstrap result}
    \lim_{N\to\infty}\sup_{\gamma\in [a\lambda_1, a\lambda_2]}
    \frac{\log\bbP(\frM_N\leq\gamma u_N)}{N^{d(1-\gamma)}}\leq -1. 
\end{equation}
We send $a,\lambda_2\to 1^-$ and $\lambda_1\to\gamma_*^+$ such that $a\lambda_1<\gamma_*<a\lambda_2$. 
However, \eqref{eq: phase transition point} and \eqref{eq: bootstrap result} directly imply $a\lambda_1\geq \gamma_*$ (when $a\lambda_2>\gamma_*$), which leads to a contradiction. 
Thus, we have $d(1-\gamma_*)\geq d-2$, i.e., $\gamma_*\leq \frac{2}{d}$. 
\end{appendix}

%
%

\begin{acks}[Acknowledgments]
We would like to thank the anonymous referee for valuable comments and suggestions, which greatly helped the revision of this work. 
\end{acks}
\begin{funding}
%
XL thanks the support of National Key R\&D Program of China (No.\ 2021YFA1002700 and No.\ 2020YFA0712900) and NSFC (No.\ 12071012).
\end{funding}


\bibliographystyle{imsart-number} 
\bibliography{reference}       

\begin{thebibliography}{32}

\bibitem{Abe2021}
\begin{barticle}[author]
\bauthor{\bsnm{Abe},~\bfnm{Yoshihiro}\binits{Y.}}
(\byear{2021}).
\btitle{Second-Order Term of Cover Time for Planar Simple Random Walk}.
\bjournal{J.\ Theoret.\ Probab.}
\bvolume{34}
\bpages{1689-1747}.
\bdoi{10.1007/s10959-020-01011-2}
\end{barticle}
\endbibitem

\bibitem{abe2021avoided}
\begin{bincollection}[author]
\bauthor{\bsnm{Abe},~\bfnm{Yoshihiro}\binits{Y.}}
(\byear{2021}).
\btitle{Avoided points of two-dimensional random walks}.
In \bbooktitle{Stochastic Analysis, Random Fields and Integrable Probability—Fukuoka 2019},
\bvolume{87}
\bpages{199--212}.
\bpublisher{Mathematical Society of Japan}.
\end{bincollection}
\endbibitem

\bibitem{Aldous2014}
\begin{bmisc}[author]
\bauthor{\bsnm{Aldous},~\bfnm{David}\binits{D.}} \AND \bauthor{\bsnm{Fill},~\bfnm{James~Allen}\binits{J.~A.}}
(\byear{2002}).
\btitle{Reversible Markov Chains and Random Walks on Graphs}.
\bnote{Unfinished monograph, recompiled 2014, available at \url{http://www.stat.berkeley.edu/~aldous/RWG/book.html}}.
\end{bmisc}
\endbibitem

\bibitem{belius2011}
\begin{barticle}[author]
\bauthor{\bsnm{Belius},~\bfnm{David}\binits{D.}}
(\byear{2011}).
\btitle{Cover times in the discrete cylinder}.
\bjournal{arXiv preprint arXiv:1103.2079}.
\end{barticle}
\endbibitem

\bibitem{Belcoverlevels}
\begin{barticle}[author]
\bauthor{\bsnm{Belius},~\bfnm{David}\binits{D.}}
(\byear{2012}).
\btitle{{Cover levels and random interlacements}}.
\bjournal{Ann.\ Appl.\ Probab.}
\bvolume{22}
\bpages{522--540}.
\bdoi{10.1214/11-AAP770}
\end{barticle}
\endbibitem

\bibitem{Bel13}
\begin{barticle}[author]
\bauthor{\bsnm{Belius},~\bfnm{David}\binits{D.}}
(\byear{2013}).
\btitle{Gumbel fluctuations for cover times in the discrete torus}.
\bjournal{Probab.\ Theory Related Fields}
\bvolume{157}
\bpages{635--689}.
\bdoi{10.1007/s00440-012-0467-7}
\bmrnumber{3129800}
\end{barticle}
\endbibitem

\bibitem{Belius2017}
\begin{barticle}[author]
\bauthor{\bsnm{Belius},~\bfnm{David}\binits{D.}} \AND \bauthor{\bsnm{Kistler},~\bfnm{Nicola}\binits{N.}}
(\byear{2017}).
\btitle{The subleading order of two dimensional cover times}.
\bjournal{Probab.\ Theory Related Fields}
\bvolume{167}
\bpages{461-552}.
\bdoi{10.1007/s00440-015-0689-6}
\end{barticle}
\endbibitem

\bibitem{belius2020tightness}
\begin{barticle}[author]
\bauthor{\bsnm{Belius},~\bfnm{David}\binits{D.}}, \bauthor{\bsnm{Rosen},~\bfnm{Jay}\binits{J.}} \AND \bauthor{\bsnm{Zeitouni},~\bfnm{Ofer}\binits{O.}}
(\byear{2020}).
\btitle{Tightness for the cover time of the two dimensional sphere}.
\bjournal{Probab.\ Theory Related Fields}
\bvolume{176}
\bpages{1357--1437}.
\end{barticle}
\endbibitem

\bibitem{benjamini2013linear}
\begin{barticle}[author]
\bauthor{\bsnm{Benjamini},~\bfnm{Itai}\binits{I.}}, \bauthor{\bsnm{Gurel-Gurevich},~\bfnm{Ori}\binits{O.}} \AND \bauthor{\bsnm{Morris},~\bfnm{Ben}\binits{B.}}
(\byear{2013}).
\btitle{Linear cover time is exponentially unlikely}.
\bjournal{Probab.\ Theory Related Fields}
\bvolume{155}
\bpages{451--461}.
\end{barticle}
\endbibitem

\bibitem{bolthausen2001entropic}
\begin{barticle}[author]
\bauthor{\bsnm{Bolthausen},~\bfnm{Erwin}\binits{E.}}, \bauthor{\bsnm{Deuschel},~\bfnm{Jean-Dominique}\binits{J.-D.}} \AND \bauthor{\bsnm{Giacomin},~\bfnm{Giambattista}\binits{G.}}
(\byear{2001}).
\btitle{Entropic repulsion and the maximum of the two-dimensional harmonic crystal}.
\bjournal{Ann.\ Probab.}
\bvolume{29}
\bpages{1670--1692}.
\end{barticle}
\endbibitem

\bibitem{BDZ95}
\begin{barticle}[author]
\bauthor{\bsnm{Bolthausen},~\bfnm{Erwin}\binits{E.}}, \bauthor{\bsnm{Deuschel},~\bfnm{Jean-Dominique}\binits{J.-D.}} \AND \bauthor{\bsnm{Zeitouni},~\bfnm{Ofer}\binits{O.}}
(\byear{1995}).
\btitle{Entropic repulsion of the lattice free field}.
\bjournal{Comm.\ Math.\ Phys.}
\bvolume{170}
\bpages{417--443}.
\bmrnumber{1334403}
\end{barticle}
\endbibitem

\bibitem{2dLDP}
\begin{barticle}[author]
\bauthor{\bsnm{Comets},~\bfnm{Francis}\binits{F.}}, \bauthor{\bsnm{Gallesco},~\bfnm{Christophe}\binits{C.}}, \bauthor{\bsnm{Popov},~\bfnm{Serguei}\binits{S.}} \AND \bauthor{\bsnm{Vachkovskaia},~\bfnm{Marina}\binits{M.}}
(\byear{2013}).
\btitle{{On large deviations for the cover time of two-dimensional torus}}.
\bjournal{Electron.\ J.\ Probab.}
\bvolume{18}
\bpages{1--18}.
\bdoi{10.1214/EJP.v18-2856}
\end{barticle}
\endbibitem

\bibitem{3dBMconcentration}
\begin{barticle}[author]
\bauthor{\bsnm{Dembo},~\bfnm{Amir}\binits{A.}}, \bauthor{\bsnm{Peres},~\bfnm{Yuval}\binits{Y.}} \AND \bauthor{\bsnm{Rosen},~\bfnm{Jay}\binits{J.}}
(\byear{2003}).
\btitle{{Brownian Motion on Compact Manifolds: Cover Time and Late Points}}.
\bjournal{Electron.\ J.\ Probab.}
\bvolume{8}
\bpages{1--14}.
\bdoi{10.1214/EJP.v8-139}
\end{barticle}
\endbibitem

\bibitem{Dembo2004}
\begin{barticle}[author]
\bauthor{\bsnm{Dembo},~\bfnm{Amir}\binits{A.}}, \bauthor{\bsnm{Peres},~\bfnm{Yuval}\binits{Y.}}, \bauthor{\bsnm{Rosen},~\bfnm{Jay}\binits{J.}} \AND \bauthor{\bsnm{Zeitouni},~\bfnm{Ofer}\binits{O.}}
(\byear{2004}).
\btitle{Cover times for {B}rownian motion and random walks in two dimensions}.
\bjournal{Ann.\ of Math.\ (2)}
\bvolume{160}
\bpages{433--464}.
\bdoi{10.4007/annals.2004.160.433}
\end{barticle}
\endbibitem

\bibitem{DRS14}
\begin{bbook}[author]
\bauthor{\bsnm{Drewitz},~\bfnm{Alexander}\binits{A.}}, \bauthor{\bsnm{R\'ath},~\bfnm{Bal\'azs}\binits{B.}} \AND \bauthor{\bsnm{Sapozhnikov},~\bfnm{Art\"em}\binits{A.}}
(\byear{2014}).
\btitle{An introduction to random interlacements}.
\bseries{SpringerBriefs in Mathematics}.
\bpublisher{Springer, Cham}.
\bdoi{10.1007/978-3-319-05852-8}
\bmrnumber{3308116}
\end{bbook}
\endbibitem

\bibitem{dubroff2021linearcovertimeexponentially}
\begin{barticle}[author]
\bauthor{\bsnm{Dubroff},~\bfnm{Quentin}\binits{Q.}} \AND \bauthor{\bsnm{Kahn},~\bfnm{Jeff}\binits{J.}}
(\byear{2025}).
\btitle{Linear cover time is exponentially unlikely}.
\bjournal{Ann. \ Probab.}
\bvolume{53}
\bpages{1--22}.
\end{barticle}
\endbibitem

\bibitem{erdos1960some}
\begin{barticle}[author]
\bauthor{\bsnm{Erd\H{o}s},~\bfnm{Paul}\binits{P.}} \AND \bauthor{\bsnm{Taylor},~\bfnm{S.~James}\binits{S.~J.}}
(\byear{1960}).
\btitle{Some problems concerning the structure of random walk paths}.
\bjournal{Acta Math.\ Acad.\ Sci.\ Hungar}
\bvolume{11}
\bpages{137--162}.
\end{barticle}
\endbibitem

\bibitem{Goodman2014}
\begin{barticle}[author]
\bauthor{\bsnm{Goodman},~\bfnm{Jesse}\binits{J.}} \AND \bauthor{\bparticle{den} \bsnm{Hollander},~\bfnm{Frank}\binits{F.}}
(\byear{2014}).
\btitle{Extremal geometry of a {B}rownian porous medium}.
\bjournal{Probab.\ Theory Related Fields}
\bvolume{160}
\bpages{127-174}.
\bdoi{10.1007/s00440-013-0525-9}
\end{barticle}
\endbibitem

\bibitem{lawler2012intersections}
\begin{bbook}[author]
\bauthor{\bsnm{Lawler},~\bfnm{Gregory~F.}\binits{G.~F.}}
(\byear{2012}).
\btitle{Intersections of Random Walks}.
\bseries{Modern Birkh{\"a}user Classics}.
\bpublisher{Springer New York}.
\end{bbook}
\endbibitem

\bibitem{LP17}
\begin{bbook}[author]
\bauthor{\bsnm{Levin},~\bfnm{David~A.}\binits{D.~A.}}, \bauthor{\bsnm{Peres},~\bfnm{Yuval}\binits{Y.}} \AND \bauthor{\bsnm{Wilmer},~\bfnm{Elizabeth~L.}\binits{E.~L.}}
(\byear{2017}).
\btitle{Markov chains and mixing times},
\bedition{Second} ed.
\bpublisher{American Mathematical Society, Providence, RI}
\bnote{With a chapter on ``Coupling from the past'' by James G.\ Propp and David B.\ Wilson}.
\bdoi{10.1090/mbk/107}
\bmrnumber{3726904}
\end{bbook}
\endbibitem

\bibitem{li2017lower}
\begin{barticle}[author]
\bauthor{\bsnm{Li},~\bfnm{Xinyi}\binits{X.}}
(\byear{2017}).
\btitle{A lower bound for disconnection by simple random walk}.
\bjournal{Ann.\ Probab.}
\bvolume{45}
\bpages{879--931}.
\end{barticle}
\endbibitem

\bibitem{LiShi2025}
\begin{barticle}[author]
\bauthor{\bsnm{Li},~\bfnm{Xinyi}\binits{X.}} \AND \bauthor{\bsnm{Shi},~\bfnm{Jialu}\binits{J.}}
\btitle{Large deviations of the cover level of random interlacements}.
\bjournal{In preparation}.
\end{barticle}
\endbibitem

\bibitem{LiSzn14}
\begin{barticle}[author]
\bauthor{\bsnm{Li},~\bfnm{Xinyi}\binits{X.}} \AND \bauthor{\bsnm{Sznitman},~\bfnm{Alain-Sol}\binits{A.-S.}}
(\byear{2014}).
\btitle{{A lower bound for disconnection by random interlacements}}.
\bjournal{Electron.\ J.\ Probab.}
\bvolume{19}
\bpages{1--26}.
\bdoi{10.1214/EJP.v19-3067}
\end{barticle}
\endbibitem

\bibitem{miller2017uniformity}
\begin{barticle}[author]
\bauthor{\bsnm{Miller},~\bfnm{Jason}\binits{J.}} \AND \bauthor{\bsnm{Sousi},~\bfnm{Perla}\binits{P.}}
(\byear{2017}).
\btitle{Uniformity of the late points of random walk on $\mathbb{Z}_n^d$ for $d\geq 3$}.
\bjournal{Probab.\ Theory Related Fields}
\bvolume{167}
\bpages{1001--1056}.
\end{barticle}
\endbibitem

\bibitem{PT15}
\begin{barticle}[author]
\bauthor{\bsnm{Popov},~\bfnm{Serguei}\binits{S.}} \AND \bauthor{\bsnm{Teixeira},~\bfnm{Augusto}\binits{A.}}
(\byear{2015}).
\btitle{Soft local times and decoupling of random interlacements}.
\bjournal{J.\ Eur.\ Math.\ Soc.}
\bvolume{17}
\bpages{2545--2593}.
\bdoi{10.4171/JEMS/565}
\bmrnumber{3420516}
\end{barticle}
\endbibitem

\bibitem{prevost2023phase}
\begin{barticle}[author]
\bauthor{\bsnm{Pr{\'e}vost},~\bfnm{Alexis}\binits{A.}}, \bauthor{\bsnm{Rodriguez},~\bfnm{Pierre-Fran{\c{c}}ois}\binits{P.-F.}} \AND \bauthor{\bsnm{Sousi},~\bfnm{Perla}\binits{P.}}
(\byear{2023}).
\btitle{Phase transition for the late points of random walk}.
\bjournal{arXiv preprint arXiv:2309.03192}.
\end{barticle}
\endbibitem

\bibitem{Szn10}
\begin{barticle}[author]
\bauthor{\bsnm{Sznitman},~\bfnm{Alain-Sol}\binits{A.-S.}}
(\byear{2010}).
\btitle{Vacant set of random interlacements and percolation}.
\bjournal{Ann.\ of Math.\ (2)}
\bvolume{171}
\bpages{2039--2087}.
\bdoi{10.4007/annals.2010.171.2039}
\bmrnumber{2680403}
\end{barticle}
\endbibitem

\bibitem{Szn17}
\begin{barticle}[author]
\bauthor{\bsnm{Sznitman},~\bfnm{Alain-Sol}\binits{A.-S.}}
(\byear{2017}).
\btitle{Disconnection, random walks, and random interlacements}.
\bjournal{Probab.\ Theory Related Fields}
\bvolume{167}
\bpages{1--44}.
\bdoi{10.1007/s00440-015-0676-y}
\bmrnumber{3602841}
\end{barticle}
\endbibitem

\bibitem{Teixeira2009}
\begin{barticle}[author]
\bauthor{\bsnm{Teixeira},~\bfnm{Augusto}\binits{A.}}
(\byear{2009}).
\btitle{{Interlacement percolation on transient weighted graphs}}.
\bjournal{Electron.\ J.\ Probab.}
\bvolume{14}
\bpages{1604--1627}.
\bdoi{10.1214/EJP.v14-670}
\end{barticle}
\endbibitem

\bibitem{TW11}
\begin{barticle}[author]
\bauthor{\bsnm{Teixeira},~\bfnm{Augusto}\binits{A.}} \AND \bauthor{\bsnm{Windisch},~\bfnm{David}\binits{D.}}
(\byear{2011}).
\btitle{On the fragmentation of a torus by random walk}.
\bjournal{Comm.\ Pure Appl.\ Math.}
\bvolume{64}
\bpages{1599--1646}.
\end{barticle}
\endbibitem

\bibitem{Windisch2008}
\begin{barticle}[author]
\bauthor{\bsnm{Windisch},~\bfnm{David}\binits{D.}}
(\byear{2008}).
\btitle{Random walk on a discrete torus and random interlacements}.
\bjournal{Electron.\ Commun.\ Probab.}
\bvolume{13}
\bpages{140-150}.
\bdoi{10.1214/ECP.v13-1359}
\end{barticle}
\endbibitem

\bibitem{CT16}
\begin{barticle}[author]
\bauthor{\bsnm{Čern{\'y}},~\bfnm{Jiř{\'i}}\binits{J.}} \AND \bauthor{\bsnm{Teixeira},~\bfnm{Augusto}\binits{A.}}
(\byear{2016}).
\btitle{{Random walks on torus and random interlacements: Macroscopic coupling and phase transition}}.
\bjournal{Ann.\ Appl.\ Probab.}
\bvolume{26}
\bpages{2883--2914}.
\bdoi{10.1214/15-AAP1165}
\end{barticle}
\endbibitem

\end{thebibliography}


\end{document}